%% file: 0.tex
\documentclass[11pt]{amsart}
\usepackage{graphicx, epstopdf, verbatim, fancyhdr, indent, wrapfig} 
\usepackage[nohug]{diagrams}\diagramstyle[labelstyle=\scriptstyle]
\usepackage[usenames,dvipsnames]{color}
\usepackage[colorlinks=true, urlcolor=MidnightBlue, linkcolor=MidnightBlue, citecolor=MidnightBlue, pdftitle={Pontryagin invariants and integral formulas for Milnor's triple linking number}, pdfauthor={Dennis DeTurck, Herman Gluck, Rafal Komendarczyk, Paul Melvin, Clayton Shonkwiler and David Shea Vela-Vick}, 
pdfsubject={Geometric Topology, Knot Theory}, pdfkeywords={mu invariant, Hopf invariant, link-homotopy, string links, helicity}]{hyperref}

\usepackage[left=1.3in,top=1.3in,right=1.3in,bottom=1.2in,head=.1in]{geometry}

\bfseries
\def\part#1{\bigbreak\noindent{\textbf{\boldmath #1}}\smallbreak} %\textsc

\newtheorem{theorem}{Theorem}

\newtheorem*{theorem*}{Bicycle Theorem}
\newtheorem{proposition}{Proposition}[section]
\newtheorem{lemma}[proposition]{Lemma}
\newtheorem{corollary}[proposition]{Corollary}

\theoremstyle{definition}
\newtheorem{example}[proposition]{Example}

\newtheorem*{definition}{Definition}

\newcommand{\thmref}[1]{Theorem~\ref{thm:#1}}
\newcommand{\propref}[1]{Proposition~\ref{prop:#1}}

\newcommand{\corref}[1]{Corollary~\ref{cor:#1}}
\newcommand{\secref}[1]{Section~\ref{sec:#1}}
\newcommand{\figref}[1]{Figure~\ref{fig:#1}}

\newcommand{\exref}[1]{Example~\ref{ex:#1}}

\def\|{|}

\newcommand{\itb}{\bfseries\itshape\boldmath}

\def\bz{\mathbb Z}

\def\br{\mathbb R}

\def\bh{\mathbb H}
\def\bd{\mathbb D}

\def\bF{\mathcal F}

\def\V{\mathbb V}

\def\bl{\mathcal L}

\def\cdotsclose{\cdot\!\cdot\!\cdot}

\def\x{\textup{\bf x}}
\def\y{\textup{\bf y}}

\def\t{\textup{\bf t}}
\def\u{\textup{\bf u}}
\def\vv{\textup{\bf v}}
\newcommand{\vc}[1]{\overset{\scriptstyle{}_\rightharpoonup}{#1}}
\def\vec{\vc{\textup{\bf v}}}
\def\vecn{\vc{\textup{\bf n}}}
\def\vecz{\vc{\textup{\bf z}}}
\def\w{\textup{\bf w}}
\def\a{\textup{\bf a}}
\def\veca{\vc{\textup{\bf a}}}
\def\b{\textup{\bf b}}
\def\c{\textup{\bf c}}
\def\m{\textup{\bf m}}
\def\n{\textup{\bf n}}
\def\p{\textup{\bf p}}
\def\bah#1{\overline#1}
\newcommand{\ph}{\varphi}
\newcommand{\reals}{\br}
\newcommand{\ints}{\bz}
\newcommand{\zero}{{\bf0}}

\def\Vert{|}
\newcommand{\dast}{*} %{{\displaystyle{\ast}}}
\newcommand{\Hopf}{\mathop{\rm Hopf}\nolimits}
\newcommand{\vol}{\mathop{\rm vol}\nolimits}
\newcommand{\Har}{\mathop{\rm Har}\nolimits}
\newcommand{\Gr}{\mathop{\rm Gr}\nolimits}
\newcommand{\BS}{\mathop{\rm BS}\nolimits}

\def\cald{\mathcal D}

\def\calc{\mathcal C}
\def\call{\mathcal L}

\def\calm{\mathcal M}
\def\caln{\mathcal N}

\def\calp{\mathcal P}

\def\st{\ | \ }

\def\sgn{\textup{sign}}
\def\deg{\textup{deg}}

\def\dt{\,\raisebox{-.6ex}{\huge$\hskip-.02in\cdot$\hskip.005in}}
\def\dte{\,\raisebox{-.3ex}{\Large$\hskip-.02in\cdot$\hskip.005in}}
\def\dth{\raisebox{-.3ex}{\huge$\cdot$}}
\def\dtw{\,\dt\,}
\def\dtew{\,\dte\,}

\def\what#1{\widehat{#1}}

\def\tor{T^3}
\def\torh{T^2}
\def\man{M}

\def\norm#1{\lceil\!\!\!\hskip.01in \lfloor #1 \rceil\!\!\!\hskip.01in\rfloor}

\def\so{\mathrm{SO}}
\def\con#1#2{\mathrm{Conf}_{#1}#2}
\def\conf{\con{3}{S^3}}
\def\confn{\con{n}{\man}}
\def\conftr{\con{2}{\br^3}}
\def\confts{\con{2}{S^3}}
\def\conftrn{\con{2}{\br^n}}
\def\conftsn{\con{2}{S^n}}

\def\gras{\mathrm{G}_2\br^4}
\def\stiefel{\mathrm{V}_2\br^4}

\def\rot{\mathrm{rot}}

\def\zbar{\bar z}

\def\pr{\mathrm{pr}}

\def\re{\mathrm{Re\,}}
\def\Im{\mathrm{Im\,}} % Imaginary part

\def\lk{\mathrm{Lk}}
\def\pt{\mathrm{*}}

\def\im{\mathrm{im\,}} % image

\def\lcm{\textup{lcm}}

\def\pb#1#2{\langle#1,#2\rangle} % pointy bracket

\begin{document}
%%%%%%%%%%%%%%%

%%%%%%%%%%%%%%%%
%% TITLE and ABSTRACT %%
%%%%%%%%%%%%%%%%

%\vskip-.2in 
%\vskip-.2in

\title{Pontryagin invariants and integral formulas \\ for Milnor's triple linking number}

\author{Dennis DeTurck \and Herman Gluck \and Rafal Komendarczyk \and \\ Paul Melvin \and Clayton Shonkwiler \and David Shea Vela-Vick}

\begin{abstract}

To each three-component link in the $3$-sphere, we associate a geometrically natural characteristic map from the $3$-torus to the $2$-sphere, and show that the pairwise linking numbers and Milnor triple linking number that classify the link up to link homotopy correspond to the Pontryagin invariants that classify its characteristic map up to homotopy.   This can be viewed as a natural extension of the familiar fact that the linking number of a two-component link in $3$-space is the degree of its associated Gauss map from the $2$-torus to the $2$-sphere.  

When the pairwise linking numbers are all zero, we give an integral formula for the triple linking number analogous to the Gauss integral for the pairwise linking numbers.  The integrand in this formula is geometrically natural in the sense that it is invariant under orientation-preserving rigid motions of the $3$-sphere, while the integral itself can be viewed as the helicity of a related vector field on the $3$-torus.

\end{abstract}

\maketitle

\parskip 4pt
\addtolength{\baselineskip}{.2pt}

\input{1introduction}

\input{2milnor}

\input{3pontryagin}

\input{4characteristic}

\input{5Aprep}

\input{6A}
\input{7A}

\input{8B}

\input{9bibliography}

\end{document}

%% file: 1introduction.tex
%%%%%%%%%%%%%%%%%
%% SECTION 1: Introduction %%
%%%%%%%%%%%%%%%%%

\vskip-.2in
\vskip-.3in

\section{Introduction}
\label{sec:intro}

\begin{wrapfigure}[10]{l}{35mm}
  \begin{center}
  \includegraphics[height=100pt]{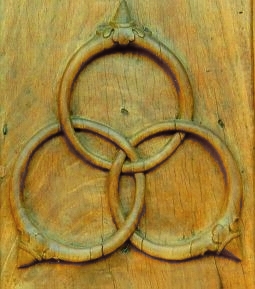}
  \put(-88,-12){\small\bf Borromean Rings}
% \put(-60,-24){\bf Rings}
  \end{center}
\label{fig:borromean}
\end{wrapfigure}
Three-component links in the $3$-sphere $S^3$ were classified up to {\itb link homotopy} -- a deformation during which each component may cross itself but distinct components must remain disjoint -- by John Milnor in his senior thesis, published in 1954.   A complete set of invariants is given by the pairwise linking numbers $p$, $q$ and $r$ of the components, and by the residue class $\mu$ of one further integer modulo the greatest common divisor of $p$, $q$ and $r$, the {\itb triple linking number} of the title.   For example, the Borromean rings shown here have $p=q=r=0$ and $\mu=\pm1$, where the sign depends on the ordering and orientation of the components.

\vskip .1in

To each such link $L$ we will associate a geometrically natural characteristic map $g_L$ from the $3$-torus $\tor = S^1\times S^1\times S^1$ to the $2$-sphere $S^2$ in such a way that link homotopies of $L$ become homotopies of $g_L$.  The definition of $g_L$ will be given below.   The assignment $L\mapsto g_L$ then defines a function 
\begin{comment}
\begin{equation}\label{eqn:g}
g:\call_3\  \longrightarrow \ [T^3,S^2]
\tag{$*$}
\end{equation}
\end{comment}
$$
g:\call_3 \longrightarrow [T^3,S^2]
$$
from the set $\call_3$ of link homotopy classes of three-component links in $S^3$ to the set $[T^3,S^2]$ of homotopy classes of maps $T^3\to S^2$, and it will be seen below that $g$ is injective.  

\begin{comment}  This perspective on link homotopy will be discussed further in the ``Background and Motivation" section at the end of the Introduction.
\end{comment}

Maps from $\tor$ to $S^2$  were classified up to homotopy by Lev Pontryagin in 1941.   A complete set of invariants is given by the degrees $p$, $q$ and $r$ of the restrictions to the $2$-dimensional coordinate subtori, and by the residue class $\nu$ of one further integer modulo {\it twice} the greatest common divisor of $p$, $q$ and $r$, the {\itb Pontryagin invariant} of the map.  

This invariant is an analogue of the Hopf invariant for maps from $S^3$ to $S^2$, and is an absolute version of the relative invariant originally defined by Pontryagin for {\it pairs} of maps from a $3$-complex to $S^2$ that agree on the $2$-skeleton of the domain.  

Our first main result, \thmref{A} below, equates Milnor's and Pontryagin's invariants $p$, $q$ and $r$ for $L$ and $g_L$, and asserts that 
$$
2\mu(L) \ = \ \nu(g_L).
$$
As a consequence, the function $g:\call_3\  \longrightarrow \ [T^3,S^2]$ above is one-to-one, with image the set of maps of even $\nu$-invariant.   

In the special case when $p=q=r=0$, we derive an explicit and geometrically natural integral formula for the triple linking number, reminiscent of Gauss' classical integral formula for the pairwise linking number.  This formula and variations of it are presented in \thmref{B} below.

In the rest of this introduction, we provide the definition of the characteristic map, give careful statements of Theorems \ref{thm:A} and \ref{thm:B}, and then discuss some motivation for our work coming from the homotopy theory of configuration spaces and from fluid dynamics and plasma physics.

%%%
\part{The characteristic map of a three-component link in the $3$-sphere}
%%%

Let $x$, $y$ and $z$ be three distinct points on the unit $3$-sphere $S^3$ in $\br^4$.  They cannot lie on a straight line in $\br^4$,  so must span a $2$-plane there.  Translate this plane to pass through the origin, and then orient it so that the vectors $x-z$ and $y-z$ form a positive basis.  The result is an element $G(x,y,z)$ of the Grassmann manifold $\gras$ of all oriented 2-planes through the origin in $4$-space.  This procedure defines the {\itb Grassmann map} 
$$
G: \conf \ \longrightarrow \ \gras
$$
pictured in \figref{gras}, where 
$$
\conf\ \subset \ S^3 \times S^3 \times S^3
$$ 
is the configuration space of ordered triples of distinct points in $S^3$.  The map $G$ is equivariant with respect to the diagonal $SO(4)$ action on $S^3 \times S^3 \times S^3$ and the usual $SO(4)$ action on $\gras$.

%%% FIGURE 1: Grassmann Map %%%
\begin{figure}[h!]
\includegraphics[height=120pt]{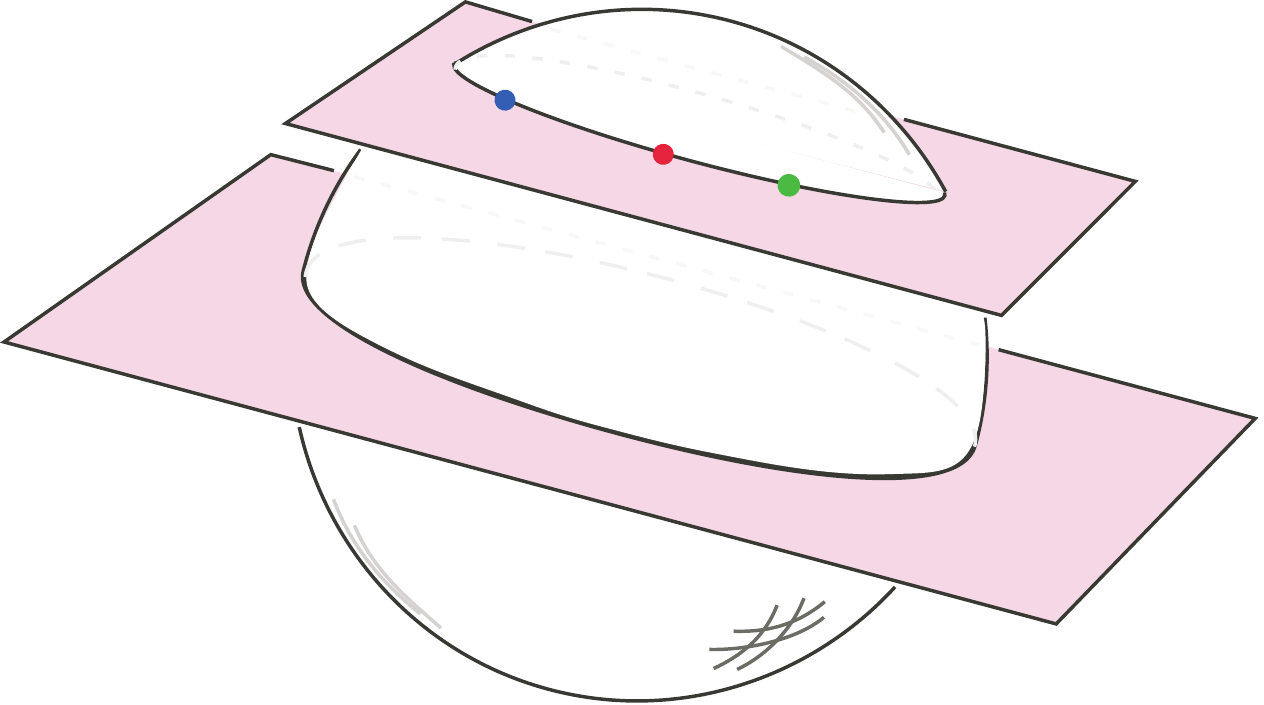}
\put(-70,27){\small $G(x,y,z)$}
\put(-135,95){\small $x$}
\put(-107,86){\small $y$}
\put(-85,80){\small $z$}
\put(-125,15){$S^3$}
\put(-190,23){$\br^4$}
\caption{The Grassmann map $(x,y,z)\mapsto G(x,y,z)$}
\label{fig:gras}
\end{figure}
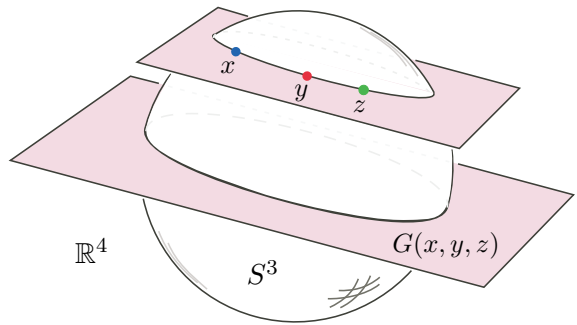
%%%%%%%%%%%%%%%%%%%%

The Grassmann manifold $\gras$ is diffeomorphic to the product $S^2\times S^2$ of two 2-spheres, as explained in \secref{characteristic} below.  Let $\pi_+$ and $ \pi_- : \gras\to S^2$ denote the projections to the two factors.  One of these will be used in the definition of the characteristic map, but the choice of which one will be seen to be immaterial.  

Now let $L$ be a link in $S^3$ with three parametrized components
$$
X \ = \ \{x(s) \st s \in S^1\} \ , \ Y \ = \ \{y(t) \st t \in S^1 \} \ \ \text{and}\ \  Z \ = \ \{z(u) \st u \in S^1 \}
$$
as indicated schematically in \figref{graslink}.  Here and throughout, we view the parametrizing circle $S^1$ as the quotient $\br/2\pi\bz$, and assume implicitly that the parametrizing functions $x=x(s)$, $y=y(t)$ and $z=z(u)$ are smooth with nowhere vanishing derivatives. 

%%% FIGURE 2: Link %%%
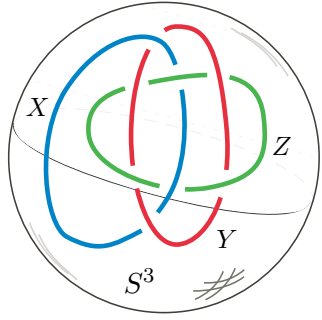
\begin{figure}[h!]
\includegraphics[height=120pt]{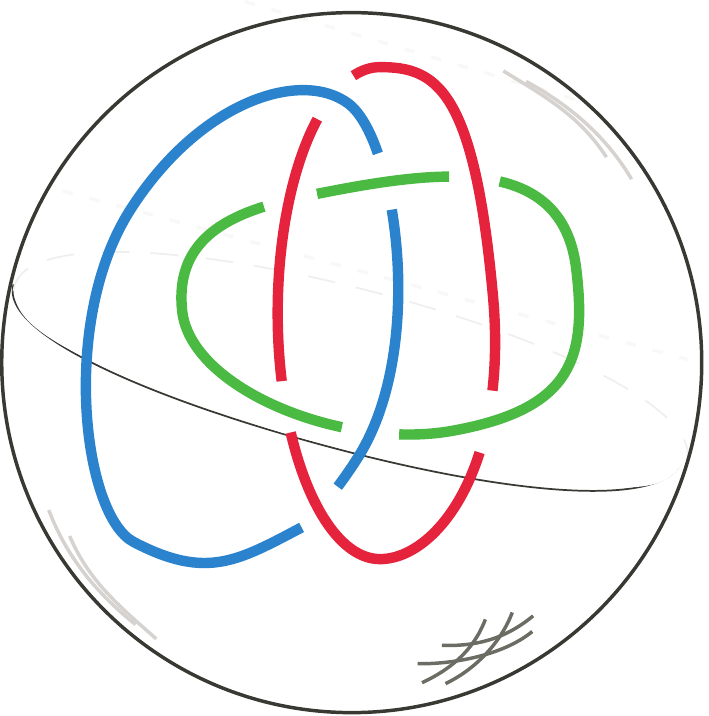}
\put(-112,75){\small $X$}
\put(-40,25){\small $Y$}
\put(-19,60){\small $Z$}
\put(-75,8){$S^3$}
\caption{The link $L$ in $S^3$}
\label{fig:graslink}
\end{figure}
%%%%%%%%%%%%%%%%%%%%

We define the {\itb characteristic map} $g_L:T^3 \to S^2$ in terms of the Grassmann map $G:\conf\to\gras$ and the projection $\pi_+:\gras\to S^2$ by the formula
$$
g_L(s,t,u) \ = \ \pi_+(G(x(s),y(t), z(u)).
$$
In other words $g_L$ is the composition $\pi_+Ge_L$ where $e_L:\tor\to\gras$ is the embedding that parametrizes the link, given by $e_L(s,t,u) \ = \ (x(s),y(t),z(u))$.   We regard $g_L$ as a natural generalization of the Gauss map $T^2\to S^2$ associated with a two-component link in $\br^3$. % given by $g(s,t) = (x(s)-y(t))/|x(s)-y(t)|$.  

It is clear that the homotopy class of $g_L$ is unchanged under reparametrization of $L$, or more generally under any link homotopy of $L$.  Furthermore, $g_L$  is ``symmetric" in that  it transforms under any permutation of the components of $L$ by precomposing with the corresponding permutation automorphism of $\tor$ multiplied by the sign of the permutation.

\begin{comment}
In \secref{characteristic} we give an explicit formula for $g_L$ as the unit normalization of a vector field on $\tor$ whose components are quadratic polynomials in the components of $x(s)$, $y(t)$ and $z(u)$.   Then in \secref{asymmetric} we set up for the proof of \thmref{A} by describing an alternative asymmetric form $h_L$ of the characteristic map, homotopic to $g_L$, and introducing a special class of ``generic" links whose associated Pontryagin invariants $\nu(h_L)$ are easily computed.
\end{comment}

%%%%%%%%%%%%%%%%%%%%%%%%%%%%%%%%%%
%%%%%%%%%%%%%%%%%%%%%%%%%%%%%%%%%%
\part{Statement of results}
%%%%%%%%%%%%%%%%%%%%%%%%%%%%%%%%%%
%%%%%%%%%%%%%%%%%%%%%%%%%%%%%%%%%%

The first of our two main results gives an explicit correspondence between the Milnor link homotopy invariants of a three-component ordered, oriented link in the $3$-sphere and the Pontryagin homotopy invariants of its characteristic map. 

%%%%%%%%%%%%%%%%%
\begin{theorem}\label{thm:A}
\textbf{\boldmath Let $L$ be a $3$-component link in $S^3$.  Then the pairwise linking numbers $p$, $q$ and $r$ of $L$ are equal to the degrees of its characteristic map $g_L:T^3\to S^2$ on the two-dimensional coordinate subtori of $\tor$, while twice Milnor's $\mu$-invariant for $L$ is equal to \mbox{Pontryagin's} $\nu$-invariant for $g_L$ modulo $2\gcd(p,q,r)$.}
\end{theorem}
%%%%%%%%%%%%%%%%%

\noindent {\bf Conventions.}  (1) In this paper, links in $S^3$ are always {\it ordered} (the components are taken in a specific order) and {\it oriented} (each component is oriented).  

\noindent (2) The two-dimensional coordinate subtori of $T^3$ are oriented to have positive intersection with the remaining circle factors.

\vskip.05in

The key idea of the proof centers on the ``delta move", a higher order variant of a crossing change.  Applied to a three-component link $L$ in the $3$-sphere, we show easily that this move increases Milnor's $\mu$-invariant by $1$, and then the bulk of the proof is devoted to showing that it increases Pontryagin's $\nu$-invariant by $2$.

%\smallskip

The background material for the proof is contained in Sections \ref{sec:milnor} and \ref{sec:pontryagin} of the paper.  In particular, in \secref{milnor} we discuss how to compute Milnor's $\mu$-invariant, describe the delta move, and prove that it increases the $\mu$-invariant by $1$.  In \secref{pontryagin}, we discuss Pontryagin's homotopy classification of maps from a 3-manifold  to the 2-sphere in terms of framed bordism of framed links, define the ``Pontryagin link'' of such a map to be the inverse image of any regular value, and then show how to convert the relative $\nu$-invariant to an absolute one when the manifold is the $3$-torus.  This section concludes with a simple procedure for computing the absolute $\nu$-invariant of any map of the 3-torus to the 2-sphere from a diagram of its Pontryagin link. 

The proof of \thmref{A} will occupy us in Sections \ref{sec:characteristic}--\ref{sec:A}, and is organized as follows.  

In \secref{characteristic}, we derive an explicit formula for the characteristic map $g_L$ needed for the proof of \thmref{B}, but not for that of \thmref{A}.  We then describe an alternative asymmetric form $h_L$ of the characteristic map that is more convenient for the proof of \thmref{A}, and that, just as for $g_L$, depends {\it a priori} on the choice of a projection to the $2$-sphere.  We show that the two characteristic maps are homotopic to one another, and that, up to homotopy, neither in fact depends on the choice of projection.   We note a close relation between $h_L$ and the classical Gauss maps of two-component links in $3$-space, and use this, together with the symmetry of $g_L$, to prove the first statement in \thmref{A}.  

In \secref{Aprep}, we set up for the proof of the rest of \thmref{A} by describing the standard open-book structure on $S^3$ with disk pages, use a link homotopy to move a given link $L$ into ``generic position'' with respect to it, and then show how to explicitly visualize the Pontryagin link of the corresponding asymmetric characteristic map $h_L$.  Finally, we use the results of \secref{pontryagin} to develop a method for computing $\nu(h_L) = \nu(g_L)$. 

In \secref{A}, we complete the proof of \thmref{A} by induction, using the methods developed in Sections \ref{sec:milnor} and \ref{sec:pontryagin} to first confirm the ``base case'', and then using the methods of \secref{Aprep} to carry out the inductive step, showing that the delta move increases Pontryagin's $\nu$-invariant by 2.

\secref{A'} contains a sketch of a more algebraic proof of this theorem using the group of link homotopy classes of three-component string links and the fundamental groups of spaces of maps of the 2-torus to the $2$-sphere. More details for this alternative proof can be found in our announcement \cite{DGKMSV}, which is otherwise an abridged version of some of the material in the current paper.

%\smallskip

In Sections~\ref{sec:B} and~\ref{sec:B'}, we investigate the special case when the pairwise linking numbers $p$, $q$ and $r$ of $L$ are all zero, and so the $\mu$ and $\nu$-invariants are ordinary integers.  In this case we will use J.\,H.\,C.~Whitehead's integral formula for the Hopf invariant, adapted to maps of the $3\text{-torus}$ to the $2$-sphere, together with a formula for the fundamental solution of the scalar Laplacian on the $3$-torus as a Fourier series in three variables, to provide an explicit integral formula for $\nu(g_L)$, and hence for $\mu(L)$ in light of \thmref{A}.  See \thmref{B}, stated below.
 
This formula will be presented in three versions: first as an integral involving differential forms on the {$3\text{-torus}$, second as the same integral expressed in terms of vector fields, and finally as an infinite sum involving Fourier coefficients. 

To state these formulas, we need some definitions.   Let $\omega$ denote the Euclidean area $2$-form on $S^2$, normalized to have total area $1$.  Then $\omega$ pulls back under the characteristic map $g_L$ to a closed $2$-form $\omega_L$ on $\tor$, which can be converted to a divergence-free vector field $\vec_L$ on $\tor$ via the usual formula 
$$
\omega_L(\vc{\textup{\bf a}}\, ,\vc{\textup{\bf b}}) \ = \ (\vc{\textup{\bf a}}\times\vc{\textup{\bf b}})\dtw\vec_L.
$$  
We call  $\omega_L$  the {\itb characteristic $2$-form} of  $L$,  and  $\vec_L$  its {\itb characteristic vector field}.  When $p$, $q$ and $r$ are all zero, $\omega_L$ is exact and $\vec_L$ is in the image of curl.  In \secref{B} we give explicit formulas for  $\omega_L$  and  $\vec_L$,  and in \secref{B'}  for the fundamental solution $\varphi$ of the scalar Laplacian on $\tor$, which appear in the integral formulas for $\mu(L)$.

For the third version of our formula for the Milnor invariant, we need to express the characteristic $2$-form and vector field in terms of Fourier series on the $3$-torus.  To that end, view $T^3$ as the quotient $(\br/2\pi\bz)^3$ and write $\x=(s,t,u)\in\br^3$ for a general point there.  Using the complex form of Fourier series, express
$$
\omega_L \ = \ \sum_{\n\in\bz^3}\left(c^s_{\n} \, dt\wedge du+c^t_{\n} \, du\wedge ds+c^u_{\n} \, ds\wedge dt\right)e^{i\n\dte\x}.
$$
Using vector notation \,$\c_{\n} = (c^s_{\n},c^t_{\n},c^u_{\n})$\ , \ $d\x \ = \ (ds, dt, du)$\ , \ $\partial_\x = (\partial/\partial s,\partial/\partial t, \partial/\partial u)$ \,and\, $\star d\x \ = \ (dt\wedge du, du\wedge ds, ds\wedge dt)$\,, the formulas for $\omega_L$ and $\vec_L$ become
$$
\omega_L\ =\ \sum_{\n\,\in\,\bz^3}\c_{\n}\, e^{i\n\dte\x}\dtw\star\!d\x \qquad\text{and }\qquad \vec_L \ = \ \sum_{\n\,\in\,\bz^3}\c_{\n}\, e^{i\n\dte\x}\dtw \partial_\x.
$$
As a consequence of the assumption that the pairwise linking numbers of $L$ vanish, the coefficient $\c_{\zero}$ is the zero vector, which will imply that $\omega_L$ is exact, or equivalently that $\vec_L$ is in the image of curl.

%%%%%%%%%%%%%%%%%
\begin{theorem}\label{thm:B}
\textbf{\boldmath If the pairwise linking numbers of a three-component link $L$ in $S^3$ are all zero, then Milnor's $\mu$-invariant of $L$ is given by each of the following equivalent formulas
$$
\begin{aligned}
\mu(L) \ &= \  1/2\, \int_{T^3} \delta(\varphi*\omega_L) \wedge \omega_L \ & \qquad (1) \\
\ &= \ 1/2\, \int_{\tor\times\tor} \vec_L(\x) \times \vec_L(\y) \dtew \nabla_{\!\y}\, \varphi(\x - \y)\ d\x\, d\y \ & \qquad (2) \\
\ & = \  8\pi^3 \sum_{\n\ne\zero} \textstyle \a_\n\times \b_\n\dtew \n/\|\n\|^2\ . & \qquad (3)
\end{aligned}
$$
where $\varphi$ is the fundamental solution of the scalar Laplacian on the $3$-torus, $\omega_L$ and $\vec_L$ are the characteristic form and vector field of $L$, and $\a_\n$ and $\b_\n$ are the real and imaginary parts of the Fourier coefficients $\c_\n$ of $\omega_L$ and $\vec_L$.}
\end{theorem}
%%%%%%%%%%%%%%%%%

\noindent {\bf Explanation of the notation.} In (1), $\delta$ is the exterior co-derivative and $\ph*\omega_L$ is the convolution of $\ph$ with $\omega_L$, discussed in detail in Sections~\ref{sec:B} and~\ref{sec:B'}. 

In (2), the difference $\x-\y$ is taken in the abelian group structure on $\tor$, the expression $\nabla_{\!\y}\, \varphi(\x - \y)$ indicates the gradient with respect to $\y$ while $\x$ is held fixed, and $d\x$ and $d\y$ are volume elements on $\tor$.  This formula is just the vector field version of (1) in which the integral hidden in the convolution is expressed openly; we will see below that it represents the ``helicity" of the vector field $\vec_L$ on $\tor$.  

Observe that the integrands in (1) and (2) are invariant under the group $SO(4)$ of orientation-preserving rigid motions of $S^3$, attesting to the naturality of the formulas.

%%%%%%%%%%%%%%%%%%%%%%%%%%%%%%%%%%
%%%%%%%%%%%%%%%%%%%%%%%%%%%%%%%%%%
\part{Background and Motivation} 
%%%%%%%%%%%%%%%%%%%%%%%%%%%%%%%%%%
%%%%%%%%%%%%%%%%%%%%%%%%%%%%%%%%%%

{\bf Configuration spaces.} \ To study the linking of simple closed curves in a $3$-manifold $\man$ from the perspective of homotopy theory, it is convenient to ignore the knotting of individual components and focus on the relation of {\itb link homotopy}.  For some background on this notion, see Milnor \cite{Milnor54, Milnor57} and, for example, Massey~\cite{Massey68}, Casson~\cite{Casson75}, Turaev~\cite{Turaev76}, Porter~\cite{Porter}, Fenn~\cite{Fenn83}, Orr~\cite{Orr}, Cochran~\cite{Cochran}, and Habegger and Lin~\cite{HabeggerLin}.

Configuration spaces come into the picture as follows.  Let $L$ be an ordered, oriented link in $\man$ with $n$ components $X,Y,\dots$ parametrized by $x=x(s),y=y(t), \dots$ for $s,t,\dots$ in $S^1$.  Then consider the {\itb evaluation map} 
$$
e_L : T^n \ \longrightarrow \ \confn\ , \qquad (s, t, \dots)\longmapsto (x, y, \dots)
$$
from the $n$-torus $T^n = S^1\times\cdots\times S^1$ to the configuration space $\confn$ of ordered $n\text{-tuples}$ of distinct points in $\man$.  Since link homotopies of $L$ become homotopies of $e_L$, the assignment $L\mapsto e_L$ induces a map from the set $\call_n(\man)$ of link homotopy classes of $n$-component links in $\man$ to the set $[T^n,\confn]$ of homotopy classes of maps from $T^n$ to $\confn$,
$$
e : \call_n(\man) \ \longrightarrow \ [T^n,\confn].
$$
We can think of the map $e$ as defining a representation from the world of link homotopy to the world of homotopy, and the basic question  is whether or not this representation is faithful, that is, one-to-one.

For two-component links in $\br^3$, the representation $e$ is faithful.  The set $\call_2(\br^3)$ is in one-to-one correspondence with the integers via the classical linking number, while the set $[T^2,\conftr]$ is also in one-to-one correspondence with the integers via the Brouwer degree, since $\conftr$ deformation retracts to $S^2$.  The correspondence $e$ is bijective since the linking number of a two-component link equals the degree of its ``Gauss map"
$$
T^2 \ \longrightarrow \ S^2\quad , \quad (s,t)\,\longmapsto\, \frac{y-x}{\|y-x\|}.
$$

\break
The profit is the famous integral formula of Gauss~\cite{Gauss},
$$
\begin{aligned}
\lk(X,Y) \ &= \ \frac1{4\pi} \,\int_{T^2} \frac{dx}{ds}  \times \frac{dy}{dt} \dtw  \frac{x-y}{\,|x-y|^3} \ \,ds \, dt \\
\ &= \ \int_{T^2} \frac{dx}{ds}  \times   \frac{dy}{dt}  \dtw  \nabla_{\!y}\,\varphi(y\!-\!x) \ ds \, dt,
\end{aligned}
$$
where $\varphi(x) = 1/4\pi\|x\|$ is the fundamental solution of the scalar Laplacian in $\br^3$.  The integrand is natural, in the sense that it is invariant under the group of orientation-preserving rigid motions of $\br^3$, acting on the link.  

By contrast, for two-component links in $S^3$ the representation $e$ is not faithful.  The configuration space $\confts$ has the homotopy type of $S^3$, and hence all maps of $T^2$ to $\confts$ are homotopically trivial.  In particular, the analogue of Gauss's linking integral in $S^3$, with an integrand which is geometrically natural in the above sense, cannot be obtained by this route.   Nevertheless, such an integral formula was found by DeTurck and Gluck \cite{DeTurckGluck} using an alternative route via electrodynamics on the $3$-sphere, and independently by Kuperberg~\cite{Kuperberg} via the calculus of double forms.

For higher-dimensional two-component links, the same dichotomy holds. 
In $\br^n$, Scott \cite{Scott} and Massey and Rolfsen \cite{MasseyRolfsen} showed that link homotopy classes of links whose components are copies of $S^k$ and $S^{n-k-1}$ are in bijective correspondence with homotopy classes of maps from $S^k\times S^{n-k-1}$ to the configuration space $\conftrn\simeq S^{n-1}$, and consequently with $\pi_{n-1}(S^{n-1}) \cong\bz$.  Finding a geometric linking integral is straightforward, as the Gauss linking integral and its proof easily generalize to this setting.  Shifting the scene to $S^n$ does not change the link homotopy story, but all maps from $S^k\times S^{n-k-1}$ to the configuration space $\conftsn\simeq S^n$ are homotopically trivial.  Geometrically natural linking integrals still exist in this situation, as demonstrated by DeTurck and Gluck \cite{DeTurckGluckb} and Shonkwiler and Vela-Vick \cite{SVV}, but finding them requires new techniques.

For $n$-component {\itb homotopy Brunnian} links in $\br^3$ -- meaning links that become trivial up to link homotopy when any single component is removed -- and analogous links in higher dimensions, Koschorke \cite{Koschorke} showed that the representation $e$ is again faithful.  This provided the first proof (up to sign) of our \thmref{A} for the case when the pairwise linking numbers are zero.

The content of the present paper is that, for arbitrary three-component links $L$ in $S^3$, the representation $e$ is faithful, and that we are led thereby to a natural integral for Milnor's triple linking number when the pairwise linking numbers vanish.  The relevant configuration space $\conf$ is easily seen to deformation retract to  $S^3 \times S^2$, where the $S^3$ coordinate records one of the three points in each triple (see \secref{characteristic}).  It follows that the evaluation map $e_L:\tor\to\conf$ is homotopic to a map of $\tor$ into an $S^2$ fiber, and this turns out to be, up to homotopy, our characteristic map $g_L$.

\thmref{A} asserts, among other things, that $L$ and $L'$ are link homotopic if and only if $g_L$ and $g_{L'}$ are homotopic, and hence that
$$
e:\call_3(S^3)\longrightarrow[T^3,\conf]
$$
is faithful.  Furthermore, it was observed above that if $h\in\so(4)$ is an orientation preserving isometry of $S^3$, then $e_{h(L)} = h\cdot e_L$ where $\cdot$ is the diagonal action, and so the integrands in the formulas for Milnor's triple linking number $\mu(L)$ in \thmref{B} are invariant under the action of $\so(4)$.

We emphasize that our Theorems \ref{thm:A} and \ref{thm:B} are set specifically in $S^3$, and that so far we have been unable to find corresponding results in Euclidean space $\br^3$ which are equivariant (for \thmref{A}) and invariant (for \thmref{B}) under the noncompact group of orientation-preserving rigid motions of $\br^3$.

%%%%%%%%%%%%%%%%%%%%%%
{\bf Fluid mechanics and plasma physics.} \ The {\itb helicity} of a vector field $V$ defined on a bounded domain $\Omega$ in $\br^3$ is given by the formula
$$
\begin{aligned}
\text{Hel}(V) \ &= \ \frac{1}{4\pi}\, \int_{\Omega\times\Omega} \!\!\!\!\!V(x) \times V(y) \dtw \frac{x-y}{\,|x-y|^3} \  \,dx \, dy \\
 \ &= \ \int_{\Omega\times\Omega} \!\!\!\!\!V(x) \times V(y) \dtw \nabla_{\!y}\,\varphi(x\!-\!y) \ dx \, dy.
\end{aligned}
$$
where, as above, $\varphi$ is the fundamental solution of the scalar Laplacian in $\br^3$, and $dx$ and $dy$ are volume elements.

Woltjer~\cite{Woltjer} introduced this notion during his study of the magnetic field in the Crab Nebula, and showed that the helicity of a magnetic field remains constant as the field evolves according to the equations of ideal magnetohydrodynamics, and that it provides a lower bound for the field energy during such evolution.  The term ``helicity'' was coined by Moffatt~\cite{Moffatt}, who also derived the above formula from Woltjer's original expression.  

There is no mistaking the analogy with Gauss's linking integral, and no surprise that helicity is a measure of the extent to which the orbits of $V$ wrap and coil around one another.  Since its introduction, helicity has played an important role in astrophysics and solar physics, and in plasma physics here on earth.  

Looking back at \thmref{B}, we see that the integral in our formula for Milnor's $\mu\textup{-invariant}$ of a three-component link $L$ in the 3-sphere expresses the helicity of the associated vector field $\vec_L$ on the 3-torus.

Our study was motivated by a problem proposed by Arnol$'$d and Khesin~\cite{ArnoldKhesin98} regarding the search for ``higher helicities'' for divergence-free vector fields.  In their own words:

\vspace{.05in}
\begin{indentation}{2.6em}{2.6em}
\small\itb
\noindent The dream is to define such a hierarchy of invariants for generic vector fields such that, whereas all the invariants of order $\leq k$ have zero value for a given field and there exists a nonzero invariant of order $k+1$, this nonzero invariant provides a lower bound for the field energy.
\end{indentation}
\vspace{.05in}

Since the helicity integral above is analogous to the Gauss linking integral, the general hope is that higher helicities will be analogous to higher order linking invariants.  But the anti-symmetry of Milnor's triple linking number appears to be an impediment to naively generalizing it to a second order helicity for vector fields, and suggests that attention be directed to yet higher order link homotopy and concordance invariants.  

The formulation in Theorems~\ref{thm:A} and~\ref{thm:B} has led to partial results that address the case of vector fields on invariant domains such as flux tubes modeled on the Borromean rings; see Komendarczyk~\cite{Komendarczyk, Komendarczyk10}.

%%%%%%%%%%%%%%%%%%%%%%
\break

{\bf Other integral formulas.} Previous integral formulas for Milnor's triple linking number and attempts to define a higher order helicity can be found in the work of Massey~\cite{Massey58, Massey68}, Monastyrsky and Retakh~\cite{MonastyrskyRetakh}, Berger~\cite{Berger90, Berger91}, Guadagnini, Martellini and Mintchev~\cite{GMM}, Evans and Berger~\cite{EvansBerger92}, Akhmetiev and Ruzmaiken~\cite{Akhmetiev94, Akhmetiev95}, Arnol$'$d and Khesin~\cite{ArnoldKhesin98}, Laurence and Stredulinsky \cite{Laurence}, Leal~\cite{Leal02}, Hornig and Mayer \cite{Hornig}, Rivi\`ere~\cite{Riviere}, Khesin~\cite{Khesin}, Bodecker and Hornig~\cite{Bodecker04}, Auckly and Kapitanski~\cite{AucklyKapitanski}, Akhmetiev~\cite{Akhmetiev05}, and Leal and Pineda~\cite{LealPineda08}.

The principal sources for these formulas are Massey triple products in cohomology, quantum field theory in general, and Chern--Simons theory in particular.  A common feature of these integral formulas is that choices must be made to fix the domain of integration and the value of the integrand.

\smallskip
%%%%%%%%%%%%%%%%%%%%%%%%%
%% Acknowledgements
%%%%%%%%%%%%%%%%%%%%%%%%%

\noindent{\bf Acknowledgements.}  We are grateful to Fred Cohen and Jim Stasheff for their substantial input and help during the preparation of this paper, and to Toshitake Kohno, whose 2002 paper provided one of the inspirations for this work.  Komendarczyk and Shonkwiler also acknowledge support from DARPA grant \#FA9550-08-1-0386, and Vela-Vick from an NSF postdoctoral fellowship.

The Borromean rings shown on the first page are from a panel in the carved walnut doors of the Church of San Sigismondo in Cremona, Italy.  The photograph is courtesy of Peter Cromwell.

%%%%%%%% THIS IS THE END OF THE INTRODUCTION %%%%%%%%

%% file: 2milnor.tex
%%%%%%%%%%%%%%%%%%%%%%%%%
%% SECTION 2: The Milnor invariant
%%%%%%%%%%%%%%%%%%%%%%%%%

\vskip-.15in
\vskip-.15in

\section{The Milnor $\mu$-invariant} 
\label{sec:milnor}

Let $L$ be a three-component link in either $\br^3$ or $S^3$, with oriented components $X$, $Y$ and $Z$, and pairwise linking numbers $p$, $q$ and $r$.  Milnor's original definition of the triple linking number $\mu(L)$, typically denoted $\bar\mu_{123}(L)$, was algebraic.  The formulation in his PhD thesis \cite{Milnor57} is expressed in terms of the lower central series of the link group $G$, the fundamental group of the complement of $L$, as follows.

Choose based meridians for the link components $X$, $Y$ and $Z$, and let $x$, $y$ and $z$ denote the corresponding elements of $G$.  In general, these three elements do not generate $G$, but they do generate the quotient of $G$ by the third term $[[G,G],G]$ in its lower central series.  In this quotient group, the longitude of $Z$ can be written as a word $w$ in $x$, $y$ and $z$ and their inverses.  Assign an integer $m_{xy}(w)$ to this word that counts with signs the number of times that $x$ appears before $y$ in $w$, allowing intervening letters.  More precisely, each appearance of $x\cdotsclose y$ or $x^{-1}\cdotsclose y^{-1}$ contributes $+1$ to $m_{xy}(w)$, while $x\cdotsclose y^{-1}$ or $x^{-1}\cdotsclose y$ contributes $-1$.  Then $\mu(L)$ is the element of $\bz_{\gcd(p,q,r)}$ defined by
$$
\mu(L) \ = \ m_{xy}(w)\ \ \textup{mod}\gcd(p,q,r).
$$

There is a geometric reformulation of this definition, found by Mellor and Melvin \cite{MellorMelvin}, that is more convenient for our purposes:  choose Seifert surfaces $F_X$, $F_Y$ and $F_Z$ for the components of $L$, and move these into general position.   Starting at any point on $X$, record its intersection with the Seifert surfaces for $Y$ and $Z$ by a word $w_X$ in $y$ and $z$.  For example a $y$ or $y^{-1}$ in $w_X$ indicates a positive or negative intersection point of $X$ with $F_Y$.  Set $m_X = m_{yz}(w_X)$, where the right hand side is computed as in the last paragraph, and similarly set $m_Y = m_{zx}(w_Y)$ and $m_Z = m_{xy}(w_Z)$.  Finally, let $t$ be the signed count of the number of triple points of intersection of the three Seifert surfaces.  Then 
$$
\mu(L) \ =  \ m_X+m_Y+m_Z - t \ \ \textup{mod}\gcd(p,q,r).
$$
It follows from this formula that $\mu(L)$ is invariant under even permutations of the components of $L$, but changes sign under odd permutations.  This is a well known property of Milnor's triple linking number.

We give two sample calculations of the triple linking number, using this geometric formulation, that will feature in our inductive proof of \thmref{A} in \secref{A}.

\begin{example} \label{ex:Lpqr}
Let $L_{pqr}$ be the link shown in \figref{Lpqr}, with components $X$, $Y$ and $Z$ and with pairwise linking numbers $p$, $q$ and $r$.   Choose the Seifert surfaces $F_X$, $F_Y$ and $F_Z$ to be the obvious disks, essentially lying in the page, bounded by $X$, $Y$ and $Z$.  

%%% FIGURE 3: Lpqr %%%
\begin{figure}[h!]
\includegraphics[height=150pt]{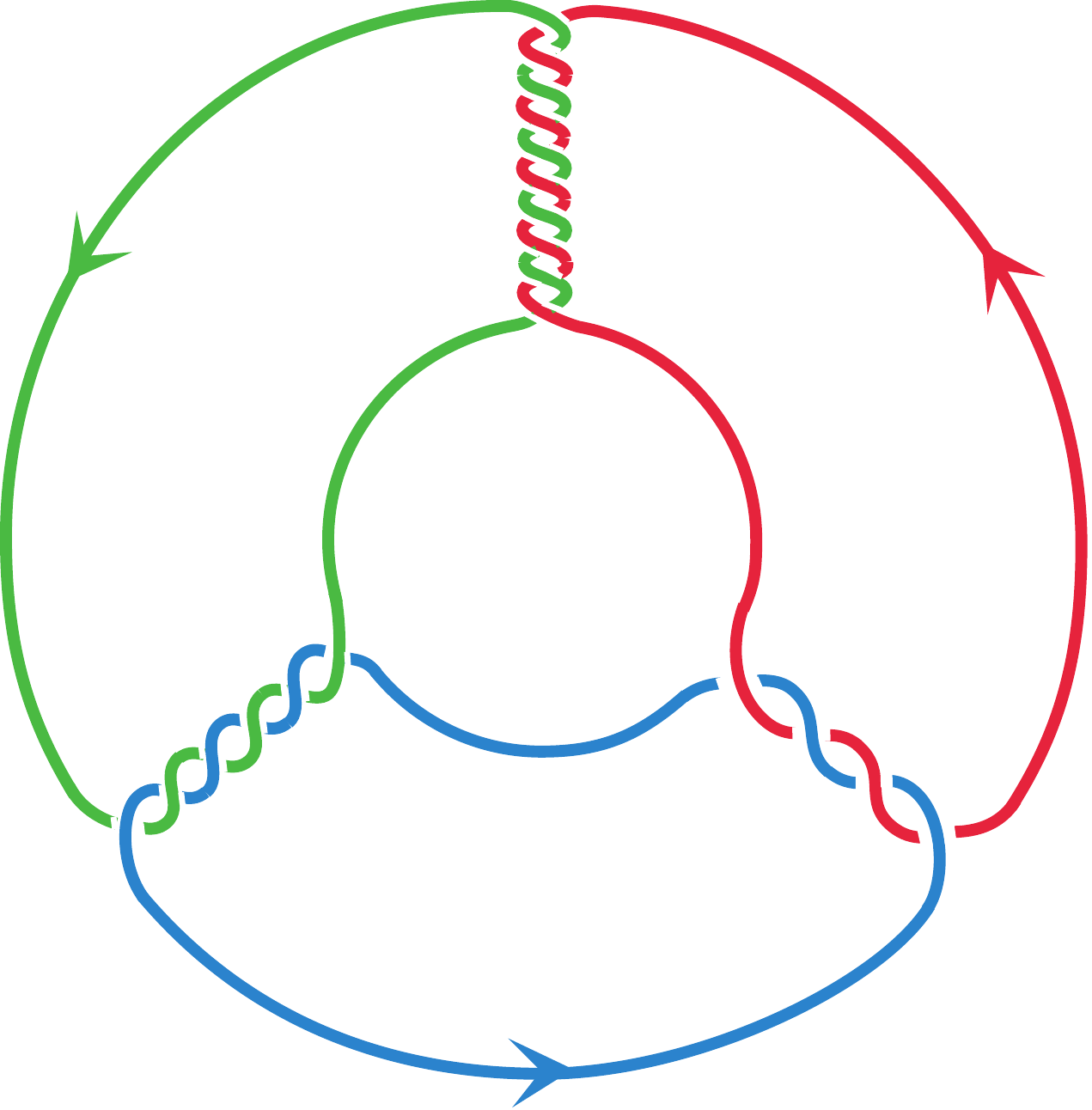}
\put(-80,25){$X$}
\put(-30,90){$Y$}
\put(-125,90){$Z$}
\put(-58,62){\small$r$}
\put(-96,64){\small$q$}
\put(-78,95){\small$p$}
\caption{The base link $L_{pqr}$ for  $p=5$, $q=3$, $r=-2$}
\label{fig:Lpqr}
\end{figure}
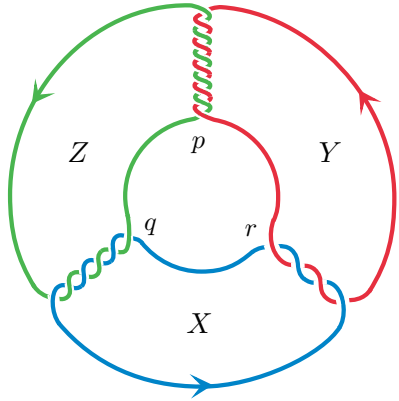
%%%%%%%%%%%%%%%%%%

Starting at appropriate points on the link components, we can read off the words 
$$
w_X = y^rz^q \ , \ \ w_Y = z^px^r \quad\textup{and}\quad w_Z = x^qy^p.
$$
Thus $m_X = qr$, $m_Y = rp$ and $m_Z = pq$.   Furthermore, it is clear that there are no triple points of intersection of the three disks.  Therefore
$$
\mu(L_{pqr}) \ = \ qr+rp+pq-0 \ = \ 0 \ \in \ \bz_{\gcd(p,q,r)}.
$$  
The links $L_{pqr}$ will serve as the {\itb base links} for our proof of \thmref{A}.
\end{example} 

\begin{example}\label{ex:delta} 
Consider the {\itb delta move} shown in \figref{delta}, transforming the link $L$ on the left to the link $\what L$ on the right.  

%%% FIGURE 4: Delta Move %%%
\begin{figure}[h!]
\includegraphics[height=100pt]{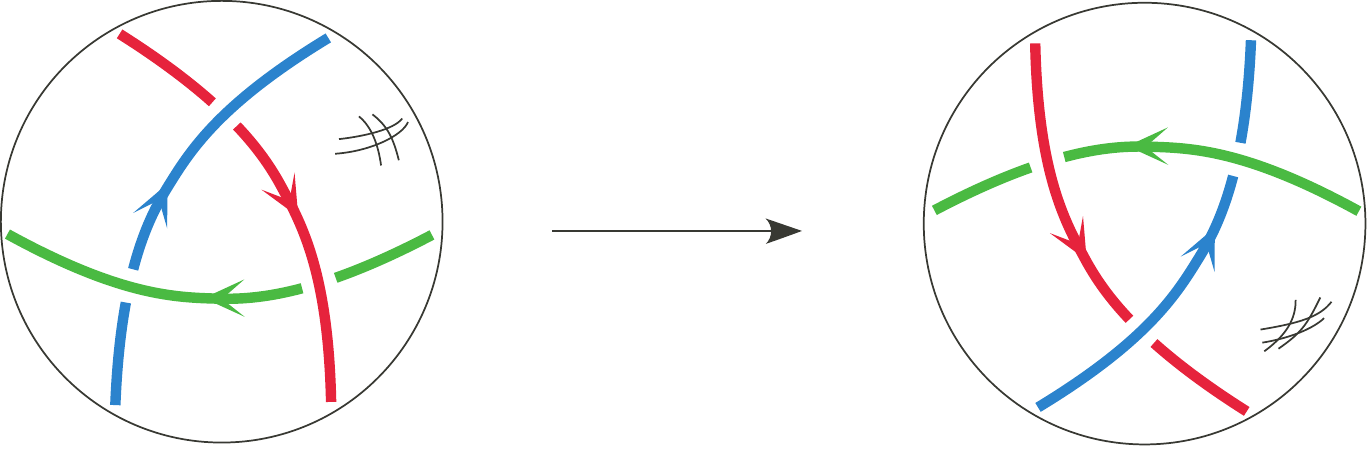}
\put(-228,95){$X$}
\put(-28,99){$\what X$}
\put(-230,-5){$Y$}
\put(-22,-5){$\what Y$}
\put(-315,45){$Z$}
\put(-113,46){$\what Z$}
%\put(-206,72){$\small B$}
%\put(-5,22){$\small B$}
\caption{The delta move $L\to \what L$}
\label{fig:delta}
\end{figure}
%%%%%%%%%%%%%%%%%% 

This move takes place within a $3$-ball, outside of which the link is left fixed.  It does not alter the pairwise linking numbers of $L$, and may be thought of as a higher order variant of a crossing change.

The delta move was introduced by Matveev \cite{Matveev}.  It was shown by Murakami and Nakanishi \cite{MurakamiNakanishi} that a suitable sequence of such moves can transform any link into any other link with the same number of components and the same pairwise linking numbers.  In particular, the base link $L_{pqr}$ above can be transformed into any other three-component link with pairwise linking numbers $p$, $q$ and $r$ by a sequence of delta moves.  The inductive step of our proof of \thmref{A} will be based on this observation, and so we determine right now the effect of the delta move on the $\mu$-invariant.   

If the three arcs involved in the delta move {\it do not} come from three distinct components of the link $L$, then the change can be achieved by a link homotopy, and hence neither Milnor's $\mu$-invariant nor Pontryagin's $\nu$-invariant for $g_L$ will change.  

But if the arcs {\it do} come from three distinct components of $L$ as shown in the figure, then $\mu$ increases by $1$, that is 
$
\mu(\what L) = \mu(L) + 1.
$
This can be seen as follows, using the geometric formula for the $\mu$-invariant.

In \figref{seifert}, we display on the left fragments of the Seifert surfaces $F_X$, $F_Y$ and $F_Z$ before the delta move, while on the right, after the delta move, each old surface is enlarged a bit to provide new surfaces $F_{\what X}$, $F_{\what Y}$ and $F_{\what Z}$. 
 
%%% FIGURE 5: Seifert Surfaces %%%
\begin{figure}[h!]
\includegraphics[height=120pt]{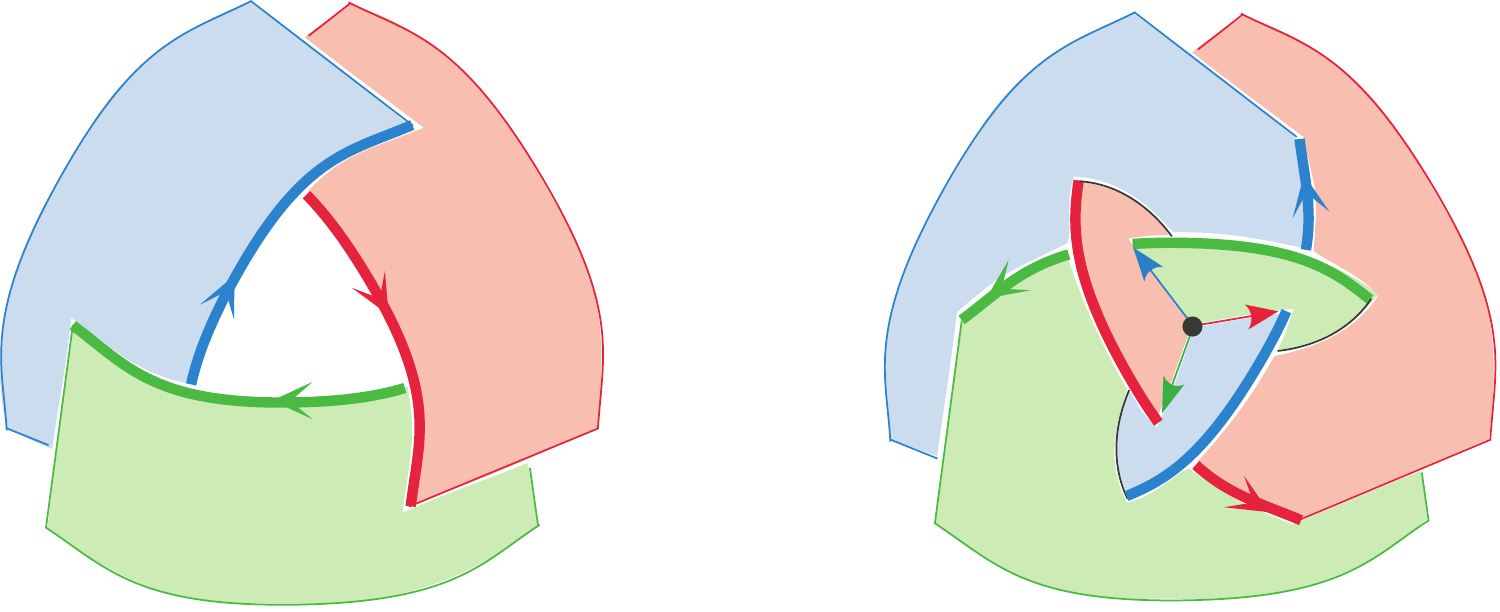}
\put(-300,90){$F_X$}
\put(-185,90){$F_Y$}
\put(-275,-5){$F_Z$}
\put(-125,90){$F_{\what  X}$}
\put(-10,90){$F_{\what  Y}$}
\put(-30,-3){$F_{\what  Z}$}
\caption{A negative triple point appears}
\label{fig:seifert}
\end{figure}
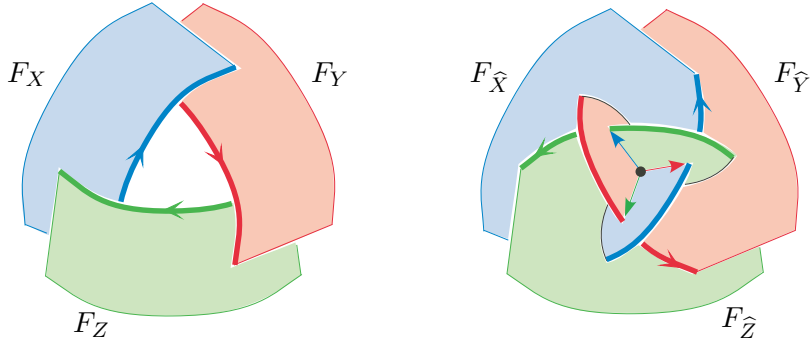
%%%%%%%%%%%%%%%%%%%%

On the left, the three surface fragments are disjoint, while on the right, after their enlargement, they are not.  Where these surfaces now come together, we have an additional isolated triple point, and since the normals to the surfaces at this point form a left-handed frame, this triple point gets a minus sign.  Thus $t$ drops by $1$.   We also see that after the delta move, the curve $\what Z$ has two extra intersections with $F_{\what Y}$, the first positive and the second negative, and no extra intersections with $F_{\what X}$.  Since these new intersection points are adjacent along the curve, and of opposite signs,  it follows that the new count $m_{\what Z} =  m_{xy}(w_{\what Z})$ is equal to the old one $m_Z$.  Similarly $m_{\what X} = m_X$ and $m_{\what Y} = m_Y$.   Hence $\mu =  m_X + m_Y + m_Z  - t$  increases by 1,  as claimed. 
 
The heart of the proof of \thmref{A} will be to show that application of the delta move increases Pontryagin's $\nu$-invariant by $2$. 

\end{example}

%% file: 3pontryagin.tex
%%%%%%%%%%%%%%%%%%%%%%%%%
%% SECTION 3: The Pontryagin invariant
%%%%%%%%%%%%%%%%%%%%%%%%%

\section{The Pontryagin $\nu$-invariant} 
\label{sec:pontryagin}

Heinz Hopf proved in 1931 that homotopy classes of maps from the $3$-sphere to the $2$-sphere are in one-to-one correspondence with the integers via his now famous Hopf invariant.  Pontryagin \cite{Pontryagin41} generalized this to give the homotopy classification for maps from an arbitrary finite $3$-complex to the $2$-sphere.  It is convenient for our purposes to restrict to smooth maps from $3$-manifolds to the $2$-sphere,  and to use Poincar\'e duality to take Pontryagin's result, originally presented via cohomology, and reformulate it in homological terms.

%%%%%%%%%%%%%%%%%%%%%%%%%%%%%%%%%%
\part{Definition and properties}
%%%%%%%%%%%%%%%%%%%%%%%%%%%%%%%%%%

Fix a closed oriented smooth $3$-manifold $\man$.  The homotopy classification of maps $f: \man\to S^2$  can be expressed using two differential topological invariants.  

The primary invariant
$$
\lambda(f) \ \in \ H_1(\man) \qquad\text{(with integral coefficients understood)}
$$
records the homology class of the preimage link $\bl = f^{-1}(*)$
of any regular value $\pt$ of $f$ (oriented as explained below), or equivalently the Poincar\'e dual of the pull-back $f^*(\omega)$ of the orientation class $\omega\in H^2(S^2)$.  It is easily shown that two maps have the same primary invariant if and only if they induce the same map on homology.     

The secondary invariant 
$$
\nu(f_0,f_1) \ \in \ \bz_{2d(\lambda)}
$$ 
compares two maps $f_0$ and $f_1$ with the same primary invariant $\lambda$.  Here $d(\lambda)$ is the divisibility of $\lambda$ as an element of the free abelian group $H_1(\man)/$torsion.  Thus $d(\lambda) = 0$ if $\lambda$ is of finite order, and otherwise $d(\lambda)$ is the largest positive integer $d$ for which $\lambda = d\kappa$ for some $\kappa\in H_1(\man)$.   

For example, if $\man$ is the $3$-torus $\tor$, then 
$
\lambda(f) = (p,q,r) \in H_1(\tor) \cong \bz^3
$
where $p$, $q$ and $r$ are the degrees of $f$ restricted to the coordinate $2$-tori, and the divisibility $d(p,q,r)$ is the greatest common divisor of $p$, $q$ and $r$. 

In the next few paragraphs, we discuss these invariants $\lambda$ and $\nu$ in more detail, and provide a natural way, in the special case that $\man$ is the $3$-torus, to transform $\nu$ from a relative invariant into an absolute one, meaning a function of a single map.

First recall that a {\it framing} of a smooth link in $M$ is a homotopy class of trivializations of the normal bundle of the link. It can be represented by an orthonormal triple $(\u,\vv,\t)$ of vector fields along the link with respect to some Riemannian metric on $M$, where $\t$ is tangent to the link.  We can use this to orient the link by $\t$ by insisting that the triple give the orientation on $M$.  Conversely, if the link is already oriented by $\t$, then the framing can be specified by a single unit normal vector field $\u$, as $\vv$ is then determined by the condition that $(\u,\vv,\t)$ be an orthonormal frame giving the orientation on $M$.  In pictures, therefore, we often indicate a framing on an oriented link by simply drawing a thin parallel push-off of the link, recording the tips of the vectors in $\u$.

Now given a map $f:\man\to S^2$ with regular value $*$, the link $\bl = f^{-1}(*)\subset \man$ inherits a framing by pulling back an oriented basis for the tangent space to $S^2$ at $*$, and acquires an orientation from this framing, as above.   Equipped with this framing and orientation, $\bl$ will be called the {\itb Pontryagin link} of $f$ at $\pt$.  

Note that any framed oriented link $\bl$ in $\man$ arises as the Pontryagin link of some map from $M$ to $S^2$.  In particular, the {\itb Pontryagin--Thom construction} produces such a map, given by wrapping each normal disk fiber of a tubular neighborhood of $\bl$ around the $2$-sphere by the exponential map, using the framing to identify the fiber with the disk of radius $\pi$ in the tangent space to $S^2$ at $*$, and sending everything outside the neighborhood to the antipode of $*$.  This construction provides a one-to-one correspondence 
$$
[M,S^2] \ \longleftrightarrow \ \Omega_1^{\text{fr}}(\man)
$$
where $[M,S^2]$ is the set of homotopy classes of maps $\man\to S^2$, and $\Omega_1^{\text{fr}}(\man)$ is the set of framed bordism classes of framed oriented links in $\man$  (see e.g.\ Chapter 7 in Milnor \cite{Milnor65}).  

Now we return to our discussion of the invariants associated with maps $f:\man\to S^2$.  

An easy argument shows that the primary invariant $\lambda(f)$, the homology class of the oriented link $f^{-1}(*)$, is independent of the choice of regular value $*$, and that $\lambda(f)$ is invariant under homotopies of $f$.  Indeed, the preimages $\bl_0$ and $\bl_1$ of any regular values for any pair of maps homotopic to $f$ are bordant in $\man\times[0,1]$ (see Milnor \cite{Milnor65} for a proof).  It follows that $\bl_0$ and $\bl_1$ are homologous in $\man$.  Conversely, homologous links in $\man$ are bordant in $\man\times[0,1]$ by a standard argument going back to Thom \cite{Thom}.   
%%%%
\begin{comment} 
This follows from transversality and the fact that $H^2(X)=[X,CP^2]$ for $X=\man$ or $\man\times[0,1]$
\end{comment} 
%%%%
This shows that the partition of $[\man,S^2]$ into subsets $[\man,S^2]_\lambda$ according to their primary invariants $\lambda\in H_1(\man)$ corresponds to the partition of $\Omega_1^{\text{fr}}\man$ into {\it unframed} bordism classes.

The secondary invariant associated with a pair of bordant framed links measures the obstruction to extending the framings on the links across any bordism between them.    More precisely, given $f_0$ and $f_1$ in $[\man,S^2]_\lambda$ with Pontryagin links $\bl_0$ and $\bl_1$, and an oriented surface 
$\bF \subset \man\times[0,1]$ with $\partial\bF  =  \bl_1\times1 - \bl_0\times0$,  the framings on $\bl_0$ and  $\bl_1$ combine to give a normal framing of $\bF$ along its boundary.  The obstruction to extending this framing across $\bF$ is measured by its relative Euler class $e(\bF)$ in $H^2(\bF,\partial\bF;\pi_1SO(2)) = \bz$,
which in homological terms is the intersection number of $\bF$ with a generic perturbation of itself that is directed by the given framings along $\partial\bF$.   This class depends on the choice of the regular values of $f_0$ and $f_1$ used to define the Pontryagin links, and on the choice of the bordism $\bF$ between the links.  But the residue class of $e(\bF)$ mod $2d(\lambda)$ does not depend on these choices; see Gompf~\cite{Gompf} and Cencelj, Repov{\v s} and Skopenkov~\cite{Cencelj} for details, and also see Auckly and Kapitanksi~\cite{AucklyKapitanski}.  
This residue class 
$$
\nu(f_0,f_1)  \ = \ e(\bF) \textup{ mod }2d(\lambda) \ \in \ \bz_{2d(\lambda)}
$$
will be referred to as the {\itb relative Pontryagin $\nu$-invariant} of $f_0$ and $f_1$. 

%%%%%%%%%%%%%%%%%%%%%%%%%%%%%%%%%%
\part{Converting the relative Pontryagin $\nu$-invariant into an absolute invariant}
%%%%%%%%%%%%%%%%%%%%%%%%%%%%%%%%%%

The task of converting $\nu$ into an {\it absolute} invariant, that is, changing it from a function of two variables to a function of one variable, requires the choice of a {\itb base map} $f_\lambda$ in each subset $[\man,S^2]_{\lambda}$ of homotopy classes of maps with primary invariant $\lambda$.  One can then define the {\itb\boldmath absolute Pontryagin $\nu$-invariant} by
$$
\nu(f) \ = \ \nu(f,f_\lambda)
$$
for any $f\in[\man,S^2]_\lambda$.  Whether such choices can be made in a topologically meaningful way depends on the manifold $\man$. 
 
For example, Pontryagin \cite[page 356]{Pontryagin41} explicitly cautioned against trying to make this choice when $\man=S^1\times S^2$.  In this case there are, up to homotopy, exactly two maps to $S^2$ with primary invariant $1\in H_1(S^1\times S^2) = \bz$, meaning degree $1$ on the cross-sectional $2$-spheres in $S^1\times S^2$.  One of these is the projection $f_0(\theta,x)  =  x$ to the $S^2$ factor,  and the other is the twist map $f_1(\theta,x)  =  \rot_\theta(x)$ that rotates the $S^2$ factor once while traversing the $S^1$ factor; here $\rot_\theta$  indicates rotation of  $S^2$ through an angle $\theta$ about its polar axis.  Note that the double twist $f_2(\theta,x)  =  \rot_{2\theta}(x)$ is homotopic to $f_0$.   

What Pontryagin observed is that $f_0$ and $f_1$ differ by an automorphism of $S^1\times S^2$, and so there is no natural way to choose which one should serve as the base map for the homology class $[S^1\times S^2,S^2]_1$.  More precisely, the automorphism $h$ given by $h(\theta,x)  =  (\theta,\rot_\theta(x))$ satisfies $f_0h=f_1$, while $f_1h=f_2$, which is homotopic to $f_0$.  Thus the maps $f_0$ and $f_1$ have equal topological status, and so neither is more basic than the other.

A key feature of this example is the existence of a homotopically nontrivial automorphism of $S^1\times S^2$ that induces the identity on homology.  In general, if a $3$-manifold $\man$ supports such an automorphism $h$, then $\lambda(fh) = \lambda(f)$ for any map $f:\man\to S^2$.   Hence the existence of $h$ provides a potential obstruction to the natural choice of base maps in $[\man,S^2]_\lambda$ for all $\lambda$.  The $3$-torus has no such automorphisms, due to the fact that its higher homotopy groups vanish.   

For each triple of integers $p$, $q$  and $r$,  we now show how to pick out a specific map $f_{pqr}:\tor\to S^2$ having these preassigned cross-sectional degrees, which can serve in a topologically meaningful way as the base map for the set $[\tor,S^2]_{(p,q,r)}$ of all such maps.   

We will describe $f_{pqr}$ by specifying its Pontryagin link $\bl_{pqr}$, as follows.  Choose three pairwise disjoint circles $\calc_s$, $\calc_t$ and $\calc_u$ that are cosets of the coordinate circle subgroups 
$$
S^1_s = S^1 \times 0 \times 0 \ , \ S^1_t = 0 \times S^1 \times 0 \ \text{ and } \ S^1_u = 0 \times 0 \times S^1
$$
of the $3$-torus (where as usual $S^1=\br/2\pi\bz$) with disjoint tubular neighborhoods $\caln_s$, $\caln_t$ and $\caln_u$.  Equip these circles with their {\itb coordinate framings} induced from the Lie framing $(\partial_s,\partial_t,\partial_u)$ of the tangent bundle of $\tor$, that is, $(\partial_t,\partial_u)$ for $\calc_s $, $(\partial_u,\partial_s)$ for $\calc_t$ and $(\partial_s,\partial_t)$ for $\calc_u$.  Then construct $\bl_{pqr}$ from $p$ parallel copies of $\calc_s$ in $\caln_s$ (meaning $p$ distinct cosets of $S^1_s$ lying in $\caln_s$), $q$ copies of $\calc_t$ in $\caln_t$ and $r$ copies of $\calc_u$ in $\caln_u$, all with their coordinate framings, as indicated in \figref{basemaps}.  Thus $f_{pqr}$ wraps the disk fibers of the tubes $\caln_s$, $\caln_t$ and $\caln_u$ around $S^2$ by maps of degree $p$, $q$ and $r$, and is constant elsewhere.

%%% FIGURE 6: Base Maps %%%
\begin{figure}[h!]
\includegraphics[height=120pt]{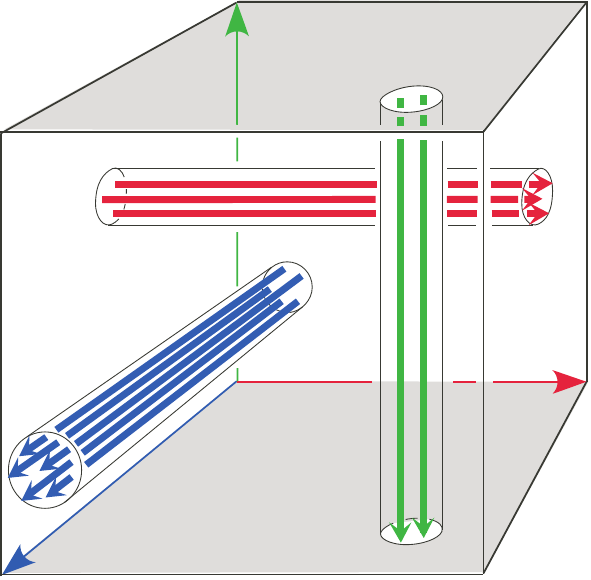}
\put(-128,-2){\small$s$}
\put(-76,122){\small$u$}
\put(1,39){\small$t$}
\put(-110,40){\small$p$}
\put(-90,66){\small$q$}
\put(-50,25){\small$r$}
\caption{Pontryagin link $\bl_{pqr}$ for the base map $f_{pqr}$}
\label{fig:basemaps}
\end{figure}
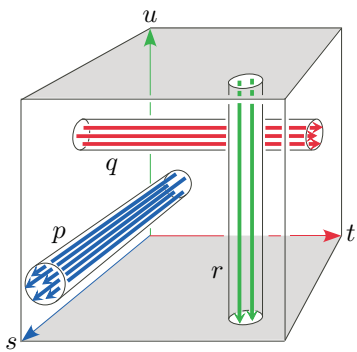
%%%%%%%%%%%%%%%%%%%%

Here and below, the $3$-torus is pictured as the cube $[0,2\pi]^3$ in $stu$-space with opposite faces identified.  The axes correspond to the coordinate circles $S^1_s$, $S^1_t$ and $S^1_u$, and the coordinate framings are the ``blackboard" framings, i.e.\ those given by parallel push-offs in the projection shown.  The link $\bl_{5\,3\,-2}$ is shown in the figure.  Note that by construction $\bl_{000}$ is empty, so $f_{000}$ is constant.

%%%%%%%%%%%%%%%%%%%%%%%%%%%%%%%%%%
\part{Computing the absolute Pontryagin $\nu$-invariant}
%%%%%%%%%%%%%%%%%%%%%%%%%%%%%%%%%%

Our goal is to give a simple procedure for computing the Pontryagin $\nu$-invariant  of a map $f:\tor\to S^2$ from a ``toral diagram" of its Pontryagin link $\bl$.

By definition, a {\itb toral diagram} of $\bl$ consists of 
\vspace{.02in}
\begin{indentation}{.05em}{1em}
\begin{enumerate}
\item a classical oriented link diagram $\cald$ in the $2$-torus $\torh$ with {\itb crossings} $\calc$, 
\smallskip
\item a finite set of signed points in $T^2$, the {\itb marked points}, partitioned into
the {\itb isolated} ones $\calm \subset T^2-\cald$, and the {\itb internal} ones $\caln\subset\cald-\calc$, \label{markedpoints} and
\smallskip
\item  integer {\itb framings} for each component of $\cald$ and each point in $\calm$.
\end{enumerate}  
\end{indentation}
\vspace{.02in}

For example, the link $\bl_{pqr}$ above is duplicated in \figref{diagrams}(a) with its toral diagram beneath it, and another example is given in \figref{diagrams}(b).        

%%% FIGURE 7: Diagrams %%%
\begin{figure}[h!]
\includegraphics[height=175pt]{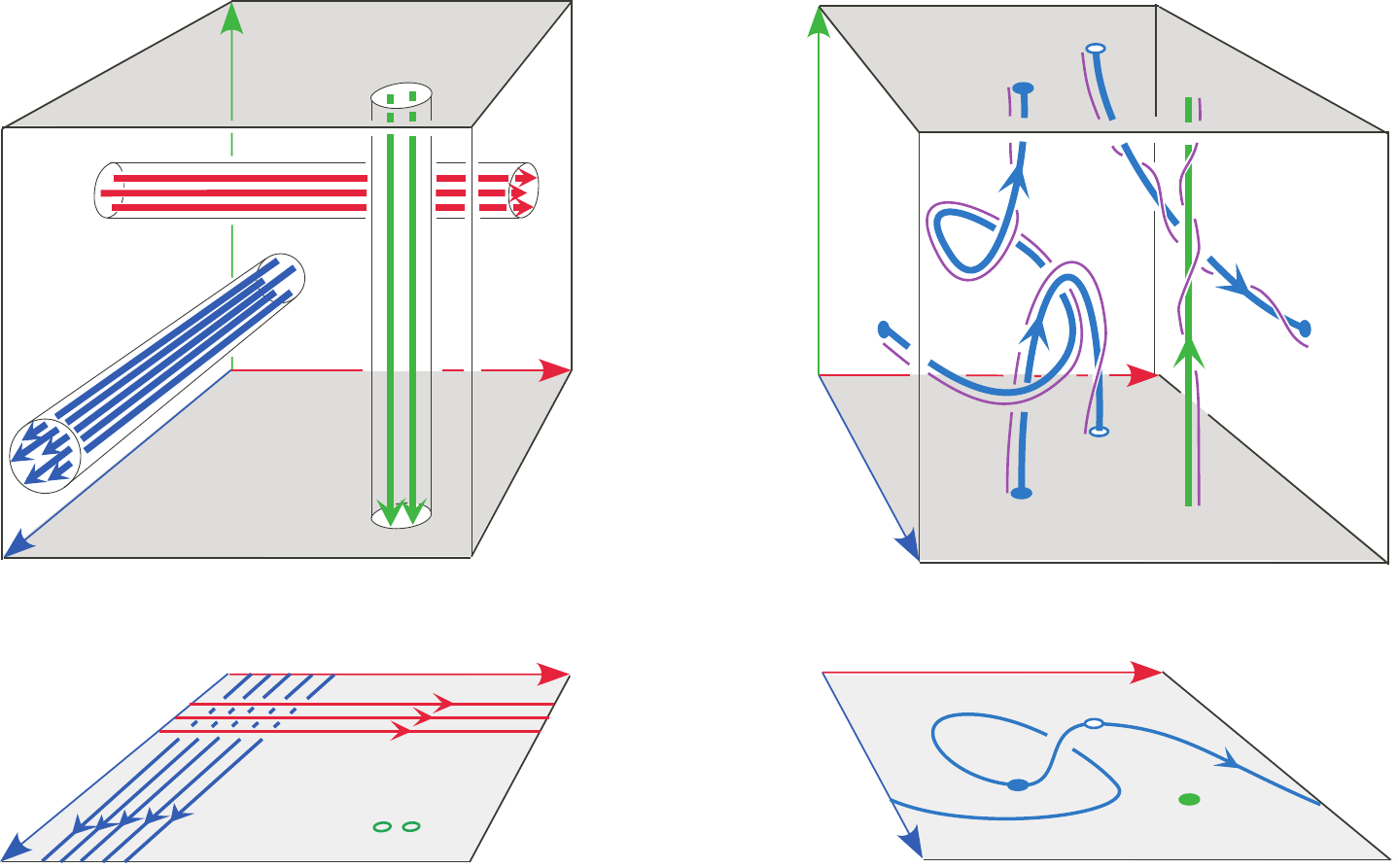}
\put(-100,24){\tiny$3$}
\put(-39,12){\tiny$2$}
\put(-300,-20){(a) The link $\bl_{pqr}$ and its diagram}
\put(-100,-20){(b) Another example}
\caption{Toral diagrams of framed links in $\tor$}
\label{fig:diagrams}
\end{figure}
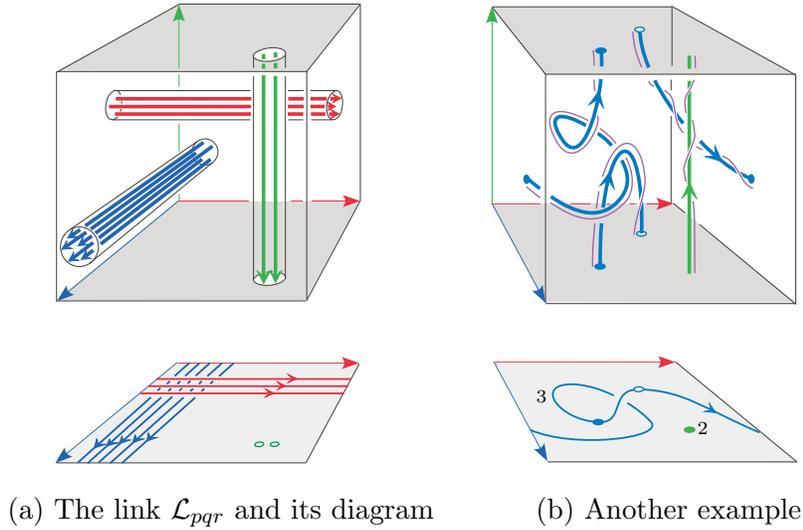
%%%%%%%%%%%%%%%%%%%%

It is understood that $\bl$ should project to $\cald\cup\calm$ under the projection $\tor\to\torh$ sending $(s,t,u)$ to $(s,t)$, and so points in $\calm$ correspond to vertical components $\pt\times\pt\times S^1$ in $\bl$, oriented up or down according to the signs.  The points in $\caln$ correspond to transverse intersections of $\bl$ with the {\itb horizontal $2$-torus} \,$\torh\times0$ -- shaded in the figures -- where the sign $+1$ or $-1$ indicates whether the curve points up or down near the intersection. 

To explain the crossings $\calc$, view $\tor$ as the cube $[0,2\pi]^3$ in $stu$-space with opposite faces identified, as before, and $\torh$ as the square $[0,2\pi]^2$ in the $st$-plane with opposite sides identified.  Above $\cald-\caln$, the link $\bl$ resides in $\torh\times(0,2\pi)$.  At a crossing, the over-crossing strand is the one with the larger vertical $u$-coordinate in the cube.

Finally, the framings specify a push-off of $\bl$ by comparison with the blackboard framing of $\cald\cup\calm$ in $\torh$.  These ``blackboard" framings are obtained by pushing $\cald\cup\calm$ off itself in the direction of a normal vector field in $\torh$, and then lifting these push-offs to a collection of framing curves for the components of $\bl$ that we call their {\it $0$-framings}.  For vertical circles, these are just the coordinate framings.  Now the $n$-framing on any given component of $\bl$ is the one obtained from the $0$-framing by adding $n$ full twists, right or left handed according as $n$ is positive or negative.

In these figures we typically indicate the framing on the link in $\tor$ by a thin push-off, and use the convention that the unlabeled components of the diagram have framing zero, that is, the blackboard framing.  We also denote the positive marked points in the diagram with solid dots, and the negative ones with hollow dots.

We have allowed isolated marked points in the definition above so that projections of links such as $\bl_{pqr}$ will qualify as toral diagrams, and to facilitate our later work.  Note, however, that a generic isotopy of $\bl$ will eliminate the set $\calm$ of isolated marked points, converting an $n$-framed isolated marked point -- corresponding to a vertical circle in $\bl$ -- into a small $(n\pm1)$-framed circle with one internal marked point -- corresponding to a spiral perturbation of the vertical circle.   More precisely, using the notation above of solid dots for positive marked points and hollow dots for negative ones, we have
%%% Trivial Circles %%%
%\begin{figure}[h!]
$$
\includegraphics[height=25pt]{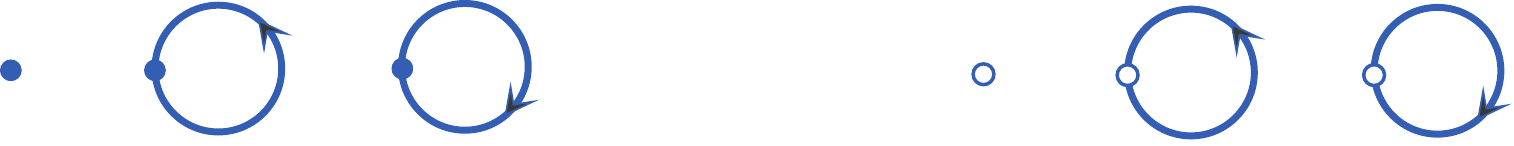}
\put(-250,11){$=$}
\put(-205,11){$=$}
\put(-261,3){\small$n$}
\put(-235,-5){\small$n-1$}
\put(-190,-5){\small$n+1$}
\put(-135,10){and}
\put(-85,10){$=$}
\put(-40,10){$=$}
\put(-93,2){\small$n$}
\put(-66,-6){\small$n+1$}
\put(-24,-6){\small$n-1$}
%\caption{Eliminating isolated marked points}
%\label{fig:trivial}
%\end{figure}
%%%%%%%%%%%%%%%%%%%%
$$
as is readily verified by a suitable picture in the $3$-torus.

Now suppose we are given a toral diagram of a Pontryagin link $\bl$ for a map $f:\tor\to S^2$.  We say that the diagram {\itb represents} $f$, and seek to compute $\nu(f)$ from it.

First observe that the primary invariant  
$$
\lambda(f) \ = \ (p,q,r) \label{primary}
$$
is easily read from the diagram.  Indeed $p$ and $q$ (which are the degrees of the projections of $\bl$ to the horizontal circle factors $S^1_s$ and $S^1_t$) are just the intersection numbers of $\cald$ with $S^1_t$ and $-S^1_s$, and $r$ (the degree of the projection of $\bl$ to the vertical circle $S^1_u$\,, or equivalently the intersection number of $\bl$ with the horizontal $2$-torus) is the sum of the signs of all the marked points.  We call $p$ and $q$ the {\itb horizontal winding numbers} of the diagram, and $r$ its {\itb vertical winding number}.   For what follows, we will also need to consider the vertical winding of the individual components of $\bl$.  Each such component $\bl_i$ projects either to some subset $\cald_i$ of $\cald$, or to a point $\calm_i$ in $\calm$ (if $\bl_i$ is vertical).  In either case, we call this projected image the ``$i^\text{th}$ component" of the diagram, and write $r_i$ for the sum of all the marked points that lie on it.  Clearly $r_i$ is just the vertical winding number of $\bl_i$, and $\sum r_i = r$, the total vertical winding number of the diagram. 

To compute $\nu(f)$, we will transform $\bl$ by a sequence of framed bordisms into $\bl_{pqr}$ with some extra twists in the framing.  The number of twists is by definition $\nu(f)$; counting these twists ultimately yields the simple formula for $\nu(f)$ that will be given in \propref{pontryagin}.   

To cleanly state this formula, we assume from the outset that our diagram has no crossings.  There is no loss of generality in doing so since the crossings can be eliminated by {\itb saddle bordisms} of $\bl$, as illustrated in \figref{crossing} (viewing $\bl$ from the top, looking down on a crossing in the horizontal torus) which of course do not change $\nu(f)$.   Furthermore, we will see that the Pontryagin link of a ``generic link" (to be defined in \secref{Aprep}), has a toral diagram without crossings.

%%% FIGURE 8: Crossing %%%
\begin{figure}[h!]
\includegraphics[height=70pt]{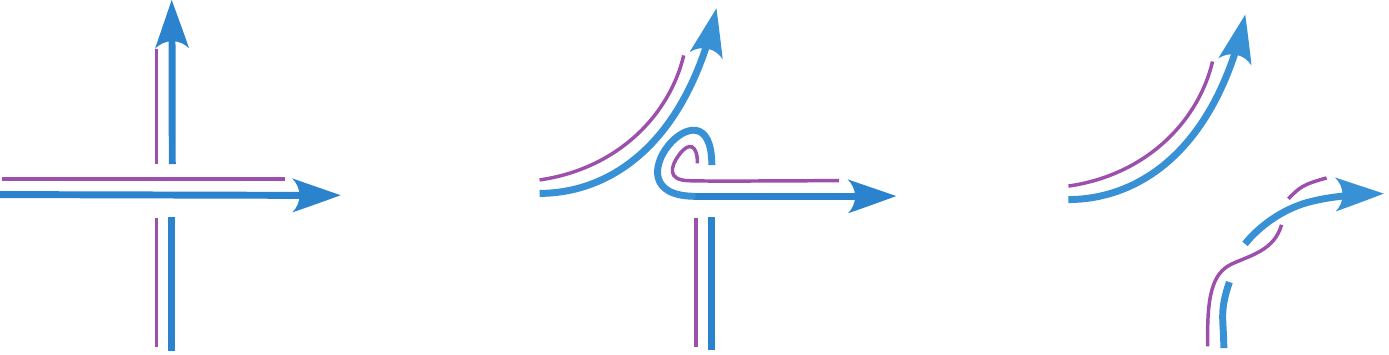}
\put(-193,28){$\sim$}
\put(-201,17){\tiny bordism}
\put(-87,28){$=$}
\caption{Using saddles to eliminate crossings}
\label{fig:crossing}
\end{figure}
%%%%%%%%%%%%%%%%%%%%

Each such saddle bordism is achieved by adding a $1$-handle, as indicated on the left side of \figref{bordism}, and drawn in full, suppressing one dimension in $\tor$, on the right side.

%%% FIGURE 9: Crossing %%%
\begin{figure}[h!]
\includegraphics[height=80pt]{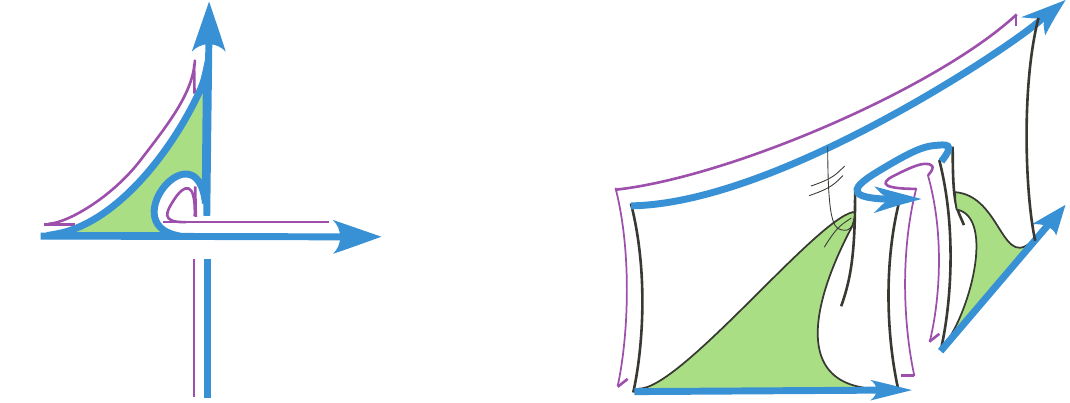}
\caption{The saddle bordism}
\label{fig:bordism}
\end{figure}
%%%%%%%%%%%%%%%%%%%%

The effect of the saddle bordism on the diagram is easy to describe.  If it is used to eliminate a self-crossing of an $n_i$-framed component $\cald_i$ (meaning $\cald_i$ is the projection of a component $\bl_i$ of $\bl$) then $\cald_i$ splits into two components whose framings must add up to $n_i\pm1$, depending on the sign of the crossing.  Beyond this condition, the framings on the new components can be chosen arbitrarily since twists in the original framing can be shifted along $\bl_i$ at will.  If the saddle bordism is used to eliminate a crossing between distinct components $\cald_i$ and $\cald_j$ of $\cald$ with framings $n_i$ and $n_j$, then the result is a single component with framing $n_i+n_j\pm1$.     

Once we have a toral diagram {\sl without crossings} representing $f$, the extra data needed to compute $\nu(f)$ is the list of {\itb vertical winding numbers} $r_i$ of its components together with one additional integer $n=\sum n_i$, the sum of all the component framings $n_i$, which we call the {\itb total framing} and place as a label next to the diagram. 

To state the formula efficiently, we will use one more list of numbers that is easily read from the diagram.    First pick a basepoint $*$ in $\torh$ away from $\cald\cup\calm$, and then for each $i$, choose an arc $\gamma_i$ that runs from the $i^\text{th}$ component of $\cald\cup\calm$ to $*$.   Now define the {\itb depth} \label{depth} of the $i^\text{th}$ component to be 
$$
d_i \ = \ 2\,\gamma_i\,\dt\,\cald \mod{2\gcd(p,q,r)}
%d} \qquad\text{where $d =\gcd(p,q,r)$}
\,,
$$
that is, twice the intersection number in $\torh$ of the arc $\gamma_i$ with the union $\cald$ of the closed curves in the diagram.  This intersection number must be properly interpreted for components of $\cald$.  In this case the initial point of $\gamma_i$ lies on $\cald$, contributing $\pm\frac12$ to the intersection number, and thus $\pm1$ to $d_i$.  It follows that $d_i$ is always an odd integer for components of $\cald$, and an even integer for components of $\calm$, meaning isolated marked points.  An example is shown in \figref{depth}.

%%% FIGURE 10: DEPTH %%%
\begin{figure}[h!]
\includegraphics[height=120pt]{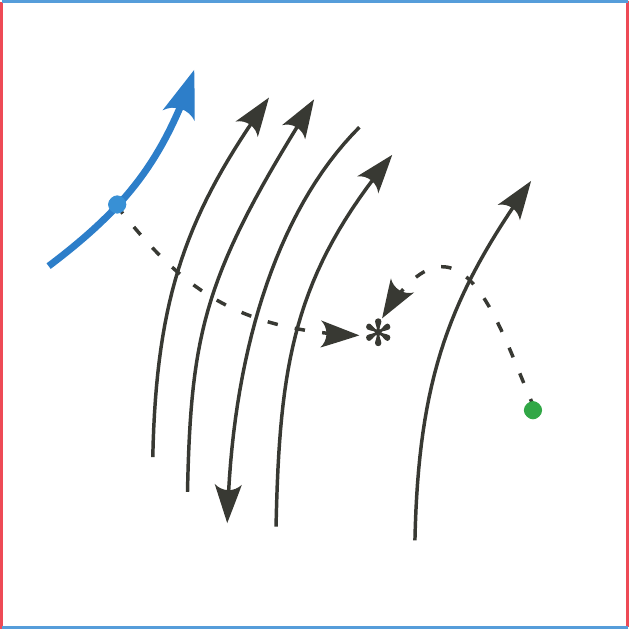}
\put(-107,87){\small$\cald_i$}
\put(-100,68){\small$\gamma_i$}
\put(-26,31){\small$\calm_j$}
\put(-20,57){\small$\gamma_j$}
\put(-195,60){$d_i = 2\cdot2\frac12 = 5$}
\put(8,47){$d_j = 2\cdot-1 = -2$}
\caption{Depths of components in the diagram}
\label{fig:depth}
\end{figure}
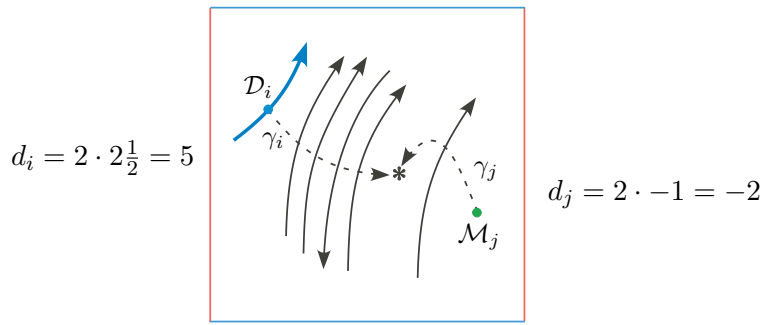
%%%%%%%%%%%%%%%%%%%%

Of course the depth $d_i$ depends on the choice of arc $\gamma_i$, but only modulo $2\gcd(p,q)$, and so it is certainly well defined modulo $2\gcd(p,q,r)$.  Furthermore, in the formula for $\nu(f)$ below, the depths appear only as coefficients in the sum $\sum d_ir_i$.  It is readily seen that this sum changes by a multiple of $\sum 2r_i = 2r$ when moving the basepoint $*$, and so it is well defined modulo $2\gcd(p,q,r)$, independent of the choices of $*$ and of the arcs $\gamma_i$. 

We can now state the main result of this section.    

%%%% PROPOSITION 3.1 %%%%%%%%%%%
\begin{proposition}
\label{prop:pontryagin}
\textbf{\boldmath Let $f:\tor\to S^2$ be a smooth map whose Pontryagin link $\bl$ is represented by a toral diagram in $\torh$ without crossings, with horizontal winding numbers $p$ and $q$, vertical winding number $r$, and total framing $n$.  Then the primary and secondary Pontryagin invariants of $f$ are given by
$$
\begin{aligned}
\lambda(f) \ &= \ (p,q,r) \quad\text{and} \\
\nu(f) \ &= \ n + pq + \textstyle\sum d_ir_i \mod{2\gcd(p,q,r)}
\end{aligned}
$$
where $r_i$ and $d_i$ are the vertical winding numbers and depths $($with respect to any chosen basepoint$)$ of the components of the diagram.}
\end{proposition}
%%%%%%%%%%%%%%%%%%%%%%%%%

\begin{proof}
The formula for $\lambda(f)$ was derived on page \pageref{primary}, and the formula for $\nu(f)$ is, as explained above, at least well defined modulo $2\gcd(p,q,r)$, independent of the choice of basepoint.  So it remains to verify that this is the correct formula for $\nu(f)$. 

First observe that, in the absence of crossings, the integers $r_i$ are all that are needed to recover the link $\bl$ up to {\sl isotopy}.  In particular, each nonvertical component of $\bl$ can be taken to wind monotonically around the last circle factor of $\tor$. Furthermore, equipped with the total framing $n$, we can recover the framed link $\bl$ up to {\sl framed bordism}.  To see this, note that a pair of saddle bordisms can be used as shown in \figref{framing} to transfer a twist in the framing from any component of $\bl$ to any other, and so we can distribute the framings in any desired way among the components.

%%% FIGURE 11: Framing %%%
\begin{figure}[h!]
\includegraphics[height=80pt]{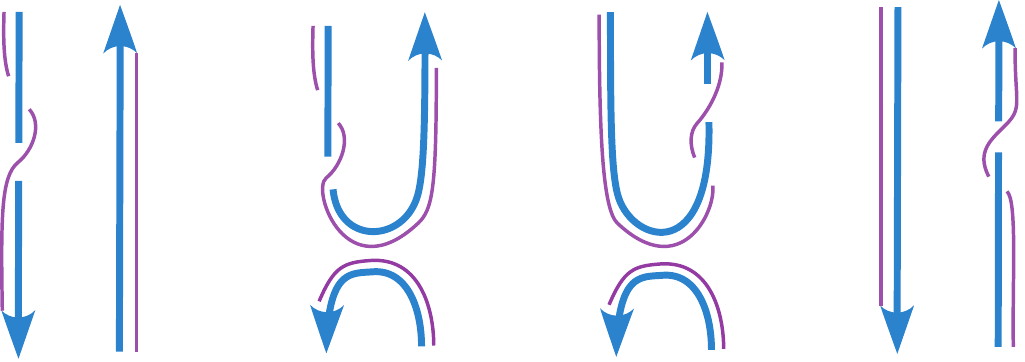}
\put(-183,32){$\sim$}
\put(-117,32){$=$}
\put(-55,32){$\sim$}
\caption{Using saddles to transfer twists}
\label{fig:framing}
\end{figure}
%%%%%%%%%%%%%%%%%%%%

We now propose to transform $\bl$ by a sequence of framed bordisms into $\bl_{pqr}$ with some extra twists in the framing.  These correspond to homotopies of $f$ and so do not change the Pontryagin invariant $\nu(f)$.  We will carry this out on a diagrammatic level, transforming our given toral diagram, with total framing $n$, to the diagram for $\bl_{pqr}$ shown in \figref{diagrams}(a) with some total framing, which by definition will equal $\nu(f)$.  So we must simply keep track of the change in the total framing as we proceed.  

We first use saddle bordisms to replace the vertical winding of each non-vertical component $\bl_i$ by $|r_i|$ zero-framed vertical components, at the cost of adding $r_i$ to the total framing.  The new vertical circles are ``adjacent" to $\cald_i$ -- meaning displaced slightly to the {\it right} of it in the projection -- and oriented up or down according to the sign of $r_i$.  

This is illustrated in \figref{winding}.  The net effect on the diagram is to reduce all the vertical winding numbers to zero on the components of $\cald$ (i.e.\ to eliminate the internal marked points $\caln$), to add $\sum|r_j|$ isolated marked points, and to add $\sum r_j$ to the total framing of the diagram, where the sums are only over the non-vertical components of $\bl$.

%%% FIGURE 12: Winding %%%
\begin{figure}[h!]
\includegraphics[height=150pt]{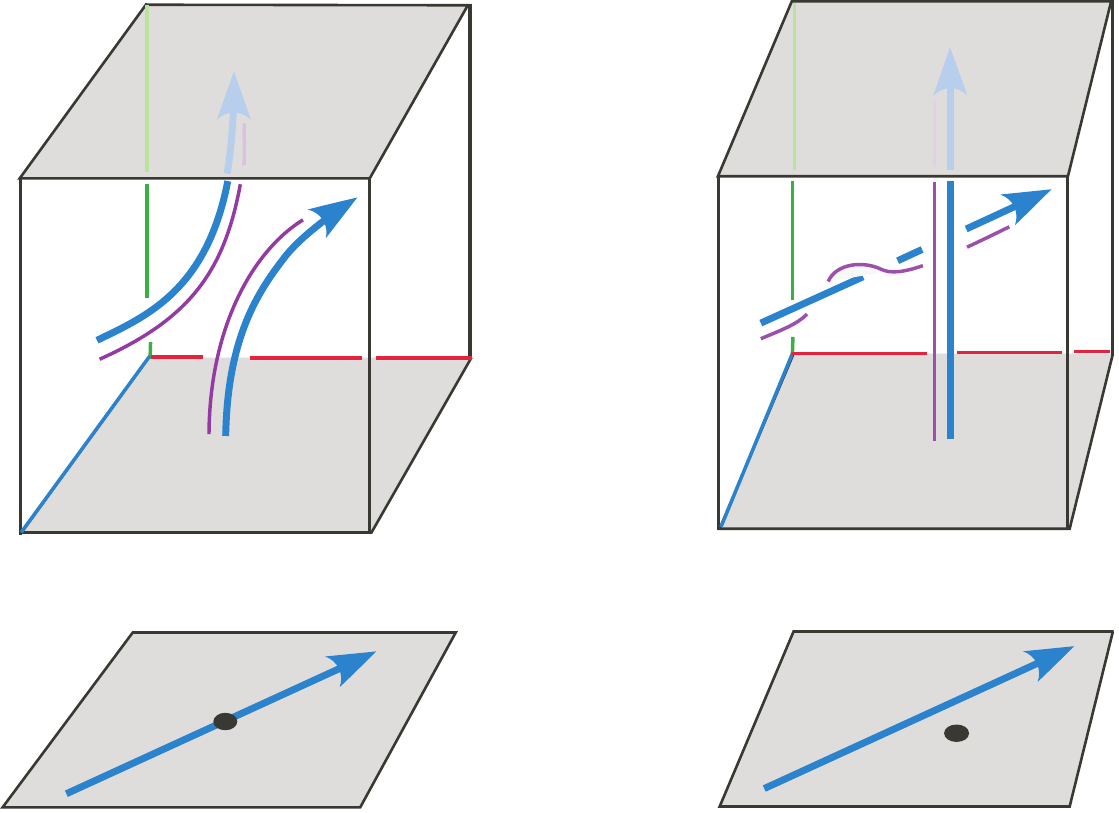}
\put(-102,100){$\sim$}
\put(-134,5){\small$n$}
\put(-4,5){\small $n+1$}
\caption{Using saddles to replace vertical winding by vertical circles}
\label{fig:winding}
\end{figure}
%%%%%%%%%%%%%%%%%%%%

Since our target is the base link $\bl_{pqr}$ (with extra twists), the vertical circles must be gathered together.  This is the purpose of our base point $*$, which will serve as a gathering spot, and the arcs $\gamma_i$, which will serve as the paths in $\torh$ along which to move the vertical circles which at the moment are adjacent to the $\bl_i$.   At each intersection of $\gamma_i$ with a component of $\cald$, these vertical circles must cross the corresponding strand of $\bl$.   This crossing can be accomplished by a pair of saddle bordisms as indicated in \figref{gathering}, adding $\pm2$ to the diagram framing.   On the diagrammatic level, this can viewed as a two step process, running the bordism shown in \figref{winding} backwards and then forwards, in order to move an isolated marked point across a component of $\cald$, as shown at the bottom of the figure. 

%%% FIGURE 13: Gathering %%%
\begin{figure}[h!]
\includegraphics[height=150pt]{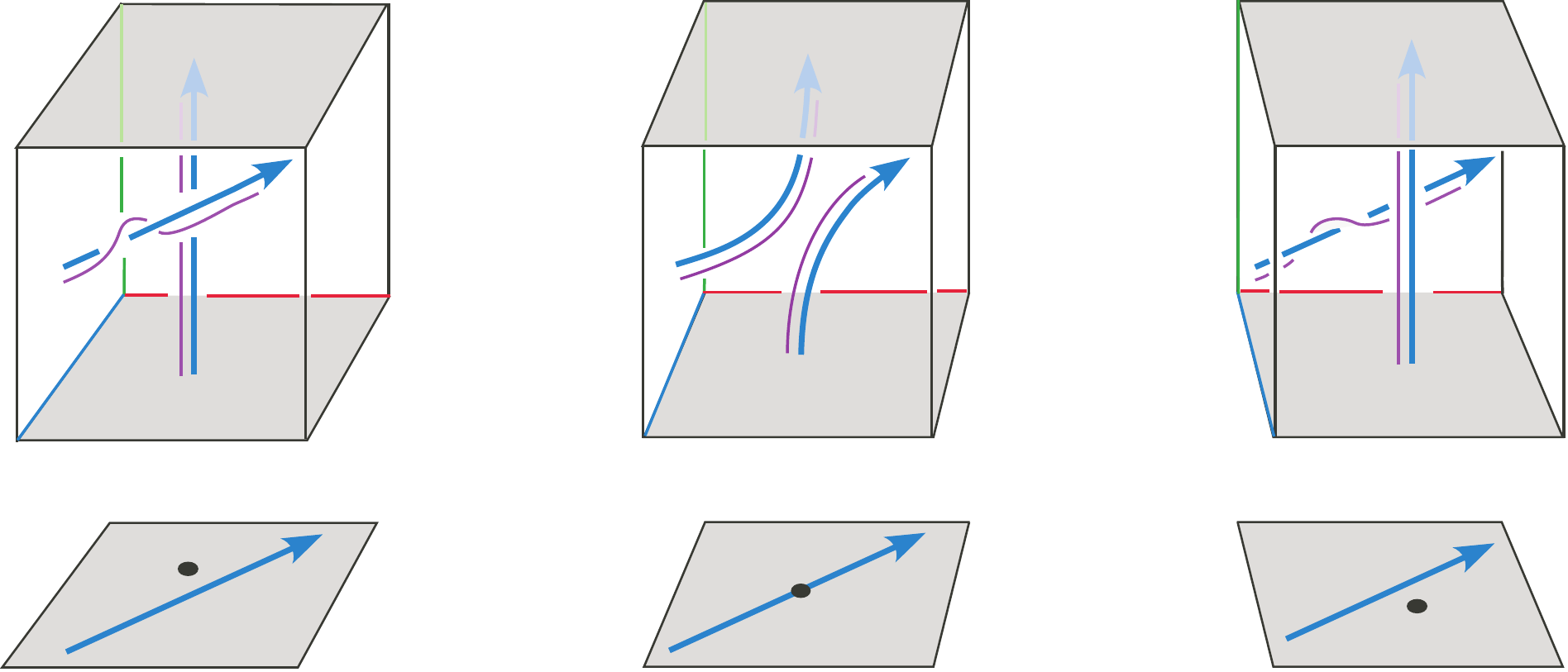}
\put(-240,100){$\sim$}
\put(-110,100){$\sim$}
\put(-275,5){\small$n-1$}
\put(-134,5){\small$n$}
\put(0,5){\small $n+1$}
\caption{Using saddles to gather the vertical circles}
\label{fig:gathering}
\end{figure}
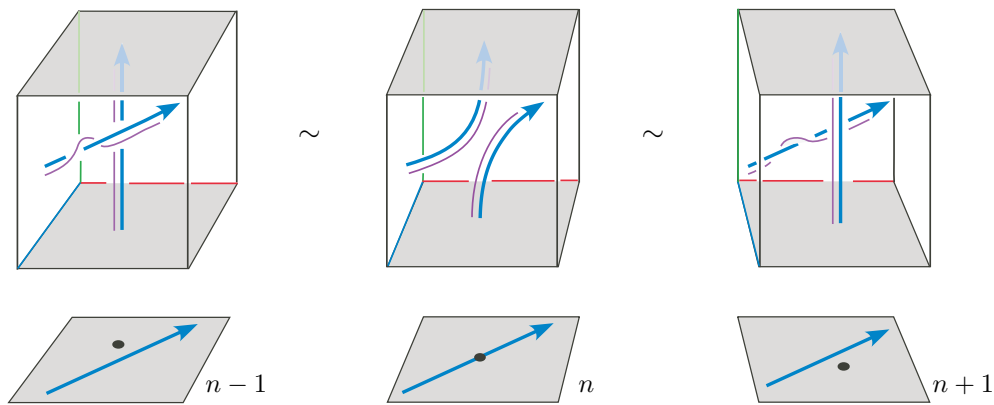
%%%%%%%%%%%%%%%%%%%%

Now it is easy to check that the act of removing the vertical winding of the nonvertical components of $\bl$, and then gathering all the resulting vertical circles (including the original ones corresponding to $\calm$) will add $\sum d_ir_i$ to the total framing of the diagram.  Visually, this process can be thought of as a ``migration" to the base point of all the marked points in the diagram. 

Next use disk bordisms to remove the null homotopic components of $\cald$ (arranging for them to be $0$-framed by storing their twists elsewhere, as explained above) and similarly use annular bordisms to remove any pair of curves that are parallel but oppositely oriented.  The diagram now consists of a $(p,q)$ torus link in $\torh$ (meaning $c$ parallel copies of the $(p',q')$ torus knot, where $p=cp'$ and $q=cq'$ with $p'$ and $q'$ relatively prime) together with $r$ isolated marked points near $*$.  The total framing, which was originally equal to $n$, has changed to $n+\sum d_ir_i$.  

Finally we transform the $(p,q)$ torus link by using saddle bordisms (reversing the process described above for removing crossings) into $p$ copies of $S^1_s$ and $q$ copies of $S^1_t$, grouped as in $\bl_{pqr}$, at the cost of adding $pq$ to the total framing.   Thus we have arrived at $\bl_{pqr}$, having changed the total framing from $n$ to $n+pq + \sum d_ir_i.$  This verifies the stated formula for $\nu(f)$, and completes the proof of \propref{pontryagin}.
\end{proof}

The following application will arise in the proof of \thmref{A} in \secref{A}.

\begin{example}\label{ex:basecase} 

If $f$ is represented by the $\pm(pq+r)$-framed $(p,q)$ torus link in $\torh$, with one component $K$ of vertical winding number $r$, and the rest of vertical winding number zero, then $\lambda(f)  = (p,q,r)$, and choosing a base point adjacent to $K$,
$$
\nu(f) \ = \ (pq+r) + pq + (1\cdot r + 0 + 0) \ = \ 2(pq+r) \ = \ 0 \ \in \bz_{2\gcd(p,q,r)}.
$$
\end{example}

%% file: 4characteristic.tex
%%%%%%%%%%%%%%%%%%%%%%%%%%%%%%%%%%
%%   SECTION 4: Explicit formula for the characteristic map      %%%
%%%%%%%%%%%%%%%%%%%%%%%%%%%%%%%%%%

\section{Explicit formulas for the characteristic map} 
\label{sec:characteristic}
%%%%%%%%%%%%%%%%%%%%%%%%%%%%%%%%%%
%%%%%%%%%%%%%%%%%%%%%%%%%%%%%%%%%%
%%%%%%%%%%%%%%%%%%%%%%%%%%%%%%%%%%

In this section we give an explicit diffeomorphism from the Grassmann manifold $\gras$ of oriented $2$-planes through the origin in $4$-space to a product of two $2$-spheres, and then use it to give a formula for the characteristic map $g_L:\tor\to S^2$ of a three-component link $L$ in the $3$-sphere. 

We also describe an alternative form $h_L:\tor\to S^2$ of the characteristic map, homotopic to $g_L$ but more convenient for the proof of  \thmref{A} that we will give in \secref{A}.   The formula for $h_L$ will reveal a close connection between the characteristic map and the classical Gauss map $T^2\to S^2$ for two-component links in $3$-space. 

Throughout we regard $\br^4$ as the algebra of {\itb quaternions}, with orthonormal basis $1,i,j,k$ and unit sphere $S^3$ (oriented so that $i,j,k$ is a positive frame for the tangent space to $S^3$ at the point $1$),  and $\br^3$ as the subspace of {\itb pure imaginary quaternions} spanned by $i$, $j$ and $k$, with unit sphere $S^2$.    Thus $S^3$ is viewed as the multiplicative group of unit quaternions, and $S^2$ as the subset of pure imaginary unit quaternions.  

For any quaternion $q = q_0+q_1i+q_2j+q_3k$, let $\bar q$ denote its conjugate $q_0-q_1i-q_2j-q_3k$, which coincides with $q^{-1}$ when $q\in S^3$, and let $\re(q) = q_0$ and $\Im(q)=q_1i+q_2j+q_3k$ denote its real and imaginary parts.

%%%%%
\part{The Grassmann manifold $\gras$}
%%%%%

It is well known
% dating back to the work of Julius Pl\"ucker in the nineteenth century
that $\gras$ can be identified with a product of two $2$-spheres.  In particular, we will use the diffeomorphism
$$
\pi: \gras \longrightarrow S^2\times S^2  
$$
that maps the oriented plane $\pb ab$, spanned by an orthonormal $2$-frame $(a,b)$ in $\br^4$, to the point $(b\bar a,\bar ab)$ in $S^2\times S^2$.  Note that both coordinates 
$$
\pi_+\pb ab \ = \ b\bar a \qquad\text{and}\qquad \pi_-\pb ab \ = \ \bar ab
$$
do in fact lie in $S^2$ since right and left multiplication by $\bar a$ are orthogonal transformations of $\br^4$, carrying the orthonormal frame $(a,b)$ to the orthonormal frames $(1,b\bar a)$ and $(1,\bar ab)$.   

To see that $\pi$ is well-defined, consider any other orthonormal basis for the plane $\pb ab$.  It must be of the form $(ac,bc)$ for some $c$ in the circle subgroup $C_{\bar ab}$ through $\bar ab$, since this group acts on the plane by rotations.   Thus it suffices to check that $bc\,\overline{ac} = b\bar a$, which is immediate, and that $\overline{ac}\, bc = \bar ab$, which is true since $c$ commutes with $\bar ab$.  

To see that $\pi$ is in fact a diffeomorphism, we can simply write down the inverse.  It maps a pair $(x,y)$ in $S^2\times S^2$ to the plane $\pb c{cy}$ where $c$ is the midpoint of any geodesic arc from $x$ to $y$ on $S^2$.   This can be verified by a straightforward calculation using the fact that conjugation by a pure imaginary quaternion rotates the $2$-sphere about that quaternion by $\pi$ radians, so $cy\bar c = x$.
% together with the observation that $C_{\bar ab}$ is the full centralizer of $\bar ab$ in $S^3$

\begin{comment}
More generally, conjugation by $\cos\theta + u\sin\theta$ rotates $S^2$ about $u$ by $2\theta$ radians.

Using this special property of the midpoint $c$, it is immediate that the map above, call it $\rho$, is a right inverse to $\pi$, since $\pi\rho(x,y) = (cy\bar c, \bar ccy) = (x,y)$.  To see that it is also a left inverse, we compute $\rho\pi\pb ab = \pb c{c\bar ab}$ where $c$ is the midpoint of any geodesic from $b\bar a$ to $\bar ab$ on $S^2$.  But then $c\bar ab\bar c = b\bar a$ by the special property of $c$, and it follows that $\bar ac$ commutes with $\bar ab$.  Therefore $c\in aC(\bar ab)$ and so $\rho\pi\pb ab \ = \  \pb c{c\bar ab} \ = \  \pb ab$.
\end{comment}

%%%%%%%%%%%%%%%%%%%%%%%%
%%%%%%%%%%%%%%%%%%%%%%%%
\part{The characteristic map $g_L$} 
%%%%%%%%%%%%%%%%%%%%%%%%
%%%%%%%%%%%%%%%%%%%%%%%%

Recall from the introduction that the Grassmann map $G:\conf\to\gras$ sends a triple $(x,y,z)$ of distinct points in $S^3$ to the plane they span in $\br^4$, translated to pass through the origin, and oriented so that
$$
G(x,y,z) \ = \ \pb{x-z}{y-z}.
$$  
We have extended notation so that for {\it any} two linearly independent vectors $a$ and $b$ in $\br^4$, the symbol $\pb ab$ denotes the oriented plane they span.  Then, given a three-component link $L$ in $S^3$, its characteristic map $g_L:\tor\to S^2$ is defined using the Grassmann map $G$ and the projection $\pi_+:\gras\to S^2$ by the formula 
$$
g_L(s,t,u) \ = \ \pi_+G(x,y,z) \ = \ \pi_+\pb{x-z}{y-z}
$$
where $x=x(s)$, $y=y(t)$ and $z=z(u)$ parametrize the components of $L$.   

To make this explicit, and to show that using $\pi_-$ in place of $\pi_+$ in the definition would not change the homotopy class of $g_L$, we need expressions for $\pi_\pm\pb ab$ when $a$ and $b$ are arbitrary linearly independent vectors in $\br^4$, but not necessarily orthonormal.  

For example, if $a$ and $b$ are orthogonal, then $b\bar a$ and $\bar ab$ are still pure imaginary, and so need only be normalized to give  $\pi_+\pb ab$ and $\pi_-\pb ab$.   

For a general pair of linearly independent vectors $a$ and $b$, the vector $c = b-(b\dt a/a\dt a)a$, where $\dt$ is the dot product in $\br^4$, is orthogonal to $a$ and satisfies $\pb ac = \pb ab$.  Therefore $\pi_+\pb ab = \pi_+\pb ac$, which equals the unit normalization of the vector $c\bar a = b\bar a - b\dt a$.  But $b\dt a = \re(b\bar a)$, and so $\pi_+\pb ab$ is the unit normalization of $\Im(b\bar a)$.  Similarly $\pi_-\pb ab$ is the unit normalization of $\Im(\bar ab)$.  Therefore, for {\it any} two linearly independent vectors $a$ and $b$, we have  
$$
\pi_\pm\pb ab \ = \ (a,b)_\pm / |(a,b)_\pm|  
$$
where $(\,,\,)_\pm$ are the skew symmetric bilinear forms on $\br^4$ defined by 
$$
(a,b)_+ \ = \ \Im\,b\bar a \qquad\text{and}\qquad (a,b)_- \ = \ \Im\,\bar ab\,.
$$

\begin{comment}
Note that both vectors $(a,b)_\pm$ have norm equal to the area of the parallelogram spanned by $a$ and $b$.
\end{comment}

It follows that $g_L(s,t,u)$ is the unit normalization of the vector
$$
\begin{aligned}
F(x,y,z) \ &= \ (x-z,y-z)_+ \ = \ (x,y)_+ + (y,z)_+ + (z,x)_+ \\
&= \ (ix\dt y + iy\dt z + iz\dt x \ , \  jx\dt y + jy\dt z + jz\dt x \ , \  kx\dt y + ky\dt z + kz\dt x)
\end{aligned}
$$
where the last expression follows from the formula $(a,b)_+ = (ia\dt b, ja\dt b, ka\dt b)$.  This formula is seen as follows.  By definition, $(a,b)_+ = \Im(b\bar a)  =  (i\dt b\bar a)i + (j\dt b\bar a)j + (k\dt b\bar a)k$.  But $q\dt b\bar a = qa\dt b$ for any $q$ (in particular $i$, $j$ or $k$) since right multiplication by the unit quaternion $a/|a|$ is an isometry, and so $(a,b)_+ = (ia\dt b)i + (ja\dt b)j + (ka\dt b)k = (ia\dt b, ja\dt b, ka\dt b)$.
% Similarly $\Im(\bar a b)= (ai\dt b, aj\dt b, ak\dt b)$
Summarizing, we have shown:
 
%%%%%%%%%
\begin{proposition}
\label{prop:char}
\textbf{\boldmath The characteristic map $g_L:\tor\to S^2$ of a three-component link $L$ in $S^3$ is given by the formula
$$
g_L(s,t,u) \ = \ F(x,y,z)/\|F(x,y,z)\|
$$
where $x = x(s)$, $y = y(t)$ and $z = z(u)$ parametrize the components of $L$, and $F:\conf\to\br^3$ is the function defined above.}
\end{proposition}
%%%%%%%%%

This formula for $g_L$ will be used in our proof of \thmref{B} in Sections \ref{sec:B} and \ref{sec:B'}.  
For \thmref{A} it will be more convenient, for the most part, to use an alternative form of the characteristic map that we introduce next.

%%%%%%%%%%%%%%%%%%%%%%%%
%%%%%%%%%%%%%%%%%%%%%%%%
\part{The asymmetric characteristic map $h_L$}
%%%%%%%%%%%%%%%%%%%%%%%%
%%%%%%%%%%%%%%%%%%%%%%%%

Given a three-component link $L$ in the $3$-sphere, we define below an asymmetric version 
$$
h_L:\tor\longrightarrow S^2
$$
of the characteristic map in which the last component of $L$ plays a distinguished role, and show that it is homotopic to the earlier defined symmetric characteristic map $g_L$.  As will be seen, the map $h_L$ can be viewed as a parametrized family of Gauss maps for twisted versions of the first two components of $L$, parametrized by the third component.  

For notational economy, we will use $\norm{q}$ to denote the unit normalization $q/|q|$ of any nonzero quaternion $q$.

The key observation that motivates the definition of $h_L$ is that the Grassmann map $G:\conf\to\gras$ factors up to homotopy through the {\itb Stiefel manifold} $\stiefel$ of orthonormal $2$-frames in $4$-space.  

More precisely, let $\pr_z$ denote {\itb stereographic projection} of $S^3-\{z\}$ onto $z^\perp$, the $3$-plane through the origin in $\br^4$ orthogonal to $z$.  Then we define the {\itb Stiefel map}
$$
H:\conf\longrightarrow\stiefel \quad,\quad (x,y,z) \longmapsto (z,\norm{\pr_zx-\pr_zy})
$$
(recall that the square brackets signify unit normalization), and will show that it is a homotopy equivalence whose composition with the canonical projection 
$$
P:\stiefel\longrightarrow\gras \quad,\quad (z,v)\longmapsto\pb zv
$$
is homotopic to the Grassmann map $G$.  

To see this, first observe that there is a deformation retraction of $\conf$ to its subspace 
$$
V  \ = \ \{(v,-z,z) \st v\perp z\}
$$ 
defined as follows.  Start with $(x,y,z)\in\conf$, and consider the points $\pr_zx$ and $\pr_zy$ in $z^\perp$.   Translation in $z^\perp$ moves $\pr_zy$ to the origin, and then dilation in $z^\perp$ makes the translated $\pr_zx$ into a unit vector.   Conjugating this motion by $\pr_z$ moves $x$ to $\norm{\pr_zx-\pr_zy}$, moves $y$ to $-z$, and leaves $z$ fixed, as pictured in \figref{stfl}, thus defining a deformation retraction of $\conf$ to its subspace $V$, sending $(x,y,z)$ to $(\norm{\pr_zx-\pr_zy},-z,z)$.

%%% FIGURE 14: Deformation retraction to V %%%
\begin{figure}[h!]
\includegraphics[height=150pt]{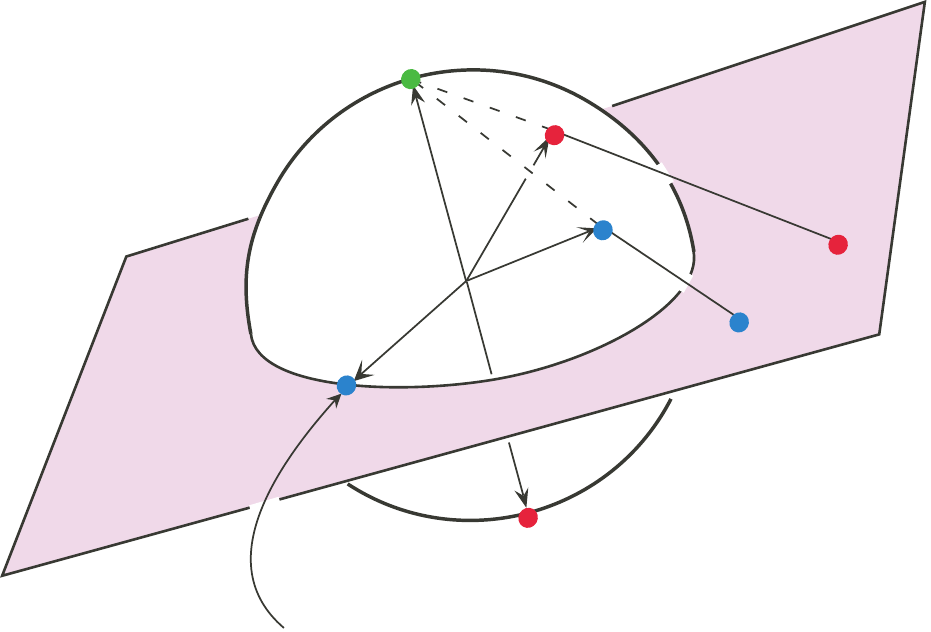}
\put(-127,136){\small$z$}
\put(-105,18){\small$-z$}
\put(-80,103){\small$x$}
\put(-92,123){\small$y$}
\put(-40,72){\small $\pr_zx$}
\put(-30,100){\small $\pr_zy$}
\put(-187,45){$z^\perp$}
\put(-145,90){$S^3$}
\put(-150,0){\small$\norm{\pr_zx-\pr_zy} \ = \ (\pr_zx-\pr_zy)/|\pr_zx-\pr_zy|$}
\caption{The deformation retraction $\conf\to V$}
\label{fig:stfl}
\end{figure}
%%%%%%%%%%%%%%%%% 

Identifying $V$ with $\stiefel$ via $(v,-z,z)\leftrightarrow(z,v)$, this shows that $H$ is a homotopy equivalence with homotopy inverse $I$ given by $I(z,v) = (v,-z,z)$.  The calculation 
$$
GI(z,v) \ = \ \pb{v-z}{-2z} \ = \ \pb zv \ = \ P(z,v)\,,
$$
shows that $GI=P$, and so $G\simeq PH$ as claimed.

Now observe that there are two natural homeomorphisms from the Stiefel manifold $\stiefel$ to $S^3\times S^2$ that arise from viewing $S^3$ as the unit quaternions and $S^2$ as the pure imaginary unit quaternions, namely $(z,v) \mapsto (z,v\bar z)$ and $(z,v)\mapsto (z,\bar zv)$.  These yield projections $\pi_\pm:\stiefel\to S^2$ given by 
$$
\pi_+(z,v) \ = \ v\bar z \qquad\text{and}\qquad \pi_-(z,v) =  \bar zv.
$$
These are just the lifts to $\stiefel$ of the previously defined projections $\pi_\pm:\gras\to S^2$ with the same names. 
 
Define the {\itb asymmetric characteristic map} $h_L:\tor\to S^2$ to be the composition $\pi_+He_L$, that is,
$$
h_L(s,t,u) \ = \ \norm{\pr_zx-\pr_zy}\, \bar z 
$$ 
where as usual $x=x(s)$, $y=y(t)$ and $z=z(u)$ parametrize the components of $L$, and the square brackets indicate unit normalization.     

%%%%%%%%%
\begin{proposition}
\label{prop:asymmetric}
\textbf{\boldmath The two versions $g_L$ and $h_L$ of the characteristic map of a three-component link $L$ in $S^3$ are homotopic.  Furthermore, these maps are independent, up to homotopy, of the choice of projections $\pi_+$ or $\pi_-$ used in their definitions.}
\end{proposition}
%%%%%%%%%

\begin{proof}
By definition $g_L = \pi_+Ge_L$ and $h_L = \pi_+He_L$, shown in the diagram below as the maps from left to right across the bottom and top, respectively.
\vskip-.1in
\begin{diagram}[size=1.5em]
&&&& \quad \stiefel \quad && \\
&&& \ruTo[leftshortfall=10pt]^H \ && \rdTo[leftshortfall=10pt]^{\pi_+} & \\
\tor\ & \rTo[l>=.5in]_{e_L} & \ \conf \ \ \  && \dTo_P && \ S^2\\
&&& \rdTo[leftshortfall=10pt]_G && \ \ruTo[leftshortfall=7pt]_{\pi_+}[rightshortfall=10pt] & \\
&&&& \quad \gras \quad && \\
\end{diagram}
\vskip.1in
\noindent Here $e_L(s,t,u)=(x(s),y(t),z(u))$ records the parametrization of the link, and $G$ and $H$ are the Grassmann and Stiefel maps with their associated projections $\pi_+$.  It was shown above that the left triangle in the diagram commutes up to homotopy, and the right triangle commutes on the nose.  Therefore $h_L$ and $g_L$ are homotopic.

%\quad\underset{e_L}\longrightarrow\quad

Now the same argument shows that the maps $h_L' = \pi_-He_L$ and $g_L' =  \pi_-Ge_L$ are homotopic.   Hence to complete the proof, it suffices to show that $h_L$ and $h_L'$ are homotopic.  To do so, view $\stiefel$ as the unit tangent bundle of $S^3$, with projection $p$ to the base $S^3$ given by $p(z,v) = z$.  Then the composition $pHe_L$ is null-homotopic, since it maps onto the third component of $L$, and so the map $He_L$ is homotopic to a map into any $S^2$-fiber of the bundle $\stiefel\to S^3$.  We choose the fiber over $z=1$, where the two projections $\pi_+(z,v) = v\zbar$ and $\pi_-(z,v) = \zbar v$ coincide.  It follows that $h_L$ and $h_L'$ are homotopic. 
\end{proof}

%%%%%%%%%%%%%%%%%%%%%%
\part{A Gaussian view of the asymmetric characteristic map}
%%%%%%%%%%%%%%%%%%%%%%

The formula above for $h_L$ involves first normalizing a vector in $z^\perp$, and then multiplying by $\zbar$ to move it to the unit sphere $S^2$ in the pure imaginary quaternion $3$-space $\br^3$.  We would like to express this directly as the normalization of a vector in $\br^3$.  

A geometric argument shows that stereographic projection $\pr_a$ is given by 
$$
\pr_ab \ = \ (\Im\,b\bar a)a/(1-\re\,b\bar a)
$$
and it follows that $(\pr_ab)c  = \pr_{ac}\,bc$ for any three unit quaternions $a$, $b$ and $c$ with $a\ne b$.  Hence the formula $h_L(s,t,u) = \norm{\pr_zx-\pr_zy}\, \bar z$ can be rewritten as
$$
\begin{aligned}
h_L(s,t,u) \ &= \ \norm{\pr_1\,x\zbar - \pr_1\,y\zbar} \\
&= \ \norm{\pr_{-1}(-y\zbar) - \pr_{-1}(-x\zbar)}
\end{aligned}
$$
where, as usual, $x=x(s)$, $y=y(t)$ and $z=z(u)$ parametrize $L$, and the square brackets indicate unit normalization in $\br^3$.  We favor the last expression because stereographic projection from $-1$ is orientation-preserving, while from $1$ it is orientation-reversing.  

Now for any two distinct unit quaternions $a$ and $z$, introduce the abbreviation
$$
a_z \ = \ \pr_{-1}(-a\zbar) \ \in \ \br^3\,,
$$
and so in particular $a_{-1} = \pr_{-1}\,a$.  Then we can write
$$
h_L(s,t,u) \ = \ \norm{y_z-x_z}\,.
$$
For fixed $u$, this is just the classical Gauss map for the two-component link $X_z\cup Y_z \subset \br^3$ that is the image of the first two components $X\cup Y$ of $L$ under the map $a\mapsto a_z$.  
Thus the asymmetric characteristic map $h_L$ can be viewed as a one-parameter family of Gauss maps for images in $3$-space of $X\cup Y$.  The third component $Z$ provides the parameter and determines the axes about which these images are gradually twisted.

In the next section, we will explain this perspective more carefully.  But we can see right now that it yields an easy proof of the first part of \thmref{A}, equating the pairwise linking numbers of the link $L$ with the degrees of the restriction of its characteristic map $g_L$ to the coordinate $2$-tori.

%%%%%%%%%%%%%%%%%%%%%%
%%%%%%%%%%%%%%%%%%%%%%
\part{Proof of the first statement in \thmref{A}}
%%%%%%%%%%%%%%%%%%%%%%
%%%%%%%%%%%%%%%%%%%%%%

It can be arranged by an isotopy of $L$ that $z(0) = -1$.  Then the restriction of $h_L$ to the coordinate $2$-torus $S^1\times S^1\times 0$ is precisely the Gauss map of the stereographic image $X_{-1}\cup Y_{-1}$ of $X\cup Y$, whose degree is equal to the linking number $\lk(X, Y )$ since $\pr_{-1}$ is orientation-preserving. Since $g_L$ is homotopic to $h_L$, the same is true for $g_L$.  But then it follows from the symmetry of $g_L$ that $\lk(X,Z)$ and $\lk(Y,Z)$ are given by the degrees of $g_L$ on $S^1\times 0\times S^1$ and on $0 \times S^1\times S^1$, respectively. 

The proof of the second statement in \thmref{A} -- which relates the triple linking number of $L$ to the Pontryagin $\nu$-invariant of its characteristic map -- is more delicate.  It will occupy us for the next two sections.

%% file: 5Aprep.tex
%%%%%%%%%%%%%%%%%%%%%%%%%%%%%%%%%%%%%%%
%%   SECTION 5: The Pontryagin $\nu$-invariant of the char map    %%%%%%
%%%%%%%%%%%%%%%%%%%%%%%%%%%%%%%%%%%%%%%

\section{The Pontryagin $\nu$-invariant of the characteristic map} 
\label{sec:Aprep}

Fix  a link $L$ in $S^3$ with three components $X$, $Y$ and $Z$ parametrized by $x=x(s)$, $y=y(t)$ and $z=z(u)$.  Recalling that $S^3$ is regarded as the group of unit quaternions, the asymmetric characteristic map $h_L:\tor\to S^2$ is given by
$$
h_L(s,t,u) \ = \ (y_z-x_z)/|y_z-x_z|
$$
where $a_z$ is an abbreviation for the vector $\pr_{-1}(-a\zbar)$ in pure imaginary quaternion $3\text{-space}$ $\br^3$.   To carry out the proof of \thmref{A}, we need a way to compute the absolute Pontryagin $\nu$-invariant  of $h_L$.  A procedure for doing so is described here.

Throughout this section and the next, $\br^3$ is pictured in the usual way, with the $ij$-plane horizontal and the $k$-axis pointing straight up as in \figref{quat}.  In particular, we view $k$ as the north pole of the unit sphere $S^2$.  

%%% FIGURE 15: ijk-space %%%
\begin{figure}[h!]
\includegraphics[height=150pt]{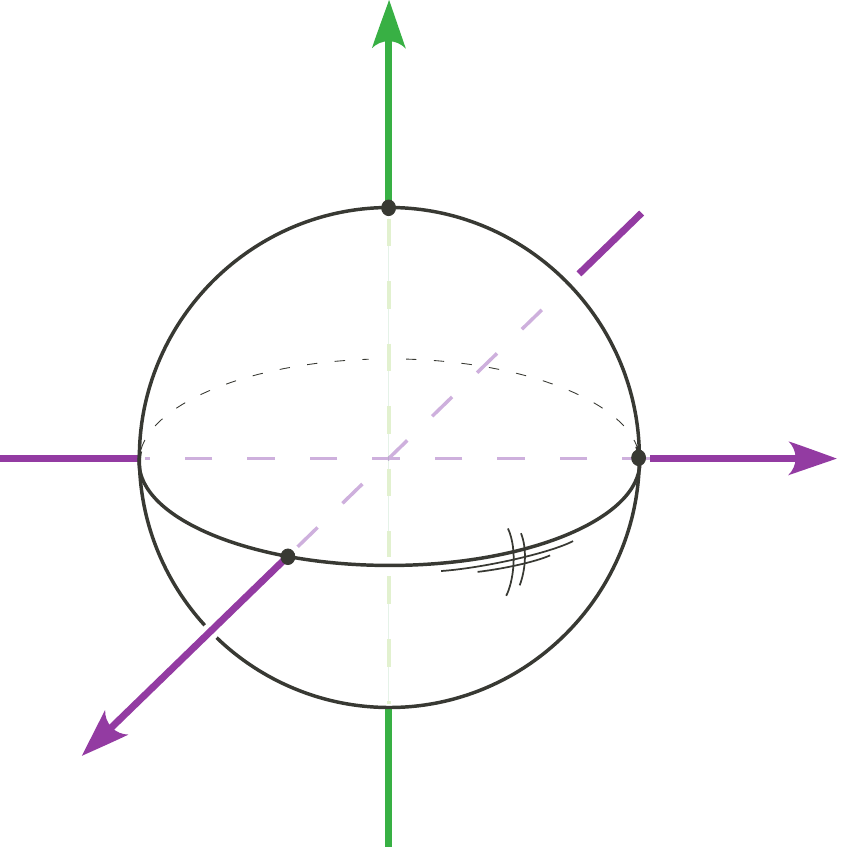}
\put(-99,40){\small$i$}
\put(-33,58){\small$j$}
\put(-89,117){\small$k$}
\caption{Pure imaginary quaternion $3$-space}
\label{fig:quat}
\end{figure}
%%%%%%%%%%%%%%%%%%%%

%%%%%%%%%%%%%%%%%%%%%%%%%%%
%%%%%%%%%%%%%%%%%%%%%%%%%%%
\part{Outline of the procedure for computing $\nu(h_L)$}
%%%%%%%%%%%%%%%%%%%%%%%%%%%
%%%%%%%%%%%%%%%%%%%%%%%%%%%

First, we will manipulate $L$ by a link homotopy into a favorable position -- referred to as {\itb generic} below -- so that, in particular, the north pole $k$ of $S^2$ is a regular value of $h_L$.  

Then we will determine the associated {\itb Pontryagin link}  
$$
\bl \ = \ h_L^{-1}(k) \ \subset \ \tor
$$
by giving a simple method for constructing a toral diagram for it, in the sense of \secref{pontryagin}.  Roughly speaking, our approach is as follows.  

By definition of $h_L$, the Pontryagin link $\bl$ consists of all $(s,t,u)\in\tor$ for which the vector in $\br^3$ from $x_z$ to $y_z$ points straight up, where $x=x(s)$, $y=y(t)$ and $z=z(u)$.   The genericity of $L$ will imply that for some points $(s,t)$ in the $2$-torus $\torh$ there is a {\it unique} $u = u(s,t)\in S^1$ that will make this happen, while for all other points $(s,t)$, {\it no} $u$ will work.  Furthermore, the set $\cald$ of all points of the first kind, called {\itb isogonal points} for reasons explained below, is a smooth $1$-dimensional  submanifold of $T^2$ whose components we call {\itb icycles}.  

Thus 
$$
\bl \ = \ \{(s,t,u)\in\tor \st (s,t)\in\cald \text{ and } u=u(s,t)\}\,,
$$
the graph of the function $u(s,t)$ over the collection $\cald$ of icycles in $T^2$.   When suitably oriented and decorated with framing and vertical winding numbers, $\cald$ will be the desired toral diagram of $\bl$.

In particular, the icycles in $\cald$ correspond to certain oriented cycles of vectors directed from $X$ to $Y$, which we call {\itb bicycles}, that are easily spotted from a picture of $L$.  Each bicycle has a {\itb longitudinal} and {\itb meridional degree}, recording how much it turns and spins relative to the standard open book structure on $S^3$.  The framing and vertical winding number of the corresponding icycle are determined by these degrees. 

Therefore a diagram for $\bl$ can be constructed once we identify the bicycles in $L$.  With this diagram in hand, the methods of \secref{pontryagin} can then be used to compute $\nu(h_L)$.

We now give the details of this procedure.  There are three 
geometries involved: spherical geometry in $S^3$,  where the link  $L$ lives, Euclidean geometry in $\br^3$,  the setting for the asymmetric characteristic map, and hyperbolic geometry in the complex upper half plane $\bh$,  which turns out to be for us the natural geometry on the pages of the ``standard" open book decomposition of $\br^3$.  We begin with an explicit construction of this open book, which provides a framework for the discussion that follows.  

\begin{comment}
Our aim is first to give a working description of the procedure, then to give some examples, and to end with the proofs.
\end{comment}

%\break

%%%%%%%%%%%%%%%%%%%%%%%%%%%
%%%%%%%%%%%%%%%%%%%%%%%%%%%
\part{The standard open books in $\br^3$ and $S^3$}
%%%%%%%%%%%%%%%%%%%%%%%%%%%
%%%%%%%%%%%%%%%%%%%%%%%%%%%

Consider the great circle $K$ in $S^3$ through $1$ and $k$, and the orthogonal great circle $C$ through $i$ and $j$.  Orient both circles by left complex multiplication by $K$ (i.e.\ from $1$ toward $k$ on $K$, and $i$ toward $j$ on $C$) so that their linking number is $+1$.

Stereographic projection from $-1$ carries $K$ 
% (punctured at $-1$) 
onto the $k$-axis in $\br^3$, and fixes $C$, which now appears as the unit circle in the $ij$-plane.  The complement $\V$ of the $k$-axis in $\br^3$ is naturally identified with the product of a circle $S^1$ (viewed as the quotient $\br/2\pi\bz$) with the complex upper half plane $\bh$ (which for later purposes will be viewed as the hyperbolic plane) via the diffeomorphism
% \V =  \br^3-k\text{-axis}
$$
\V \ \longrightarrow \ S^1 \times \bh
$$ 
with coordinates \,$\ell: \V\longrightarrow S^1$\, and \,$m: \V\longrightarrow\bh$\, given by 
$$
\ell(q) \ = \ \arg(q_1+q_2i)\qquad\text{and}\qquad m(q) \ = \ q_3 + |q_1+q_2i|i\,.
$$
for $q=q_1i+q_2j+q_3k\in \V =  \br^3-k\text{-axis}$.  In other words, if $q=(r,\theta,z)$ in cylindrical coordinates, then $\ell(q) = \theta$ and $m(q) = z+ri$.  We call $\ell$ and $m$ the {\itb longitudinal} and {\itb meridional projections} in $\br^3$, and refer to $\ell(q)$ as the {\itb polar angle} of $q$.  
% Two points with the same polar angle will be called {\itb isogonal}.  

The longitudinal projection defines the {\itb standard open book} in $\br^3$, with binding the $k$-axis, and with pages $P_\theta = \ell^{-1}(\theta)$ for $\theta\in S^1$.  The pages are just the oriented vertical half-planes bounded by the $k$-axis, each indexed by its polar angle as shown in \figref{book}.  The meridional projection serves to identify each page $P_\theta$ with the hyperbolic plane $\bh$, with ``center" \,$i_\theta = i\,\cos\theta + j\,\sin\theta$\, corresponding to $i$.  The union of all the page centers is the unit circle $C$.

%%% FIGURE 16: ijk-space %%%
\begin{figure}[h!]
\includegraphics[height=170pt]{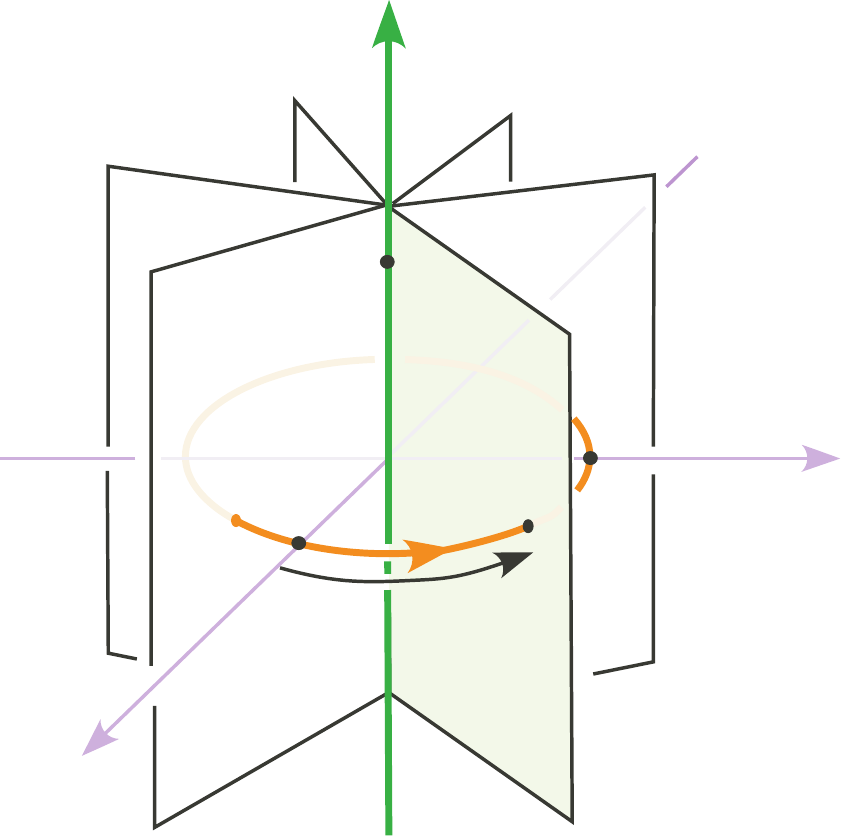}
\put(-113,63){\small$i$}
\put(-64,66){\small$i_\theta$}
\put(-49,82){\small$j$}
\put(-89,114){\small$k$}
\put(-90,43){\small$\theta$}
\put(-107,-10){binding}
\put(-145,142){pages}
\put(-53,10){$P_\theta$}
\put(-133,62){$C$}
\caption{The standard open book in $\br^3$}
\label{fig:book}
\end{figure}
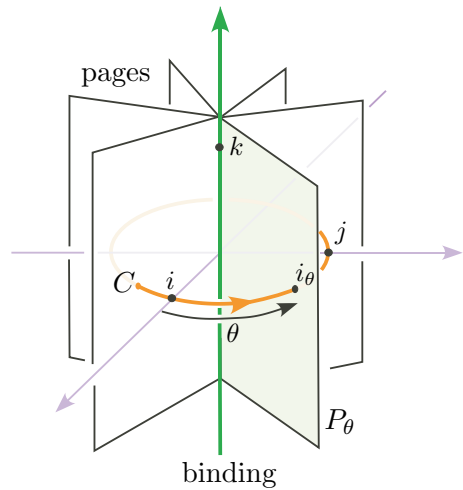
%%%%%%%%%%%%%%%%%%%%

\break

The longitudinal and meridional projections  in $\br^3$ lift to $S^3$ \label{longproj} (in the complement of $K$) by composing with stereographic projection from $-1$.  Relying on the context, we continue to denote them by $\ell$ and $m$, and to refer to $\ell(q)$ as the polar angle of $q$.  Explicitly
$$
\ell(q) \ = \ \arg(q_1+q_2i)\qquad\text{and}\qquad m(q) \ = \ \frac{q_3 + |q_1+q_2i|i}{1+q_0}\,.
$$
for $q = q_0+q_1i+q_2j+q_3k\in S^3-K$.  Here we are using 
the formula 
$$
\pr_{-1}q \ =  \ (\Im q) / (1 + \re q)  =  (q_1 i + q_2 j + q_3 k) / (1 + q_0)\,,
$$
and then plugging this into the formula above for the longitudinal and meridional projections in $\br^3$.  

Just as in $\br^3$, the longitudinal projection in $S^3$ defines the {\itb standard open book} in $S^3$, with binding $K$, and with pages $H_\theta = \ell^{-1}(\theta)$.  The pages are now open great hemispheres  in $S^3$ with centers (i.e.\ poles) $i_\theta$ along the great circle $C$, and with $K$ as equator.  By design, $\pr_{-1}$ maps each hemispherical page $H_\theta$ in $S^3$ onto the corresponding half-planar page $P_\theta$ in $\br^3$.

We now turn our attention back to links in the $3$-sphere.

%%%%%%%%%%%%%%%%%%%%%%%%%%%
%%%%%%%%%%%%%%%%%%%%%%%%%%%
\part{Generic links}
%%%%%%%%%%%%%%%%%%%%%%%%%%%
%%%%%%%%%%%%%%%%%%%%%%%%%%%

A three-component link $L=X\,\cup\,Y\,\cup\,Z$ in $S^3$ is {\itb generic} if

\vspace{.02in}
\begin{indentation}{.05em}{1em}
\begin{enumerate}
\item $Z$ coincides with the oriented binding $K$ of the standard open book, and
\smallskip
\item $X$ and $Y$ wind ``generically" around $Z$.
\end{enumerate}  
\end{indentation}
\vspace{.02in}

More precisely, the second condition requires the restriction to $X\,\cup\,Y$ of the longitudinal projection \,$\ell:S^3-K\longrightarrow S^1$, which sends each point to its polar angle, to be a Morse function with just one critical point per critical value.  Geometrically, this means $X$ and $Y$ are transverse to the pages of the standard open book in $S^3$, except for finitely many pages where exactly one of them turns around at a single point.   These will be called the {\itb critical points} of $L$, while all other points on $X\cup Y$ will be called {\itb regular points}.   Each regular point $w$ has a sign, denoted {\bf\boldmath $\sgn(w)$}, when viewed as an intersection point of $L$ with the page containing $w$.  Thus $\sgn(w) = +1$ or $-1$ according to whether $L$ is oriented in the direction of increasing or decreasing polar angle near $w$.  

An example of a generic link is shown in \figref{generic}.  It has four critical points, two on $X$ and two on $Y$, indicated by dots in the picture.

%%% FIGURE 17: Generic Link %%%
\begin{figure}[h!]
\includegraphics[height=150pt]{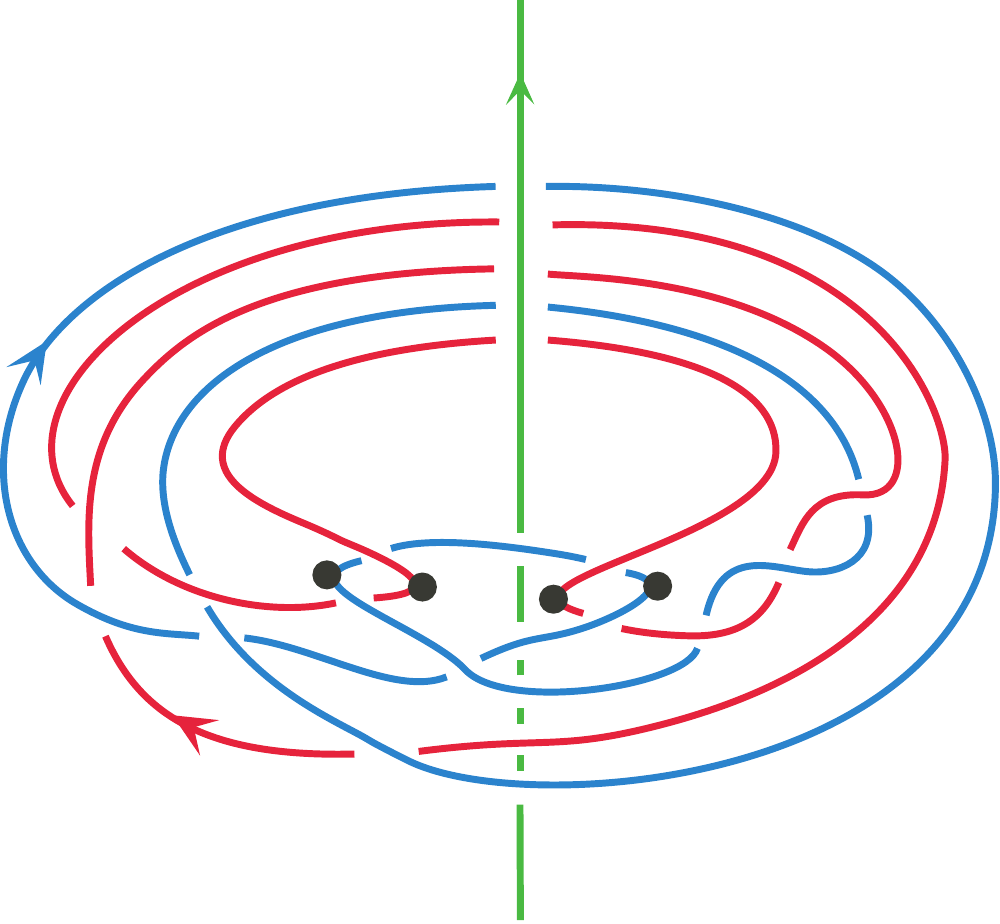}
\put(-150,110){$X$}
\put(-152,30){$Y$}
\put(-90,138){$Z$}
\caption{A generic link}
\label{fig:generic}
\end{figure}
%%%%%%%%%%%%%%%%%%%%

Any three-component link in $S^3$ is evidently link homotopic to a generic one:  first unknot the last component by a link homotopy and move it to coincide with $K$, and then adjust the first two components by a small isotopy to satisfy the genericity condition (2). 

\smallskip

{\it From this point on we assume that $L$ is generic without further mention.}

\break

As a consequence, we will show the following:

\vspace{.02in}
\begin{indentation}{.05em}{1em}
\begin{enumerate}
\item[(a)] The north pole $k\in S^2$ is a regular value of the characteristic map $h_L$.
\smallskip
\item[(b)] There is a simple method for constructing a toral diagram for the associated Pontryagin link $\bl = h_L^{-1}(k)$ from a picture of $L$.
\end{enumerate}  
\end{indentation}
\vspace{.02in} 

This is the content of the ``bicycle theorem" below.  To state it precisely, we need to introduce the key notion of a {\it bicycle} in $L$, and its associated {\it icycle} in $T^2$.

%%%%%%%%%%%%%%%%%%%%%%%%%%%
%%%%%%%%%%%%%%%%%%%%%%%%%%%
\part{Bicycles and icycles}
%%%%%%%%%%%%%%%%%%%%%%%%%%%
%%%%%%%%%%%%%%%%%%%%%%%%%%%

Assume, as always, that the components $X$, $Y$ and $Z$ of $L$ are parametrized by smooth functions $x=x(s)$, $y=y(t)$ and $z=z(u)$ with nowhere vanishing derivatives.   In particular, points $(s,t)$ in the $2$-torus parametrize pairs of points $x\in X$ and $y\in Y$.  Suppose that $x$ and $y$ have the same polar angle $\theta$, or equivalently that they lie in a common page $H_\theta$ of the standard open book in $S^3$.  Then we call $(s,t)$ an {\itb isogonal point} in $T^2$, and call $(x,y)$ a {\itb page vector} in $L$ with polar angle $\theta$, reflecting the fact that the vector in $\br^3$ from $\pr_{-1}(x)$ to $\pr_{-1}(y)$ lies entirely on the half-planar page $P_\theta$.  

A page vector $(x,y)$ will be called {\itb critical} if $x$ or $y$ is a critical point of $L$, and {\itb regular} if both $x$ and $y$ are regular.  A regular page vector is {\itb positive} if the oriented strands of $L$ through $x$ and $y$ point in the same direction, meaning $\sgn(x) = \sgn(y)$, and is {\itb negative} if they point in opposite directions.  These notions are illustrated in \figref{pagevectors}, in which the vectors labeled $1$ and $2$ are positive regular page vectors, $3$ is negative regular, and $4$ is critical.

%%% FIGURE 18: Page vectors %%%
\begin{figure}[h!]
\includegraphics[height=170pt]{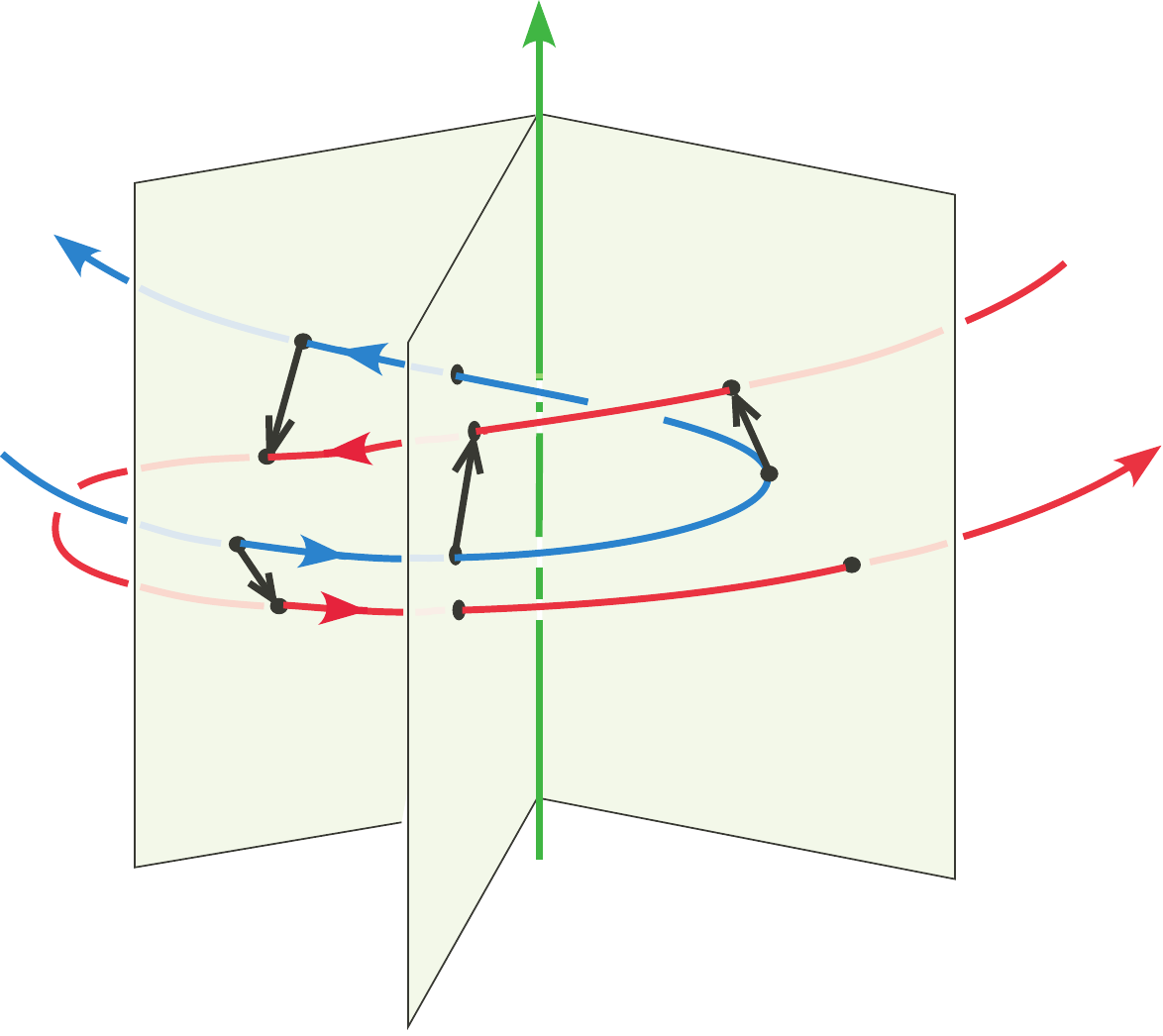}
%\put(-94,67){\small$\vec_1$}
%\put(-93,30){\small$\vec_2$}
%\put(-65,45){\small$\vec_3$}
%\put(-40,60){\small$\vec_4$}
\put(-150,105){\tiny$1$}
\put(-156,73){\tiny$2$}
\put(-112,85){\tiny$3$}
\put(-65,98){\tiny$4$}
\put(-185,105){$X$}
\put(-20,102){$Y$}
\put(-106,18){$Z$}
\caption{Page vectors}
\label{fig:pagevectors}
\end{figure}
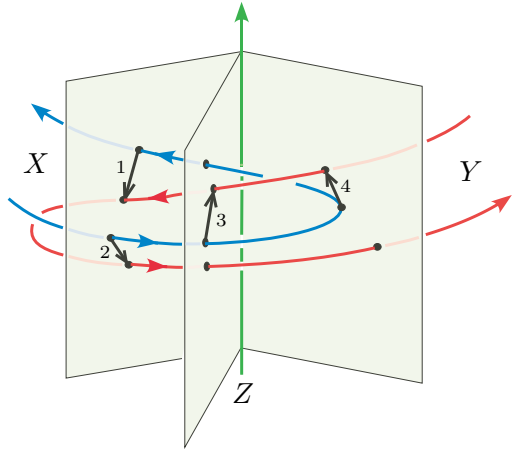
%%%%%%%%%%%%%%%%%%%%

Now consider the spaces
$$
\cald \ = \ \{\text{isogonal points in $T^2$}\} \qquad\text{and}\qquad \calp \ = \ \{\text{page vectors in $L$}\} \,.
$$
By definition $\cald$ parametrizes $\calp$.  The genericity of $L$ implies that $\calp$ consists of a finite number of disjoint cycles of page vectors, and that $\cald$ consists of a finite collection of smooth simple closed curves in $T^2$.   Orient $\calp$ so that it points to the right (meaning in the direction of increasing polar angle) at each positive regular page vector in it, and to the left at each negative regular page vector.  This gives a well-defined orientation on $\calp$, inducing one on $\cald$ as well.  We call these the {\itb preferred orientations} on $\calp$ and $\cald$.     

\begin{definition}
A {\itb bicycle} (or ``bi-cycle")  in $L$ is a component $\calp_i$ of $\calp$, that is, an oriented cycle of page vectors.  Each bicycle is parametrized by a component $\cald_i$ of $\cald$, which we call its associated {\itb icycle} (or ``i-cycle").
\end{definition}

%%%%%%%%%%%%%%%%%%%%%%%%%%%
%%%%%%%%%%%%%%%%%%%%%%%%%%%
\part{Some examples of bicycles and their associated icycles}
%%%%%%%%%%%%%%%%%%%%%%%%%%%
%%%%%%%%%%%%%%%%%%%%%%%%%%%

We first draw in \figref{local} four local pictures of a bicycle near a regular page vector, and below them, their parametrizing icycles.  It is understood that these pictures take place somewhere in front of the upward pointing $Z$ axis. The four cases represent the possible directions of  $X$ and $Y$ relative to the page containing the vector.  In each case, the orientation of the bicycle is indicated by a squiggly arrow.
%%% FIGURE 19: Local Bicycle %%%
\begin{figure}[h!]
\includegraphics[height=175pt]{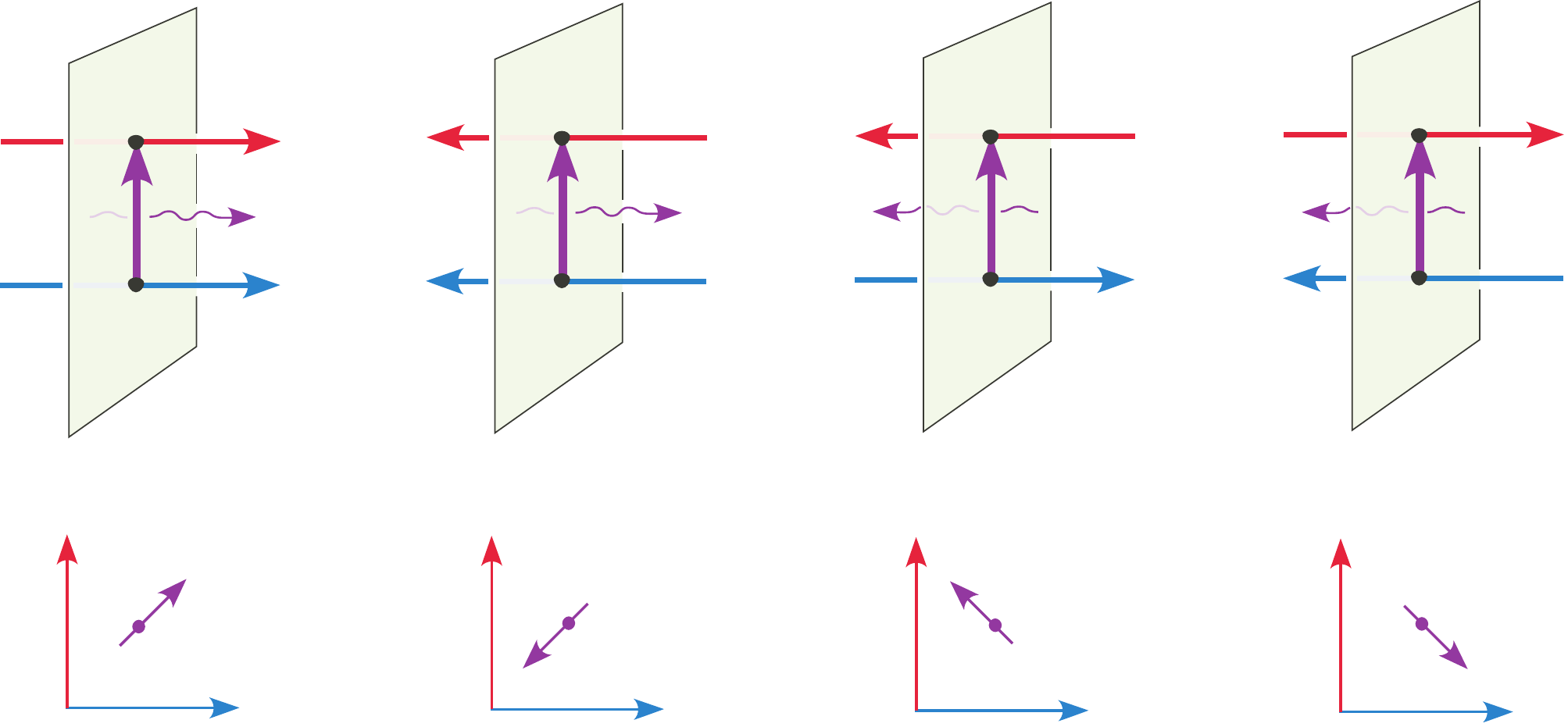}
\put(-374,146){\small$Y$}
\put(-376,96){\small$X$}
%\put(-347,147){\small$y$}
%\put(-348,98){\small$x$}
\put(-215,147){\small$Y$}
\put(-217,96){\small$X$}
%\put(-245,148){\small$y$}
%\put(-245,99){\small$x$}
\put(-115,147){\small$Y$}
\put(-171,96){\small$X$}
%\put(-143,148){\small$y$}
%\put(-142,99){\small$x$}
\put(-65,147){\small$Y$}
\put(-13,96){\small$X$}
%\put(-38,148){\small$y$}
%\put(-39,99){\small$x$}
\put(-318,3){\scriptsize$s$}
\put(-363,50){\scriptsize $t$}
\put(-216,3){\scriptsize$s$}
\put(-261,50){\scriptsize $t$}
\put(-114,2){\scriptsize$s$}
\put(-159,49){\scriptsize $t$}
\put(-12,2){\scriptsize$s$}
\put(-57,49){\scriptsize $t$}
\caption{Local pictures of a bicycle and its corresponding icycle}
\label{fig:local}
\end{figure}
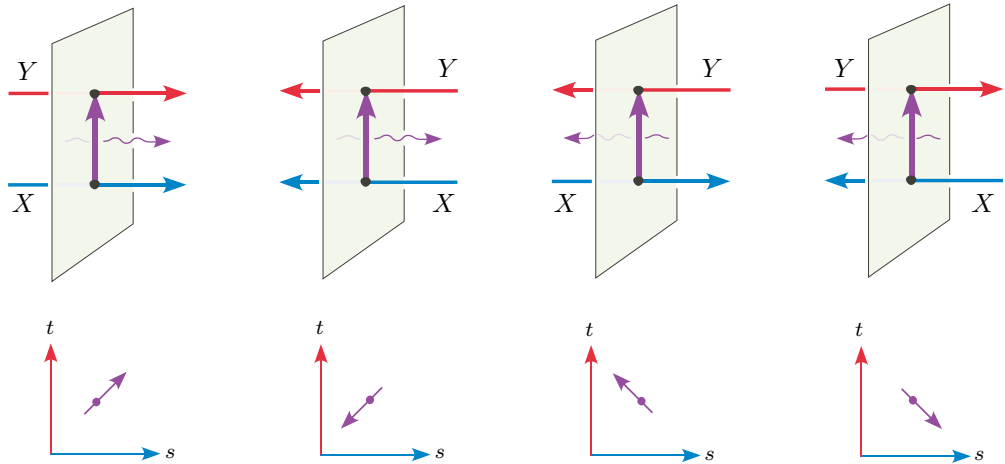
%%%%%%%%%%%%%%%%%%%% 

For an example of a full bicycle, consider the ``clasp" between $X$ and $Y$ pictured in \figref{claspbike}(a).  This gives rise to the bicycle in \figref{claspbike}(b), passing successively through the vectors labeled $1,2,3,4$ and then back to $1$.  The associated icycle is a counterclockwise circle in $T^2$, shown in \figref{claspbike}(c).  The route taken by this bicycle is ``short" in the sense that it does not wind around the binding, although it does spin within the pages.

%%% FIGURE 20: Clasp/Bicycle %%%
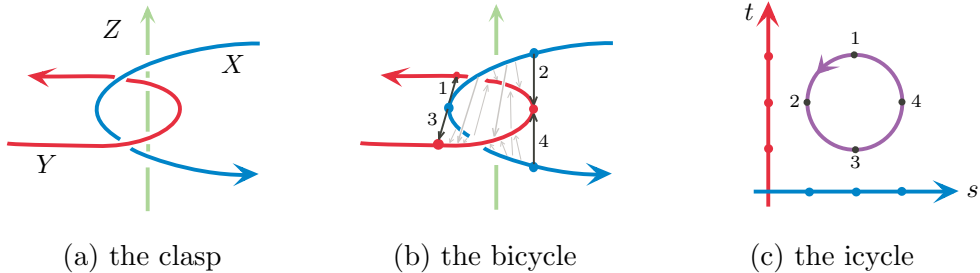
\begin{figure}[h!]
\includegraphics[height=80pt]{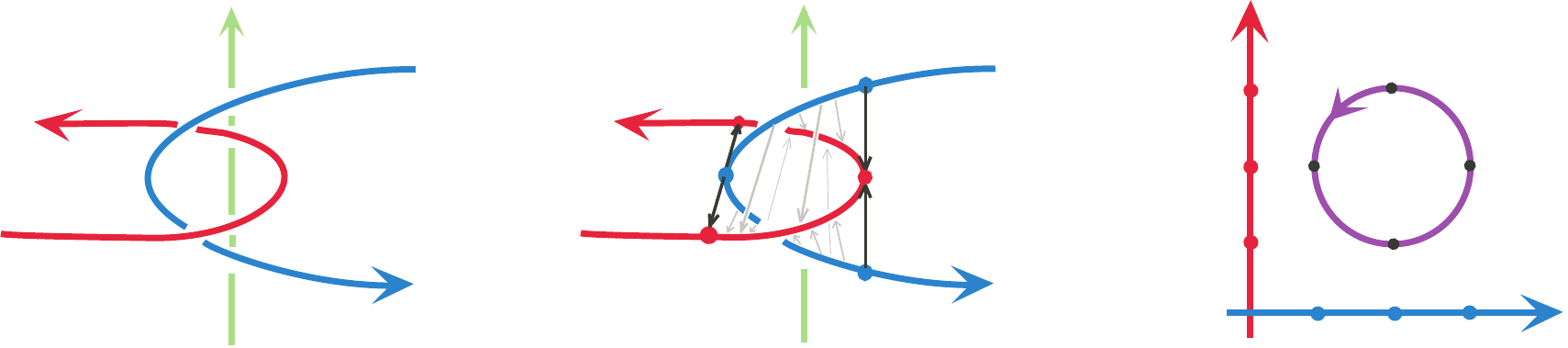}
\put(-280,53){\small$X$}
\put(-350,15){\small$Y$}
\put(-325,66){\small$Z$}
\put(-198,45){\tiny$1$}
\put(-160,51){\tiny$2$}
\put(-202,33){\tiny$3$}
\put(-160,25){\tiny$4$}
\put(2,6){\small$s$}
\put(-82,73){\small$t$}
\put(-42,64){\tiny$1$}
\put(-65,40){\tiny$2$}
\put(-42,17){\tiny$3$}
\put(-19,40){\tiny$4$}
\put(-340,-20){(a) the clasp}
\put(-215,-20){(b) the bicycle}
\put(-80,-20){(c) the icycle}
\caption{Bicycle arising from a clasp}
\label{fig:claspbike}
\end{figure}
%%%%%%%%%%%%%%%%%%%% 

As another example, the link shown in \figref{generic} and reproduced in \figref{icycles}(a) below has three bicycles.  Two of them are short, arising from the clasps as in the previous example, while the remaining long one oscillates back and forth in the longitudinal direction, eventually making one full revolution around the binding.  It is an instructive exercise left to the reader to recover the plot of the associated icycles in \figref{icycles}(b), in which the trivial circles labeled $A$ and $B$ correspond to the clasps in $L$ with the same labels, $C$ labels the icycle that parametrizes the long bicycle, and the corners of the square parametrize the pair $(x,y)$ indicated by the dots in \figref{icycles}(a).   

%%% FIGURE 21: Icycles %%%
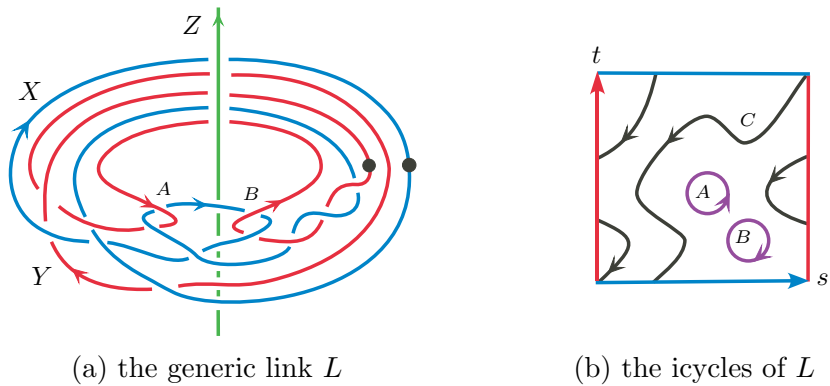
\begin{figure}[h!]
\includegraphics[height=125pt]{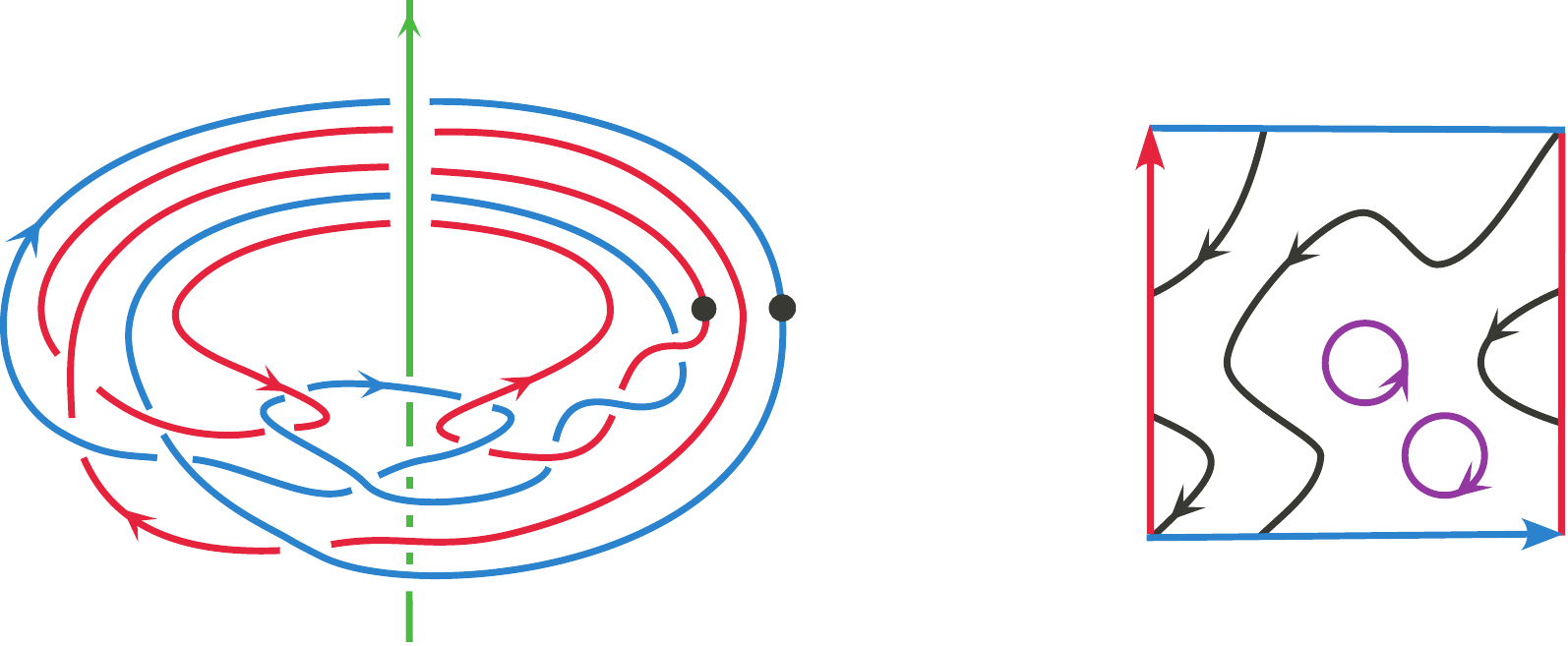}
\put(2,20){\small$s$}
\put(-83,104){\small$t$}
\put(-300,90){\small$X$}
\put(-295,20){\small$Y$}
\put(-238,115){\small$Z$}
\put(-248,54){\tiny$A$}
\put(-215,52){\tiny$B$}
\put(-44,53){\tiny$A$}
\put(-29,36){\tiny$B$}
\put(-27,80){\tiny$C$}
\put(-280,-15){(a) the generic link $L$}
\put(-90,-15){(b) the icycles of $L$}
\caption{A generic link and its icycles}
\label{fig:icycles}
\end{figure}
%%%%%%%%%%%%%%%%%%%%

Before proceeding, we remind the reader that our interest in icycles associated to $L$ stems from the fact that -- when suitably decorated -- they give a toral diagram for a Pontryagin link of the characteristic map $h_L$.   This is the content of the bicycle theorem.  
% To state it, we must first identify the characteristics of bicycles that determine these decorations.

%%%%%%%%%%%%%%%%%%%%%%%%%%%
%%%%%%%%%%%%%%%%%%%%%%%%%%%
\part{The Bicycle Theorem}
%%%%%%%%%%%%%%%%%%%%%%%%%%%
%%%%%%%%%%%%%%%%%%%%%%%%%%%

The longitudinal and meridional projections on $S^3-K$, defined earlier,
% on page \pageref{longproj}
induce projections by the same name on the space $\calp$ of page vectors,
$$
S^1 \, \overset{\ell}{\longleftarrow} \, \calp \, \overset{m}{\longrightarrow} \, \bh \,,
$$
given by $\ell(x,y) = \ell(x) = \ell(y)$ and $m(x,y) = \arg(m(y)-m(x))$.  In other words $\ell(x,y)$ is the polar angle that parametrizes the common hemispherical page in $S^3$ containing $x$ and $y$, and $m(x,y)$ is the argument of the vector from $m(x)$ to $m(y)$ in 
% the upper-half plane 
$\bh$.  

Using these projections, we define the {\itb longitudinal} and {\itb meridional degrees} of a bicycle $\calp_i$ in $L$ by  
$$
\ell_i \ = \ \deg(\ell|\calp_i) \qquad\text{and}\qquad  m_i \ = \ \deg(m|\calp_i). \label{degrees}
$$
These integer invariants record, respectively, the number of times $\calp_i$ travels around the binding $Z$, and the number of times its vectors spin around in the pages as it goes.   

For example, any bicycle arising from a clasp between $X$ and $Y$ has zero longitudinal degree, while its meridional degree can be $\pm1$.  In particular, the one shown in \figref{claspbike} has meridional degree $-1$, while the ones labeled $A$ and $B$ in \figref{icycles}(a) have degrees $1$ and $-1$, respectively.  The long bicycle in \figref{icycles}(a), labeled $C$, has longitudinal degree $2$ and meridional degree $1$.  

For any icycle $\cald_i$ in $T^2$, parametrizing a bicycle $\calp_i$ in $L$, define the {\itb framing} $n_i$ and {\itb vertical winding number} $r_i$ of $\cald_i$ by
$$
n_i \ = \ -\ell_i-m_i \qquad\text{and}\qquad r_i \ = \ m_i 
$$
where $\ell_i$ and $m_i$ are the longitudinal and meridional degrees of $\calp_i$.

We can now state the main result of this section.

%%%%% BICYCLE THEOREM  %%%%%%
\begin{theorem*}\label{thm:bicycle}
\textbf{\boldmath Let $L$ be a generic link in $S^3$.  Then
\begin{indentation}{-.2em}{1.2em}
\begin{itemize}
\item[(a)] The north pole $k\in S^2$ is a regular value of $h_L:\tor\to S^2$.
\smallskip  
\item[(b)] The collection $\cald$ of icycles of $L$, together with their framings and vertical winding numbers as defined above, forms a toral diagram for the associated Pontryagin link $\bl = h_L^{-1}(k)$.
\end{itemize}
\end{indentation}}
\end{theorem*}
%%%%%%%%%%%%%%%%%%

Before proving this theorem, we illustrate how it can be used to compute the Pontryagin invariant of the characteristic map of a generic link.

%%%%%%%%%%%%%%%%%%%%%%%%%%%
%%%%%%%%%%%%%%%%%%%%%%%%%%%
\part{Computing $\nu(h_L)$ for a generic link $L$ using the bicycle theorem}

As a first example, again consider the link $L$ pictured in \figref{icycles}(a).   As noted above, it has three bicyles $A$, $B$, $C$, with longitudinal degrees $0,0,2$, meridional degrees $1,-1,1$, and so by definition, framings $-1,1,-3$ and vertical winding numbers $1,-1,1$.   

By the bicycle theorem, the Pontryagin link for the characteristic map $h_L$ has toral diagram as shown in \figref{icycles}(b) with vertical winding numbers $1$, $-1$ and $1$ on the icycles $A$, $B$ and $C$, and with global framing $n=-1+1-3=-3$.  Thus the total vertical winding number is $r = 1-1+1=1$ and from the diagram we compute the horizontal winding numbers to be $p=-1$ and $q=-2$.  (These values for the winding numbers of the diagram are confirmed by the calculations $p = \lk(Y,Z)=-1$, $q = \lk(X,Z)=-2$ and $r = \lk(X,Y)=1$.)  Thus the invariant $\nu(h_L)$ is well defined modulo $2 = 2\gcd(-1,-2,1)$.

Using a base point in the lower right corner of the diagram, and straight line paths from the icycles to the base point, the depths of the icycles $A$, $B$ and $C$ are $1$, $-1$ and $-1$.  Thus by \propref{pontryagin} we conclude that 
$$
\nu(h_L) \ = \ n+pq + \textstyle\sum d_ir_i \ = \ -3+2+1  \ = \ 0 \ \in \ \bz_2.
$$

Although the purpose of this example is to illustrate how the bicycle theorem is used for computations, we note that \thmref{A} (yet to be proved) yields the same result here effortlessly, since it implies that 
the Pontryagin invariant of the characteristic map of any three-component link in $S^3$ is even.  Therefore, when the pairwise linking numbers $p$, $q$ and $r$ are relatively prime, as they are in this case, the computation is  modulo $2\gcd(p, q, r)  =  2$,  and so the Pontryagin invariant is zero.   

%%%%%%%%%%%%%%%%%%%%%%%%%%%
%%%%%%%%%%%%%%%%%%%%%%%%%%%
\part{Double crossing changes}
%%%%%%%%%%%%%%%%%%%%%%%%%%%
%%%%%%%%%%%%%%%%%%%%%%%%%%%

For our next example we analyze the effect on $\nu(h_L)$ of changing two crossings of opposite signs between the first two components of a generic link $L=X\cup Y\cup Z$.  This will be a key step in our inductive proof of \thmref{A}.

To this end, choose a positive and a negative crossing between $X$ and $Y$ in a suitable projection of $L$, and let $(x_+,y_+)$ and $(x_-,y_-)$ be the corresponding page vectors.  Changing both of the crossings yields a new link $\what L$, with the same pairwise linking numbers as $L$.  This is illustrated in \figref{dblbor} where $L$ is the Borromean rings, shown on the left, and $\what L$ is the unlink, shown on the right.  

%%% FIGURE 22: Double crossing %%%
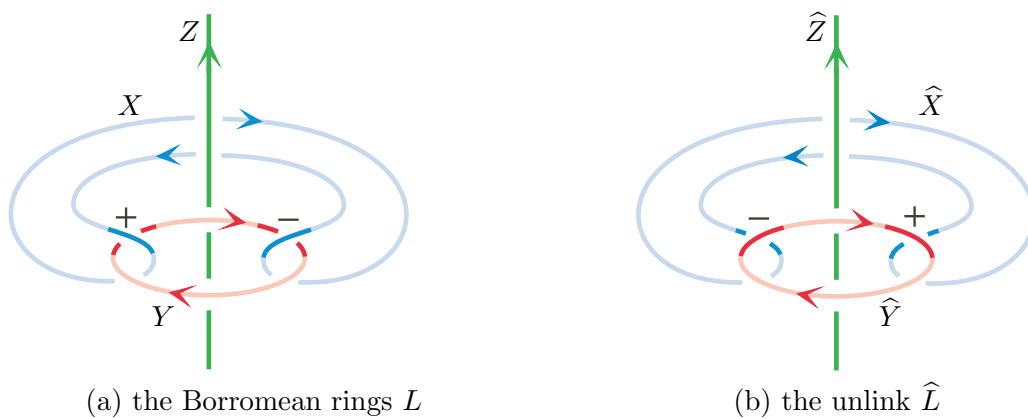
\begin{figure}[h!]
\includegraphics[height=135pt]{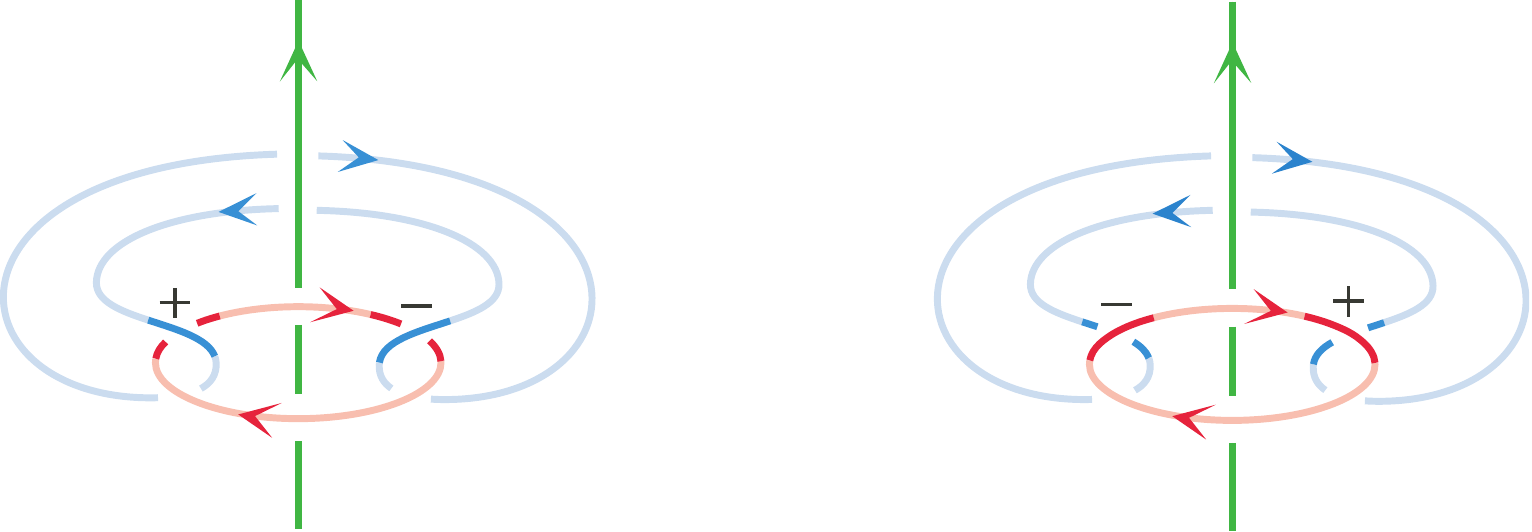}
\put(-348,97){$X$}
\put(-335,17){$Y$}
\put(-325,125){$Z$}
\put(-360,-15){(a) the Borromean rings $L$}
\put(-45,97){$\what X$}
\put(-60,17){$\what Y$}
\put(-88,125){$\what Z$}
\put(-115,-15){(b) the unlink $\what L$}
\caption{A double crossing change}
\label{fig:dblbor}
\end{figure}
%%%%%%%%%%%%%%%%%%%%

We then say that $\what L$ is obtained from $L$ by a {\itb double crossing change}, and propose to use the bicycle theorem to compute the resulting change
$$
\Delta\nu \ = \  \nu(h_{\what L}) - \nu(h_{L})
$$
in Pontryagin invariants.

As it turns out, there is a simple formula for $\Delta\nu$ involving the link $L_0$ obtained from $L$ by ``smoothing" both crossings in the usual way: 
\vskip .1pt
$$
\includegraphics[height=20pt]{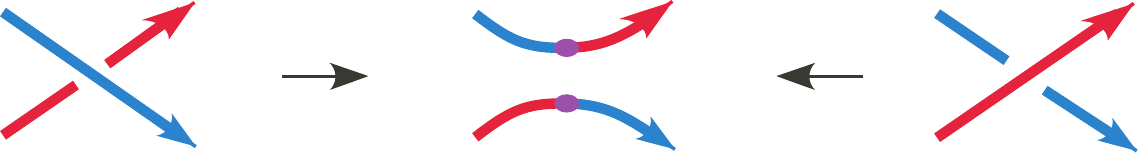}
$$
\vskip .1pt
The smoothed link $L_0$ has three components, $P$ and $Q$ (replacing $X$ and $Y$) and $Z$.  The component $P$ goes from $+$ to $-$ along $Y$, and then back from $-$ to $+$ along $X$, where we retain the $+$ and $-$ labels after smoothing, while $Q$ goes from $+$ to $-$ along $X$, and then back from $-$ to $+$ along $Y$.  This is illustrated in \figref{smooth} for the double crossing change shown in \figref{dblbor}, where we use  $X_{+-}$ to denote the arc on $X$ from $+$ to $-$\,, and so forth.  

%%% FIGURE 23: Smoothing a double crossing %%%
\begin{figure}[h!]
\includegraphics[height=135pt]{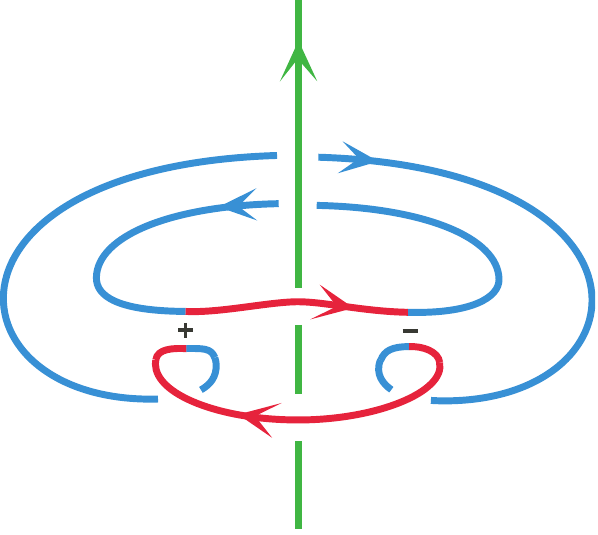}
\put(-23,60){$P$}
\put(0,50){$Q$}
\put(-90,125){$Z$}
\put(-103,99){\footnotesize$X_{+-}$}
\put(-65,74){\footnotesize$X_{-+}$}
\put(-95,61){\footnotesize$Y_{+-}$}
\put(-67,21){\footnotesize$Y_{-+}$}
%%%%%%%%%%%
\caption{The smoothed link $L_0 \ = \ P\cup Q\cup Z$}
\label{fig:smooth}
\end{figure}
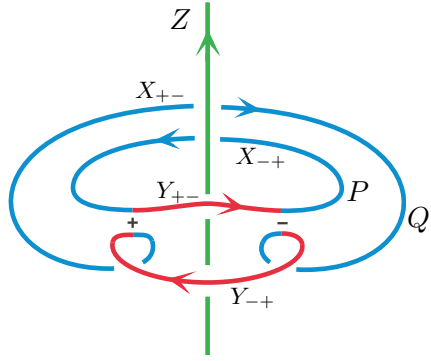
%%%%%%%%%%%%%%%%%%%%

Then we have the following consequence of the Bicycle Theorem:

\begin{corollary}\label{cor:double}
\textbf{\boldmath \textup{(Double Crossing)}
If $L = X\cup Y\cup Z$ is transformed into $\what L$ by a double crossing change, and $P$ and $Q$ are the components of the associated smoothing of $X\cup Y$, as explained above, then the corresponding Pontryagin invariants change by 
$$
\Delta\nu \ = \ 2\,\lk(P,Z) \ = \ -2\,\lk(Q,Z) \ \in \ \bz_{2\gcd(p,q,r)}
$$
where $p$, $q$ and $r$ are the pairwise linking numbers of the components of $L$.}
\end{corollary}

\begin{proof}
First note that
$$
\lk(P\cup Q,Z) \ = \  \lk(X\cup Y,Z) \ = \ p+q \ = \ 0 \ \in \ \bz_{\gcd(p,q,r)}
$$
so it suffices to establish the first equality in the corollary.

We may assume that the two page vectors  $\vv_+ = (x_+,y_+)$ and $\vv_- = (x_-,y_-)$ associated with the chosen crossings lie on distinct hemispherical pages $H_+$ and $H_-$ of the standard open book in $S^3$, and that these pages contain no critical points in the link.   These two page vectors may lie in distinct bicycles in $L$, or they may lie in the same bicycle.

Suppose first that $\vv_+$ and $\vv_-$  lie in distinct bicycles $\calp_+$ and $\calp_-$ in $L$.  When we change $L$ to $\what L$, these bicycles will change, but the icycles  $\cald_+$ and $\cald_-$ in the $2$-torus that parametrize them will stay the same.  However,  their vertical winding numbers  $r_+$ and $r_-$ (which record the meridional degrees of $\calp_+$ and $\calp_-$) and their framings $n_+$ and $n_-$ (which record the negative of the sum of the meridional and longitudinal degrees of $\calp_+$ and $\calp_-$) will change.  In particular, we claim that 
$$
\what r_+\ = \ r_ + -1 \qquad\text{and}\qquad \what r_- \ = \ r_- + 1\,,
$$
and consequently \, $\what n_+ = n_+ + 1$ \,and\, $\what n_- = n_- - 1$ \, since the longitudinal degrees of the bicycles clearly do not change.  The figure below helps us to see this.
$$
\includegraphics[height=30pt]{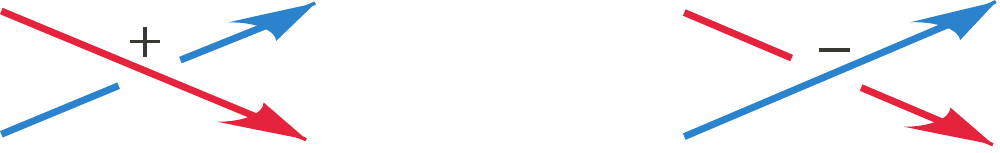}
\put(-212,0){\small$X$}
\put(-209,25){\small$Y$}
\put(-175,-1){\small$L$}
\put(-77,0){\small$\what X$}
\put(-75,25){\small$\what Y$}
\put(-40,-2){\small$\what L$}
$$
\vskip .2pt
In this figure, we start with a positive crossing in $L$ and change it to a negative crossing in $\what L$. Since the strands of $X$ and $Y$ are both pointing to the right, the bicycle $\calp_+$ is moving from left to right.  During this motion the page vectors in $\calp_+$  undergo half a {\it counter-clockwise} rotation with respect to the preferred orientation on the pages.  In the corresponding picture for $\what L$ , we see half a {\it clockwise} rotation.  Therefore the bicycle in $\what L$ has one more full clockwise rotation in the pages than in $L$, and so $\what r_+ =  r_+ -1$ as claimed.

If, for example, we switched the arrow on the strand of $Y$, we would have a negative crossing, but then the corresponding bicycle would be moving from right to left, and we would see that $\what r_- = r_- +1$ as claimed.  With this guidance, we leave the remaining cases to the reader.

Now suppose that the page vectors $\vv_+$ and $\vv_-$ lie in the same bicycle $\calp_+ = \calp_-$.  Then the above changes in vertical winding number $r$ for the corresponding icycle $\cald_+=\cald_-$ will cancel, and so we see that neither the vertical winding number nor the framing of the icycle will change, that is $\what r_\pm = r_\pm$ and $\what n_\pm = n_\pm$.

Note that in either case, whether $\vv_+$ and $\vv_-$ lie in the same or different bicycles, the {\it total framing} $n$ of the diagram (the sum of the framings of the icycles) does not change.

At this point we recall \propref{pontryagin}, which tells us that 
$$ 
\nu(h_L)  \ = \  n  +  p q  +  \textstyle\sum d_i r_i\,. 
$$
In passing from $L$  to  $\what L$ via the double crossing change, we have just seen that the total framing $n$ does not change, and the pairwise linking numbers $p$ and $q$ certainly do not change.  We have also seen that the icycles stay the same, so their depths $d_i$ do not change.  Only the winding numbers $r_+$ and $r_-$ of the (possibly equal) icycles $\cald_+$ and $\cald_-$ may change.  In particular, they also do not change when the page vectors
$$
\vv_+ \ = \ (x(s_+), y(t_+)) \qquad\text{and}\qquad \vv_- \ = \ (x(s_-), y(t_-))
$$ 
lie in the same bicycle, and so $\Delta\nu = 0$ in this case, while they change to $\what r_+ = r_+ - 1$ and $\what r_- = r_- + 1$ when $\vv_+$ and $\vv_-$ lie in distinct bicycles, in which case we have
$$
\begin{aligned}
\Delta\nu \ &= \  \nu(h_{\what L}) - \nu(h_L) \\
\ &= \ d_-(\what r_- - r_-) + d_+(\what r_+ - r_+) \\
\ &= \ d_- - d_+\,.
\end{aligned}
$$
Thus in either case it remains to prove that $d_- - d_+ = 2\,\lk(P,Z)$. 

To show this, we must compute the depths $d_+$ and $d_-$ of $\cald_+$ and $\cald_-$.  This requires a choice of base point, and then a choice of paths $\gamma_+$ and  $\gamma_-$ from $\cald_+$ and $\cald_-$ to this base point.  We let $(s_+,t_-)$ be the base point, and use the vertical path $\gamma_+$ from $(s_+,t_+)$ to $(s_+,t_-)$ and the horizontal path $\gamma_-$ from $(s_-,t_-)$ to $(s_+,t_-)$.  

The depth $d_-$ of the component $\cald_-$ of $\cald$ counts the intersections of $\gamma_-$ with $\cald$, which means that it counts the times that $(x(s),y(t_-))$ is a page vector for $s_-\le s\le s_+$, assuming parametrizations set up so that $s_-<s_+$\,.  Since $y(t_-)$ already lies in the page $H_-$, this means we are counting the times that $x(s)$ also lies in $H_-$ for $s_-\le s\le s_+$\,.  In other words we are counting the intersection number $X_{-+}\dt H_-$\,.  Examining \figref{local}, we find that the signs of the intersection  points of $\gamma_-$ with $\cald$ {\it agree} with the signs of the corresponding intersection points of $X_{-+}$ with $H_-$, and hence
$$
d_- \ = \ 2\, X_{-+} \dth H_- \mod{2\gcd(p,q,r)}.
$$

In a similar fashion, we find that intersections of $\gamma_+$ with $\cald$ correspond to intersections of $Y_{+-}$ with $H_+$, but in this case, the signs of corresponding points of intersection are {\it opposites}, and so the depth of $\cald_+$ is 
$$
d_+ \ = \ -2\,Y_{+-} \dth H_+\mod{2\gcd(p,q,r)}.
$$ 

Therefore the assertion that  $d_- - d_+ = 2\lk(P,Z)$, which will complete the proof of the corollary, reduces to the identity
$$
X_{-+} \dth H_-  +  Y_{+-} \dth H_+ \ = \ \lk(P,Z)\,.
$$
But this follows easily from the fact that $P = X_{-+} \cup Y_{+-}$.  \figref{linking} helps us to see this.  In the figure, the arc  $X_{-+}$  winds around the vertical $Z$ axis $k$ times and the arc $Y_{+-}$ winds around it $\ell$ times, while the entire loop $P = X_{-+} \cup Y_{+-}$ winds around it $k+\ell+1$ times.   By our half-counting convention, we have 
$$
X_{-+} \dt H_- \ = \ 1/2 + k \qquad\text{and}\qquad Y_{+-} \dt H_+ \ = \ 1/2 + \ell
$$
and the result follows.  
\end{proof}

%%% FIGURE 22: Linking %%%
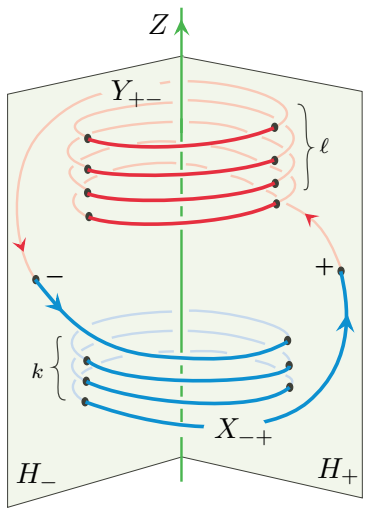
\begin{figure}[h!]
\includegraphics[height=190pt]{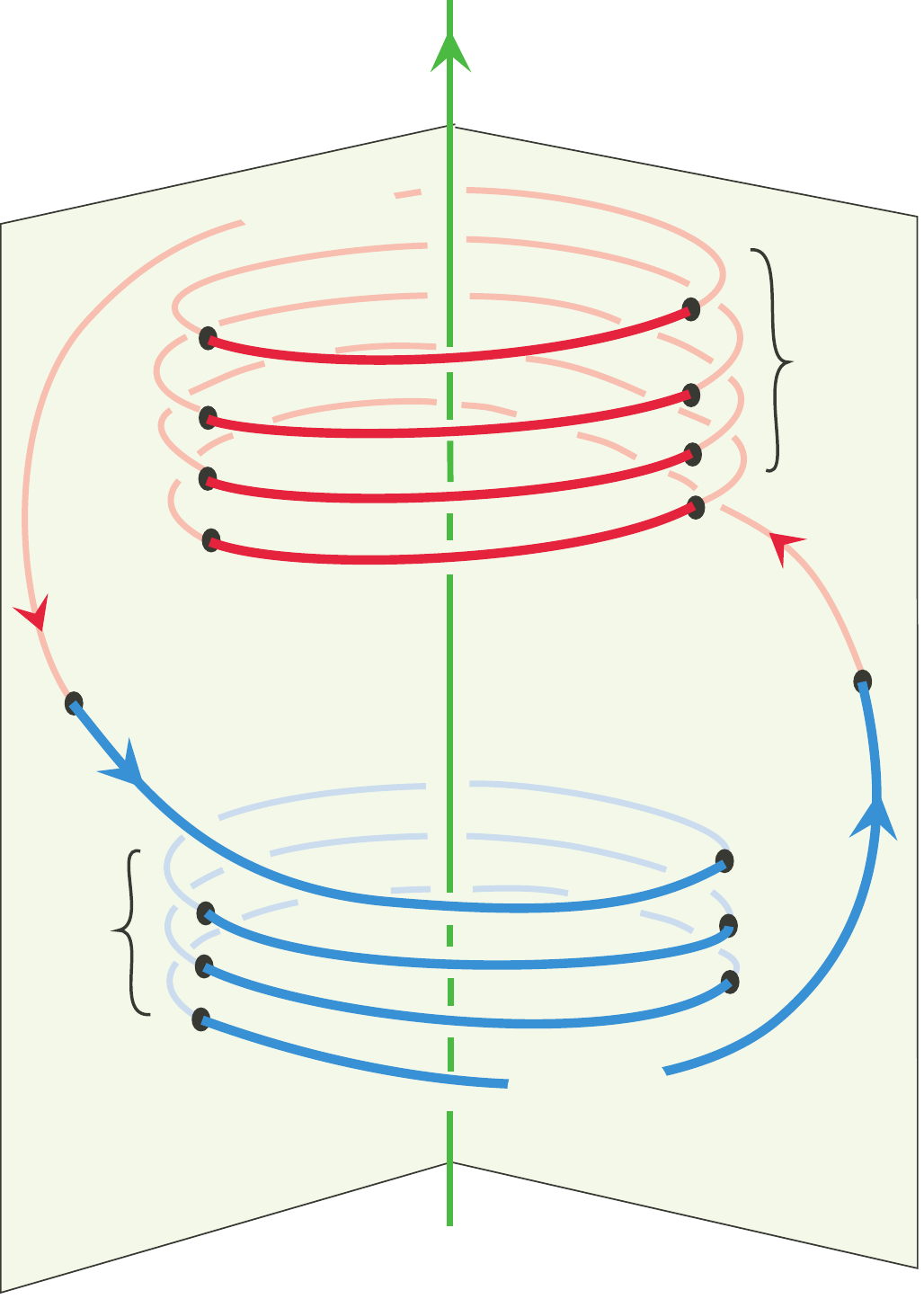}
\put(-132,10){$H_-$}
\put(-18,12){$H_+$}
\put(-57,27){$X_{-+}$}
\put(-96,155){$Y_{+-}$}
\put(-82,180){$Z$}
\put(-126,50){$\scriptstyle k$}
\put(-17,135){$\scriptstyle \ell$}
\put(-121,86){\boldmath$\scriptstyle -$}
\put(-19,90){\boldmath$\scriptstyle +$}
\caption{The linking of $P=X_{-+} \cup Y_{+-}$ and $Z$}
\label{fig:linking}
\end{figure}
%%%%%%%%%%%%%%%%%%%%

For the inductive step of the proof of \thmref{A} in the next section, we will need to apply this formula for $\Delta\nu$ under the double crossing change $L\to\what L$ shown in \figref{double} (which will be seen to be equivalent to a delta move).  It is understood that $L$ and $\what L$ should coincide outside the picture, where in fact they can be arbitrary.  Indeed, if not generic outside the ball, they can be adjusted by a link homotopy to become so, and then the methods described above apply.   Since the component $Q$ of the smoothed link $L_0$ is just a meridian of $Z$ with $\lk(Q,Z) = -1$, it follows that $\Delta\nu = 2$.  

Thus we have proved the following: \label{double}

%%% FIGURE 25: Double %%%
\begin{figure}[h!]
\includegraphics[height=100pt]{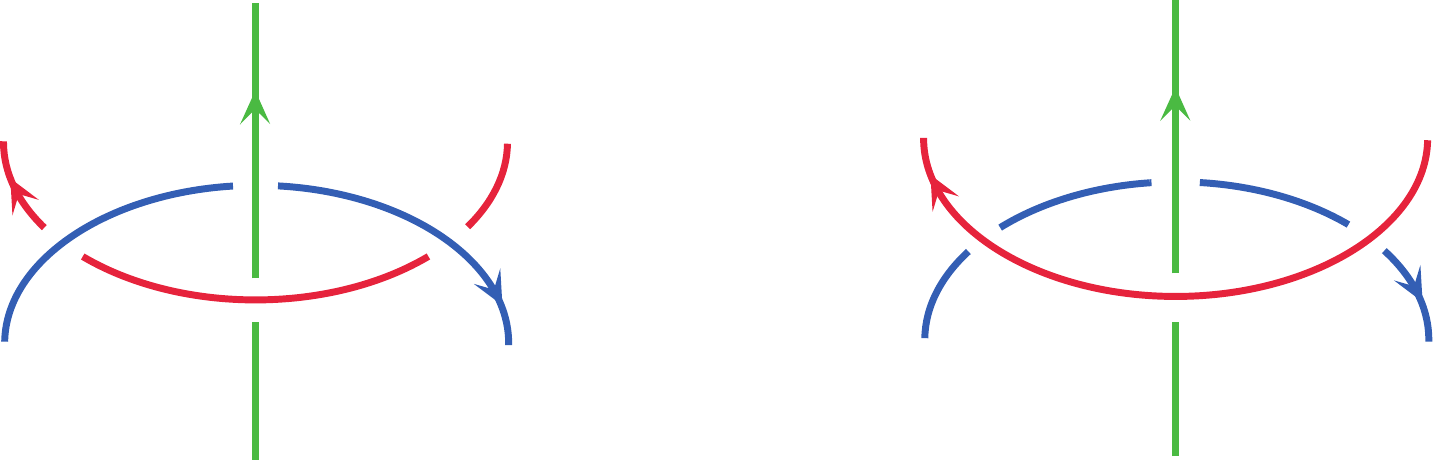}
\put(-315,14){\small$X$}
\put(-203,73){\small$Y$}
\put(-267,90){\small$Z$}
\put(-300,37){\footnotesize$+$}
\put(-218,37){\footnotesize$-$}
\put(-100,37){\footnotesize$-$}
\put(-18,37){\footnotesize$+$}
\put(-285,-15){(a) the link $L$}
\put(-90,-15){(b) the link $\what L$}
\caption{A simple double crossing change}
\label{fig:double}
\end{figure}
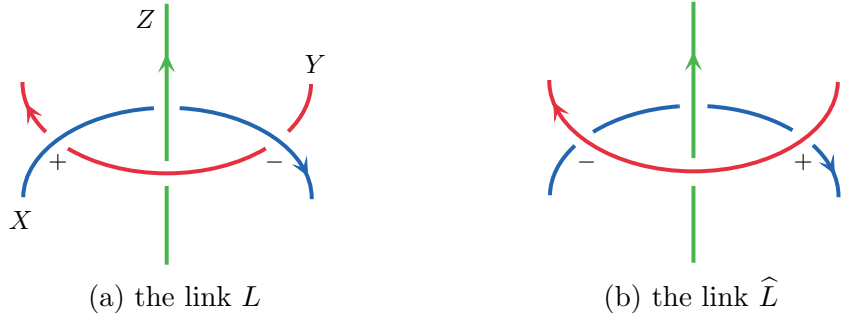
%%%%%%%%%%%%%%%%%%%%

\begin{corollary}\label{cor:double}
\textbf{\boldmath 
If two links $L$ and $\what L$ coincide outside a $3$-ball, and appear in the ball as shown in \figref{double},  then $\nu(h_{\what L}) = \nu(h_L) + 2$.}
\end{corollary}

We now embark on the proof of the bicycle theorem, which will occupy us for the rest of this section.  

%%%%%%%%%%%%%%%%%%%%%%%%%%%
%%%%%%%%%%%%%%%%%%%%%%%%%%%
\part{Proof of the bicycle theorem}
%%%%%%%%%%%%%%%%%%%%%%%%%%%
%%%%%%%%%%%%%%%%%%%%%%%%%%%

Start with a generic link $L$ in $S^3$ where, as usual, $x=x(s)$, $y=y(t)$ and $z=z(u)$ parametrize the three components $X$, $Y$ and $Z$.   Then $Z$ coincides with the binding $K$ of the standard open book, which is the circle subgroup of $S^3$ containing the quaternion $k$.  Looking back at the formula for the characteristic map 
$$
h_L(s,t,u) \ = \ (y_z-x_z)/|y_z-x_z| \quad\text{where}\quad a_z = \pr_{-1}(-a\zbar)\,,
$$
we see that we must understand the $K$-action sending $z\in K$ to the automorphism $a\mapsto a\zbar$ of $S^3$, and the induced $K$-action on $\br^3$ via stereographic projection.  Visualizing how these actions transform the pages of the standard open books will aid us in our subsequent arguments, and we do this next.

%%%%%%%%%%%%%%%%%%%%%%%%%%%
%%%%%%%%%%%%%%%%%%%%%%%%%%%
\part{The geometry of the $K$-action}
%%%%%%%%%%%%%%%%%%%%%%%%%%%
%%%%%%%%%%%%%%%%%%%%%%%%%%%

For any $z =  \cos\alpha + k\sin\alpha$ in the binding $K$,
right multiplication by $\zbar =  \cos\alpha -k\sin\alpha$ is an isometry of $S^3$ that rotates $K$ by $-\alpha$ radians, likewise rotates the orthogonal great circle $C$ of page centers by $\alpha$ radians, and so advances each hemispherical page $H_\theta$ to the page $H_{\theta+\alpha}$ while simultaneously rotating the page by $-\alpha$ radians about its center $i_\theta$.  Therefore, as $z$ traverses $K$, any given page turns once around $K$, successively occupying the positions vacated by the other pages.  During this time, the page spins once negatively about its center, so that in total it is following a left-handed screw motion along $C$.   

This turning of the hemispherical pages about $K$ in $S^3$ is transferred by stereographic projection to a turning of the half-planar pages about the $k$-axis in $\br^3$, while the {\it spherical} rotations of the pages in $S^3$ become {\it hyperbolic} rotations of the pages in $\br^3$ about their centers.  This is a consequence of the conformality of stereographic projection, which implies that the page identification $H_\theta \leftrightarrow P_\theta$ is conformal.  Once again, the net effect is a left-handed screw motion along $C$.

This description of the $K$-action has the following technical consequence that is critical for our study of the characteristic map of a generic link.

\break

  %%%%%  TWIST LEMMA 5.2  %%%%%
\begin{lemma}\label{lem:twist}
\textbf{\boldmath \textup{(Twist Lemma)}
Let $x$ and $y$ be distinct points in $S^3$ lying in the complement of the binding $K$ of the standard open book.
\begin{indentation}{-.2em}{1.2em}
\begin{itemize}
\item[(a)]  If $x$ and $y$ lie on different pages, then for $z\in K$, the vector $y_z-x_z$ never lies on a page of the corresponding open book in $\br^3$, and in particular never points straight up.
\smallskip  
\item[(b)] If $x$ and $y$ lie on the same page, then as $z$ traverses $K$, the vectors $y_z-x_z$ lie on successive pages in $\br^3$, turning once positively around the binding, and spinning once counterclockwise without backtracking within the pages as they go.  In particular, $y_z-x_z$ points straight up for a unique $z = \tau(x,y) \in K$.    
\end{itemize}
\end{indentation}}
\end{lemma}
%%%%%%%%%%%%%%%%%

\begin{proof}
Since the $K$-action carries pages to pages, $x$ and $y$ will lie on the same page in $S^3$ if and only if $x_z$ and $y_z$ lie on the same page in $\br^3$ for all $z\in K$.  Part (a) of the lemma is now obvious.   Using the meridional projection, and the description of the $K$-action above, part (b) of the lemma translates into the following statement.  For any pair of distinct points $x$ and $y$ in the upper half-plane model of the hyperbolic plane $\bh$, the function 
$$
d(\alpha) \ = \ \arg(\rot_\alpha y-\rot_\alpha x)
$$
is strictly increasing, where $\rot_\alpha:\bh\to\bh$ is hyperbolic rotation about $i$ by $\alpha$ radians.   Since we wish to prove this for all $x$ and $y$, it suffices to show that $d'(0)>0$ (because $d'(\alpha)$ for one choice of $x$ and $y$ is equal to $d'(0)$ for some other choice of $x$ and $y$).    

To prove this, we transfer the problem to the Poincar\'e disk $\bd$ using the conformal map 
$$
f:\bd\ \longrightarrow\ \bh \quad,\quad f(z) \ = \  \frac{1-z}{1+z}\ i
$$
that sends $0$ to $i$.  Hyperbolic rotation of $\bh$ about $i$ by any angle is conjugate by $f$ to euclidean rotation of $\bd$ by the same angle.   Thus we must show that for any pair of distinct points $x$ and $y$ in $\bd$, the function
$$
g(\alpha) \ = \ \arg(f(e^{i\alpha}y)-f(e^{i\alpha}x))
$$
has positive derivative at $\alpha = 0$.   Noting that  $f(a)-f(b) = 2i(b-a)/(1+a)(1+b)$ and using the fact that the argument function converts products into sums and quotients into differences, we find that $g(\alpha)$ differs by a constant from the function 
$$
h(\alpha) \ = \ \alpha - \arg(1+e^{i\alpha}x) - \arg(1+e^{i\alpha}y),
$$  
and so it remains to show that $h'(0)>0$.  But a simple geometric argument using the central angle theorem from elementary plane geometry shows that, for any given $z\in \bd$, the derivative of the function $\arg(1+e^{i\alpha}z)$ at $\alpha = 0$ is strictly less than $1/2$.  Therefore $h'(0) > 1 - 1/2-1/2 \, = \, 0$ as desired.
\end{proof}

It follows from the twist lemma that, for any point $(s,t)$ in the $2$-torus that is isogonal for our generic link $L$ (meaning that  $x=x(s)$ and $y=y(t)$ have the same polar angle, and hence lie on the same page of the open book),  there exists a unique $u = u(s,t)\in S^1$ for which the vector $y_z-x_z$ points straight up, namely the $u$ for which $z = z(u) = \tau(x,y)$.   It follows that the link
$$
\bl \ = \ h_L^{-1}(k) \ = \ \{(s,t,u)\in\tor \st (s,t)\in\cald \text{ and } u=u(s,t)\}\,, \label{pontryaginlink}
$$
is the graph of the function $u(s,t)$ over the collection $\cald$ of isogonal curves in $\torh$, as asserted in the introduction.  

We now use the twist lemma to prove part (a) of the bicycle theorem, asserting that $k$ is a regular value of $h_L$.  This will endow the components of $\bl$ with orientations (and thus vertical winding numbers) and framings, as defined in \secref{pontryagin}.  The proof of the bicycle theorem will then be completed by showing that these agree with the preferred orientations, vertical winding numbers and framings of the components of $\cald$, as defined earlier in this section.

%%%%%%%%%%%%%%%%%%%%%%%%%%%
%%%%%%%%%%%%%%%%%%%%%%%%%%%
\part{Why is $k$ a regular value of $h_L$?}
%%%%%%%%%%%%%%%%%%%%%%%%%%%
%%%%%%%%%%%%%%%%%%%%%%%%%%%

Suppose that $h_L(s,t,u) = k$.  This means that the vector from $x_z $ to $y_z$ points straight up, and so lies in some page $P$ of the standard open book in $\br^3$.
% \foot{\,In particular $P=P_{\theta+\alpha}$ where $\theta$ is the polar angle of $x$ and $y$, and $z = \cos\alpha + k\sin\alpha$.\label{footnote}}  
We must show that at $(s,t,u)$ the vectors $\partial_s h_L$, $\partial_t h_L$ and $\partial_u h_L$ span the tangent space to $S^2$ at $k$.  

We are guided by \figref{regular}, depicting a neighborhood of the vertical vector $\vec$ from $x_z $ to $y_z$ in $\br^3$, lying on the page $P$.   In this figure, $X_z$ and $Y_z$ denote the images of $X$ and $Y$ under the $z$-action, one or both of which must be transverse to $P$ since $L$ is generic.  We arbitrarily depict $X_z$ tangent to $P$ and $Y_z$ transverse to it.  

%%% FIGURE 26: Regular Value %%%
\begin{figure}[h!]
\includegraphics[height=170pt]{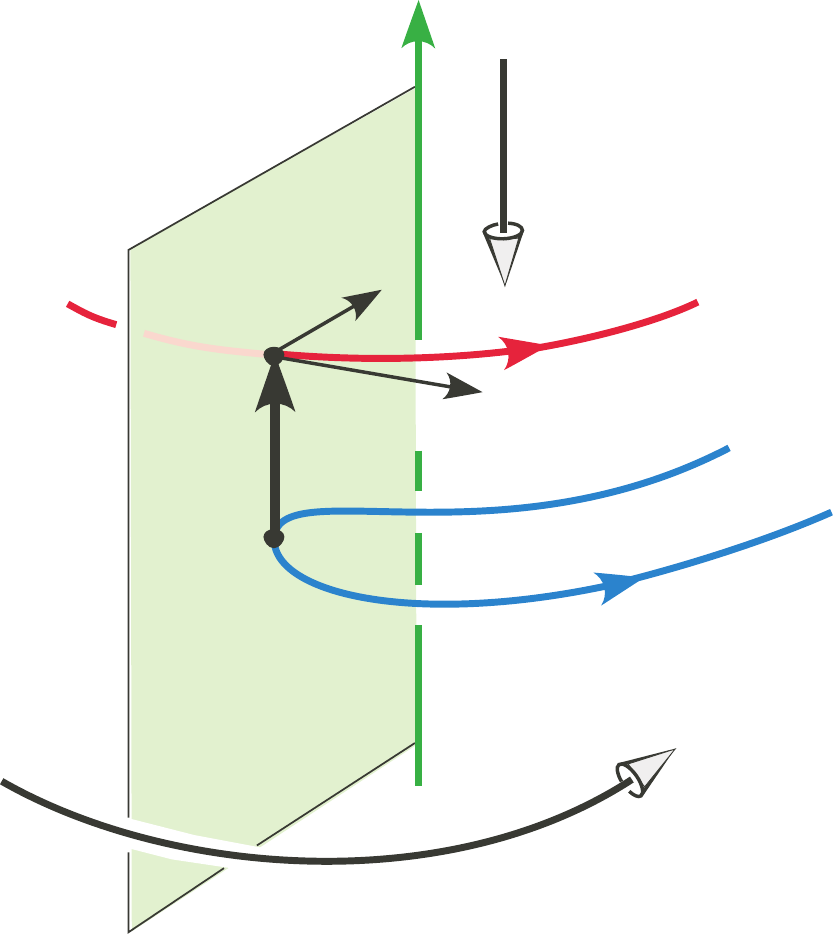}
\put(-114,105){\small$y_z$}
\put(-113,66){\small$x_z$}
\put(-113,83){\small$\vec$}
\put(-21,116){$Y_z$}
\put(-17,84){$X_z$}
\put(-100,120){\tiny$\partial_u h_L$}
\put(-85,92){\tiny$\partial_t h_L$}
\put(-120,29){$P$}
\put(-95,5){\scriptsize TURN}
\put(-56,158){\scriptsize S}
\put(-56.5,150){\scriptsize P}
\put(-55.5,142){\scriptsize I}
\put(-57,134){\scriptsize N}
\caption{$k$ is regular}
\label{fig:regular}
\end{figure}
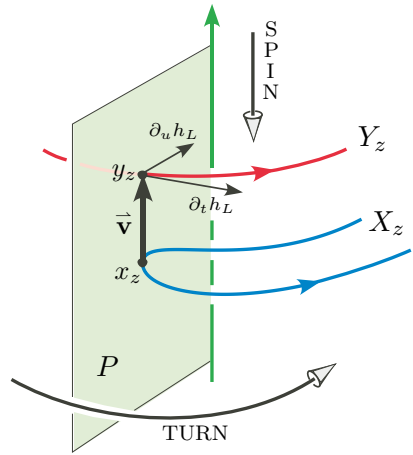
%%%%%%%%%%%%%%%%%%%%

The hollow arrows in the figure indicate the left-handed screw motion of the $K$-action.  In particular, as the $u$ parameter increases along $K$, this action (the stereographic image of right multiplication by $-\zbar$) turns the pages in the indicated direction while spinning them in a left-handed fashion about their centers. 

As the $t$ parameter along $Y$ increases, the tip of $\vec$ moves to the right along $Y_z$, so that $\partial_t h_L$ also points to the right, transverse to $P$.  For clarity, we draw this partial derivative vector at the tip of $\vec$, even though it is really tangent to $S^2$ at its north pole $k$. 

As the $u$ parameter along $Z$ increases, two things happen which affect $\vec$.  The page containing it turns around the binding, but such an action keeps $\vec$ vertical, so has no infinitesimal effect.  In addition, the page spins around its center, shifting the binding in a downward direction.  Thus by the twist lemma, the vector $\partial_u h_L$ is nonzero  and tangent to the page, pointing toward the binding when placed with its initial point at $y_z$.     

It follows that $\partial_t h_L$ and $\partial_u h_L$ are linearly independent, confirming that $(s,t,u)$ is a regular point of the map $h_L$.  Since $(s,t,u)$ was chosen arbitrarily in $h_L^{-1}(k)$, we see that $k$ is a regular value of $h_L$, proving part (a) of the Bicycle Theorem.

%%%%%%%%%%%%%%%%%%%%%%%%%%%
%%%%%%%%%%%%%%%%%%%%%%%%%%%
\part{Finishing the proof of the Bicycle Theorem}
%%%%%%%%%%%%%%%%%%%%%%%%%%%
%%%%%%%%%%%%%%%%%%%%%%%%%%%

Starting with a three-component link $L$ in generic position in $S^3$, we have shown above 
that the Pontryagin link $\bl$ of its characteristic map $h_L$ is given by
$$
\bl \ = \ h_L^{-1}(k) \ = \ \{(s,t,u)\in\tor \st (s,t)\in\cald \text{ and } u=u(s,t)\}\,. 
$$
It remains to prove that the preferred orientation, vertical winding number $r_i$ and framing $n_i$ of each isogonal curve $\cald_i$, as specified in \figref{local} and just before the statement of the bicycle theorem,  agree with the orientation, vertical winding number and framing of the component $\bl_i$ of the Pontryagin link that lies over $\cald_i$. 

To see that the orientations agree, choose a regular page vector $(x,y)$ parametrized by an isogonal point $(s,t)$ in $\cald_i$, and consider the orientation of $Y$ relative to the page containing $x$ and $y$, as recorded by the sign of $y$.  

\figref{orientation} shows the associated vertical vector $(x_z,y_z)$ in $3$-space for the case when $\sgn(y) = +1$,  causing the displaced link component $Y_z$ to point to the {\it right}.  Arguing as above (see \figref{regular}) we see that  $\partial_t h_L$ also points to the right, while $\partial_u h_L$ points toward the binding as always.  Thus $\partial_t h_L$ and $\partial_u h_L$ form a positive basis for the tangent plane to $S^2$ at $k$.  Since $\partial_s$,  $\partial_t$ and $\partial_u$ form a positive basis for the tangent space to $\tor$ at each of its points, it follows that the link component $\bl_i$ must be oriented in the direction of {\it increasing} $s$ near the associated point $(s,t,u)$ in the $3$-torus, and this agrees with the orientation assigned to $\cald_i$ in \figref{local}.  

%%% FIGURE 27: Orientation %%%
\begin{figure}[h!]
\includegraphics[height=170pt]{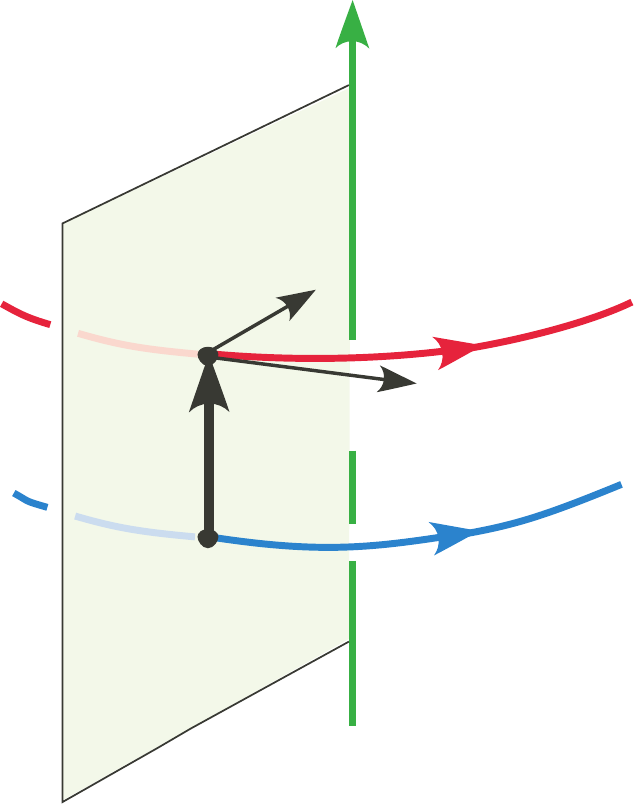}
\put(-100,100){\small$y_z$}
\put(-100,48){\small$x_z$}
\put(-19,110){$Y_z$}
\put(-17,50){$X_z$}
\put(-81,113){\tiny$\partial_u h_L$}
\put(-66,79){\tiny$\partial_t h_L$}
\caption{Orientation of the Pontryagin link}
\label{fig:orientation}
\end{figure}
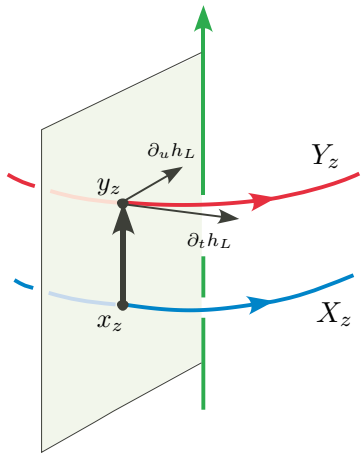
%%%%%%%%%%%%%%%%%%%%

Similarly, if $\sgn(y)=-1$ then $\bl_i$  points in the direction of {\it decreasing} $s$ near $(s,t,u)$.  Referring again to \figref{local}, it is evident that this agrees with the preferred orientation on $\cald_i$.  

\begin{comment}
Alternatively, one can show that $\sgn(x)=+1$ or $-1$ according to whether $\bl_i$ is oriented in the direction of increasing or decreasing $t$, which also agrees with the preferred orientation on $\cald_i$.
\end{comment}

Next consider the vertical winding number $r_i$ of $\cald_i$.  By definition $r_i$ is equal to the meridional degree $m_i$ of the associated bicycle $\calp_i$, that is, the number of times that the page vectors spin in the pages of the open book as the bicycle is traversed.  By the twist lemma, this is equal to the number of times that $z$ spins around $K$, or equivalently that $u = u(s,t)$ spins around the circle as $(s,t)$ traverses $\cald_i$, which is the vertical winding number of $\bl_i$. 

Finally consider the framing $n_i$ of $\cald_i$.  By definition $n_i = -\ell_i-m_i$, where $\ell_i$ and $m_i$ are the longitudinal and meridional degrees of the associated bicycle $\calp_i$.
This integer specifies a normal vector field $\vecn$ to $\bl_i$ by adding $n_i$ full twists (positive or negative according to the sign of $n_i$) to the lift $\vecz$ of a normal vector field to $\cald_i$ in $T^2$, the ``zero" or ``blackboard" framing.  

We must show that the vector field $\vecn$ coincides with the Pontryagin framing of $\bl_i$.  In other words, the differential of $h_L$ carries it onto a homotopically trivial loop of nonzero tangent vectors to $S^2$ at $k$.  To see this, first note that $\vecz$ is homotopic to the vertical vector field $\partial_u$, remaining nonzero and transverse to $\bl_i$ during the homotopy.  Now since $\partial_u h_L$ is always tangent to the page containing $x_z$ and $y_z$, as shown above, it is clear that each longitudinal circuit of the bicycle $\calp_i$ causes $dh_L(\vecz)$ to spin once in the same direction about the binding.  By the twist lemma, each meridional spin of the bicycle also causes $dh_L(\vecz)$ to spin once about the binding.  Thus $dh_L(\vecz)$ spins $ \ell_i+m_i$ times as the bicycle is traversed.  Adding $n_i= -\ell_i-m_i$ full twists to the framing negates this spinning, and so $dh_L(\vecn)$ does not spin at all, up to isotopy, as asserted. 

This completes the proof of the bicycle theorem.

%% file: 6A.tex
%%%%%%%%%%%%%%%%%%%%%%%%%%%
%% SECTION 7: Proof of Theorem A   %%
%%%%%%%%%%%%%%%%%%%%%%%%%%%

%%%%%%%%%%%%%%%%%%%%%%%%%%%
\section{Proof of \thmref{A}}\label{sec:A}
%%%%%%%%%%%%%%%%%%%%%%%%%%%

Let $L$ be a three-component link in the $3$-sphere with characteristic map $g_L:\tor\to S^2$.  \thmref{A} asserts that the degrees of $g_L$ on the coordinate $2$-tori of $\tor$ are equal to the pairwise linking numbers $p$, $q$ and $r$ of the components of $L$, and that Pontryagin's absolute $\nu$-invariant of $g_L$ is equal to twice Milnor's triple linking number $\mu(L)$ mod $2\gcd(p,q,r)$.  

The first statement was proved easily at the end of \secref{characteristic}.  We now prove the second statement by an inductive argument, relying heavily on the techniques and results of the last section.  Throughout, we will use the asymmetric characteristic map $h_L:\tor\to S^2$ in place of $g_L$, since these two maps are homotopic.

%%%%%%%%%%%%%%%%%%%%%%
%%%%%%%%%%%%%%%%%%%%
\part{Proof of the base case of \thmref{A}}
%%%%%%%%%%%%%%%%%%%%
%%%%%%%%%%%%%%%%%%%%%%

In \exref{Lpqr} we introduced the  ``base links" $L_{pqr}$ with pairwise linking numbers $p$, $q$ and $r$ and $\mu(L_{pqr}) = 0$.  See \figref{Lpqr}.  We will show that \thmref{A} holds for these links.

To accomplish this, we must show that for each $p$, $q$ and $r$, the associated characteristic map $h_{L_{pqr}}$ is homotopic to the base map $f_{pqr}:\tor\to S^2$ used in \secref{pontryagin} to convert the relative Pontryagin $\nu$-invariant to an absolute $\nu$-invariant.  In other words, we must show that $\nu(h_{L_{pqr}}) = 0$.

First move $L_{pqr}$ by a link homotopy into generic position, and let $L=X\cup Y\cup Z$ be the resulting link.  Then $Z$ coincides with the binding $K$ of the standard open book, while $X$ and $Y$ wind around $Z$ in a generic fashion, $q$ and $p$ times respectively, linking each other $r$ times along the way.  In fact this winding can be made {\it monotonic}, to appear as shown in \figref{genericLpqr} for the case $(p,q,r)=(5,3,-2)$. 

%%% FIGURE 28: Generic Lpqr %%%
\begin{figure}[h!]
\includegraphics[height=200pt]{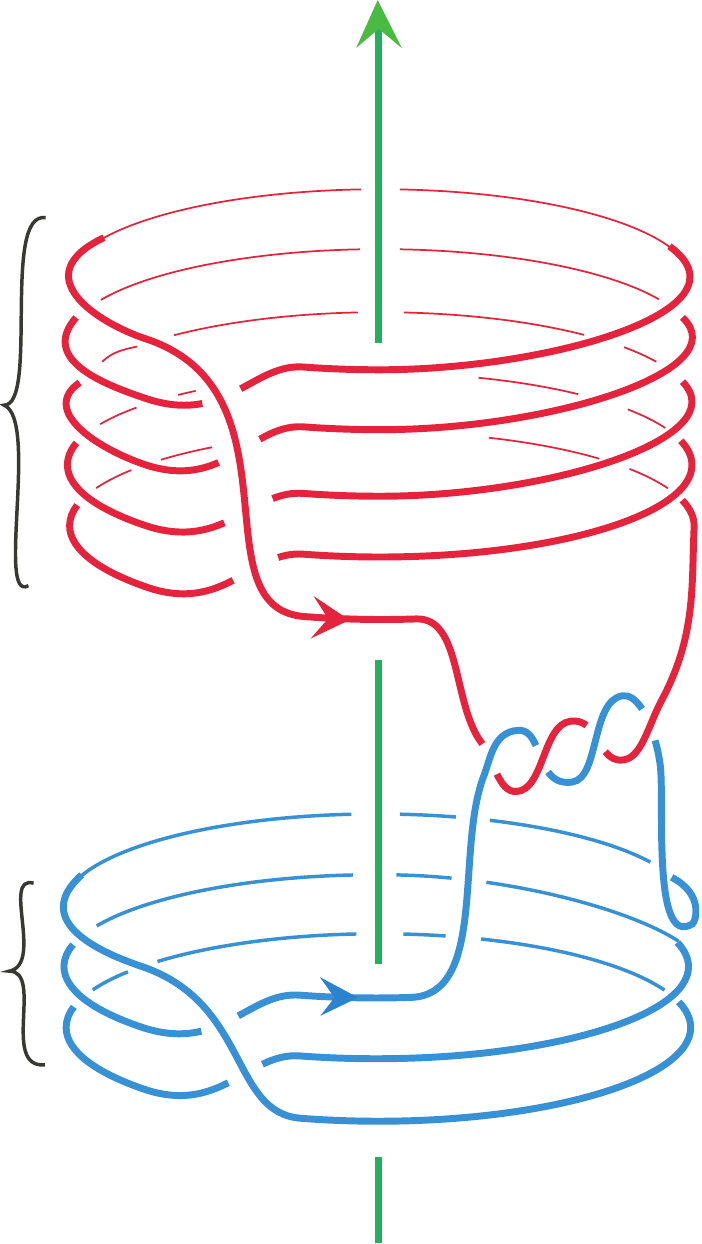}
\put(-122,42){\small$q$}
\put(-122,133){\small$p$}
\put(-25,90){\small$r$}
\put(-1,20){$X$}
\put(3,157){$Y$}
\put(-65,180){$Z$}
\caption{A generic link $L$ representing $L_{pqr}$}
\label{fig:genericLpqr}
\end{figure}
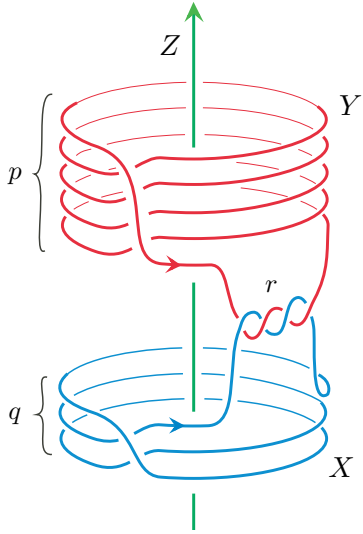
%%%%%%%%%%%%%%%%%%  

Observe that one need not actually construct a link homotopy carrying $L_{pqr}$ to $L$, but need only check that the pairwise and triple linking numbers of $L_{pqr}$ and $L$ coincide, and then appeal to Milnor's classification.  The verification that $\mu(L)=0$ is completely analogous to the calculation for $\mu(L_{pqr})$, using the obvious stacked disks, joined by half twisted bands, as Seifert surfaces for $X$ and $Y$, and a hemispherical page in the open book on $S^3$ as a Seifert surface for $Z$. 

Now the bicycles in $L$ are easily identified.  There are $d = \gcd(p,q)$ of them, all of longitudinal degree $\ell = \lcm(p,q)$, and all but one of meridional degree $0$, the remaining one being of degree $r$.  The parametrizing icycles are all parallel to a $(p/d,q/d)$ torus knot in the $2$-torus, with vertical winding numbers all equal to zero, except one equal to $r$.   By the bicycle theorem, these icyles form a diagram for the Pontryagin link of $h_L$, with global framing $-(pq+r)$.  This is exactly the situation described in \exref{basecase}, and so  
$$
\nu(h_{L_{pqr}}) \ = \ \nu(h_L) \ = \ 0 \ \in \ \bz_{2\gcd(p,q,r)}
$$
as asserted.

%%%%%%%%%%%%%%%%%%%%%%
%%%%%%%%%%%%%%%%%%%%%%
\part{Proof of the inductive step of \thmref{A}}
%%%%%%%%%%%%%%%%%%%%%%
%%%%%%%%%%%%%%%%%%%%%%

First recall from \exref{delta} that any three-component link $L$ with pairwise linking numbers $p$, $q$ and $r$ is link homotopic to a link obtained from $L_{pqr}$ by a sequence of delta moves of the type shown in \textup{\figref{delta}} $($or its inverse$)$, by a result of Murakami and Nakanishi \cite{MurakamiNakanishi}.  It was shown in that example that each such move increases $\mu(L)$ by $1$.  

Now observe that such a delta move $\Delta:L\to\what L$ can be viewed as a double crossing change $\bd$, of the kind shown in \figref{double},
% on page \pageref{double} 
composed with an isotopy, as indicated in \figref{deltadouble}.  By \corref{double}, this double crossing change increases $\nu(L)$ by $2$, and so $\Delta$ does the same. 

%%% FIGURE 29: DeltaDouble %%%
\begin{figure}[h!]
\includegraphics[height=180pt]{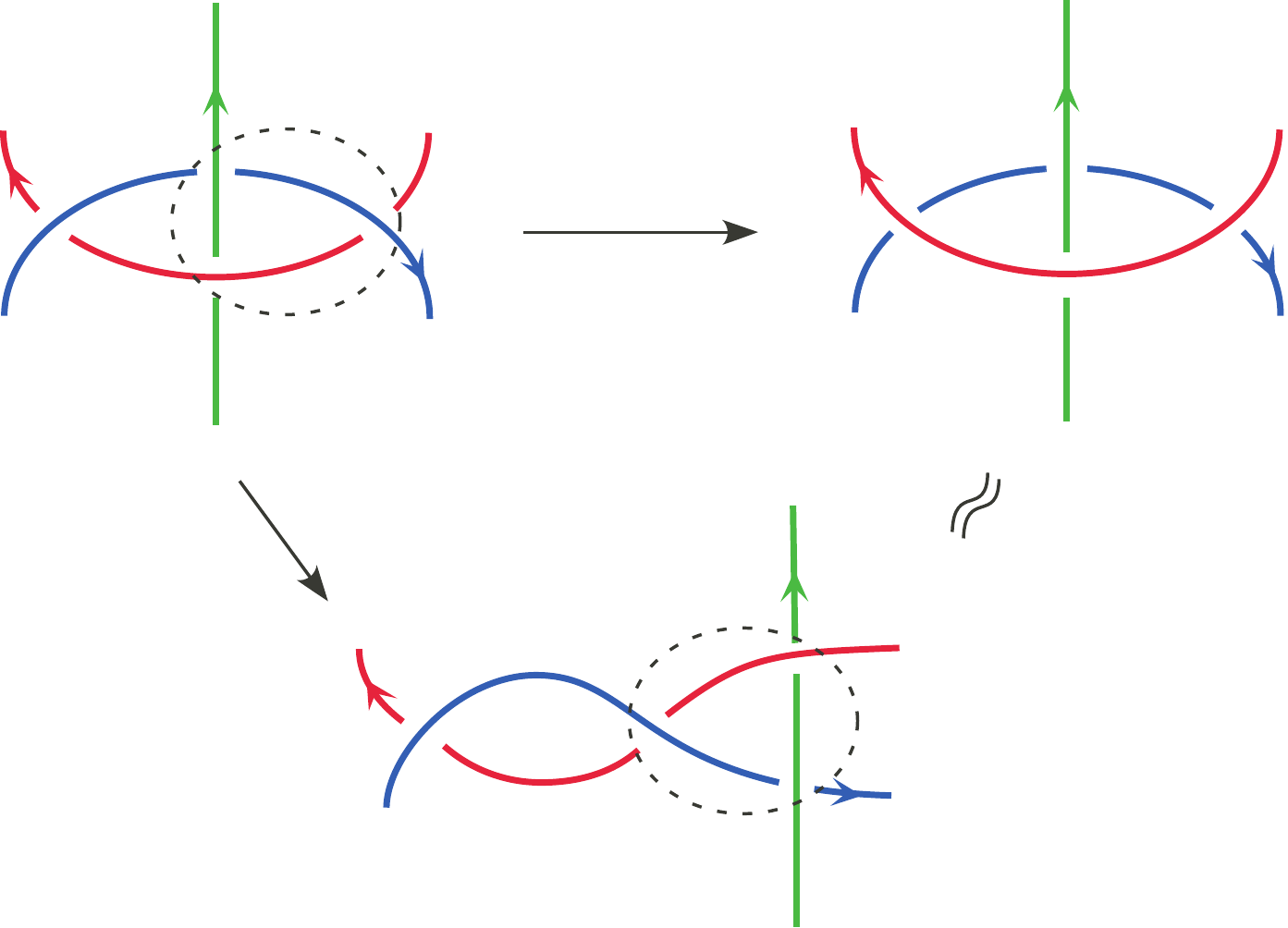}
\put(-55,72){ isotopy}
\put(-132,142){\small$\bd$}
\put(-210,68){\small$\Delta$}
\put(-270,135){$L$}
\put(-255,107){\small$X$}
\put(-170,160){\small$Y$}
\put(12,135){$\what L$}
\put(-140,5){$\what L$}
\caption{Another view of the delta move}
\label{fig:deltadouble}
\end{figure}
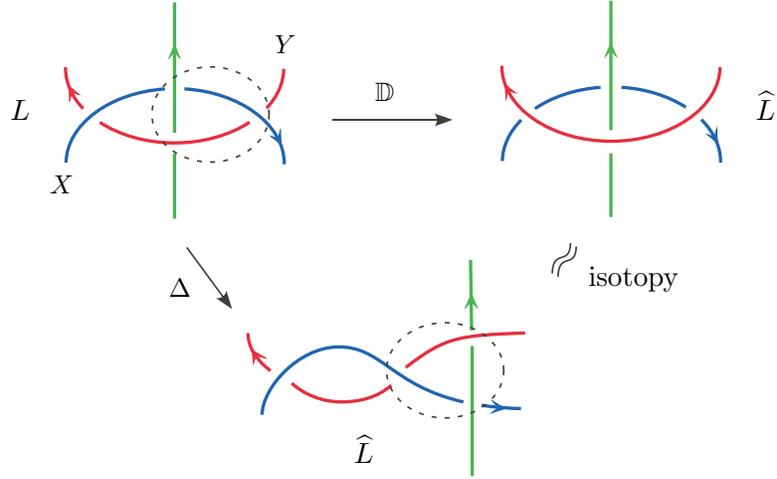
%%%%%%%%%%%%%%%%%%   

Since this argument applies to an arbitrary delta move, it follows that $\nu(L) = 2\mu(L)$ for {\it all} three component links $L$.  This completes the proof of \thmref{A}. 

%% file: 7A.tex
%%%%%%%%%%%%%%%%%%%%%%%%%
%% SECTION 7: Sketch of algebraic proof of Theorem A
%%%%%%%%%%%%%%%%%%%%%%%%%

%%%%%%%%%%%%%%%%%%%%%%%%%%%
\section{Sketch of an algebraic proof of \thmref{A}}\label{sec:A'}
%%%%%%%%%%%%%%%%%%%%%%%%%%%

In this section we sketch an entirely different proof of the theorem %\thmref{A} 
using string links and maps of the $2$-torus to the $2$-sphere.  The reader is referred to our paper \cite{DGKMSV} for more details.  
 
The proof is organized around the following key diagram, in which the left half, devoid of algebraic structure, represents the topological problem we are trying to solve, while the right half represents the algebraic structures that we impose on the left half via the two horizontal maps in order to solve the problem. 

\begin{equation*}\label{eqn:keydiagram}
\begin{diagram}[size=3em]
\fbox{\parbox{0.24\linewidth}{\begin{center} \small Three-component \\ links in $S^3$ \\ up to link homotopy\end{center}}}  
&&& 
\lTo^{\text{closing up}} 
&&& 
\fbox{\parbox{0.24\linewidth}{\begin{center} \small Three-component \\ string links \\ up to link homotopy\end{center}}}    
\\ \\
\dTo^{\text{characteristic}\atop\text{maps}}_{\ g} 
&&&&&& 
\dTo_{\text{makes diagram}\atop\text{commutative}}  
\\ \\
\fbox{\parbox{0.24\linewidth}{\begin{center} 
\small Maps of $\tor$ to $S^2$ \\ up to homotopy\end{center}}}  
&&& 
\lTo_{\text{natural map}} 
&&& 
\fbox{\parbox{0.24\linewidth}{\begin{center} 
\small Fundamental groups \\ of spaces of maps \\ of $T^2$ to $S^2$\end{center}}}  \\ 
\end{diagram}
\end{equation*}

\

In the upper left corner of the diagram we have the set of link homotopy classes of three-component links in the $3$-sphere $S^3$, and in the lower left corner the set of homotopy classes of maps of the $3$-torus $\tor$ to the $2$-sphere $S^2$.  The vertical map $g$ between them assigns to the link homotopy class of $L$ the homotopy class of its characteristic map $g_L$.  \thmref{A} describes $g$ and asserts that it is one-to-one.

A $k$-component {\itb string link} consists of $k$ disjoint, oriented, properly embedded arcs in a cube, with their tails on the bottom face and their tips on the top face directly above their tails. The product of two such string links with endpoints in a common position is given by stacking the second one on top of the first.  When a string link moves by a link homotopy, each strand is allowed to cross itself, while different strands must remain disjoint, just as for links.  Then the above product induces a group structure on the set of link homotopy classes of $k$-component string links.

In the upper right corner of the diagram we have the group of link homotopy classes of three-component string links.  Following Habegger and Lin \cite{HabeggerLin}, an explicit presentation for this group can be given in which the pairwise linking numbers and the Milnor $\mu\text{-invariant}$  figure prominently. 
 
The upper horizontal map denotes the operation of closing up a string link to a link.  It is easily shown that this map is onto, and that point inverse images are conjugacy classes of string links,  a special circumstance for links with three components which fails for four or more components.  

We see in the figure below that the string link on the right closes up to the Borromean rings on the left.

%%% FIGURE 30: Borr String Link Closure %%%
\begin{figure}[h!]
\includegraphics[height=110pt]{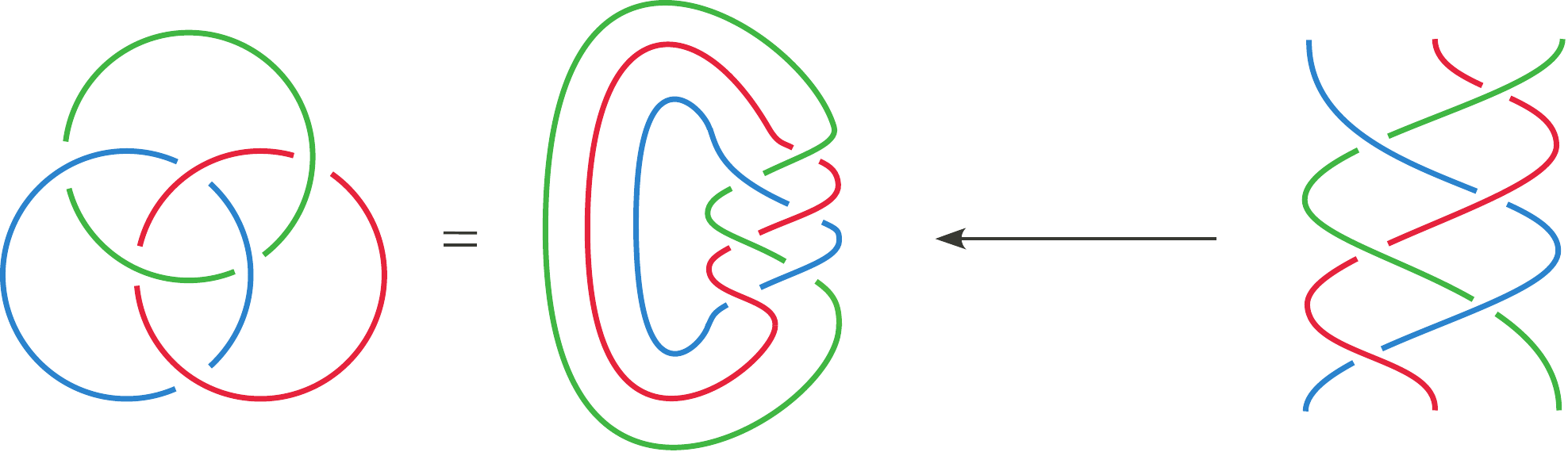}
\put(-136,58){closes up}
\caption{Borromean rings as the closure of a string link}
\label{fig:br}
\end{figure}
%%%%%%%%%%%%%%%%%%%%

\noindent Thus the Borromean rings, a ``primitive example" in the world of links, is the closure of a string link which is itself a commutator of simpler string links. 
 
In the lower right corner of the diagram we have the union of the fundamental groups of the components of the space of continuous maps of the $2$-torus $T^2$ to the $2$-sphere $S^2$, one for each possible degree.  The work of Fox \cite{Fox} on torus homotopy groups (the case $p = 0$) and its generalization by Larmore and Thomas \cite{LarmoreThomas} provide explicit presentations for these groups in which the degrees on subtori and the Pontryagin $\nu$-invariant figure prominently. 
 
The lower horizontal map takes the homotopy class of a based loop in the space of maps of $T^2$ to $S^2$, ignores base points, and interprets this as a homotopy class of maps of $\tor$ to $S^2$ in the usual way.  The resulting map is onto, and point inverse images are conjugacy classes in the various fundamental groups.
 
Finally, the vertical map on the right side is defined to make the diagram commute.  We have a group in the upper right corner, but demote it to a union of groups to match what is below it, so that the vertical map can be regarded as a union of group homomorphisms.  Using the explicit presentations of domain and range, one can determine the effects of this map on generators, and then use the commutativity of the diagram to show that the vertical map $g$ on the left is precisely as described in \thmref{A}.

%% file: 8B.tex
%%%%%%%%%%%%%%%%%%%%%%%%%
%% SECTION 9: Proof of \thmref{B}
%%%%%%%%%%%%%%%%%%%%%%%%%

%%%%%%%%%%%%%%%%%%%%%%%%%%%
\section{Proof of \thmref{B}, formulas (1) and (2)}\label{sec:B}
%%%%%%%%%%%%%%%%%%%%%%%%%%%

\thmref{B} provides three formulas for Milnor's $\mu$-invariant of a three-component link $L$ in $S^3$ whose pairwise linking numbers vanish:

$(1)$\ (differential forms) \  
$\mu(L) \ = \ \displaystyle{\frac12\, \int_{T^3} \delta(\varphi*\omega_L) \wedge \omega_L}$

$(2)$\ (vector fields) \ 
$\mu(L) \ = \ \displaystyle{\frac12\, \int_{\tor\times\tor} \vec_L(\x) \times \vec_L(\y) \dtw 
\nabla_{\!\y}\,\varphi\left(\x - \y\right)\ d\x\, d\y}$

$(3)$\ (Fourier series) \ 
$\mu(L) \ = \  \displaystyle{8\pi^3 \sum_{\n\ne\zero} \a_\n\times \b_\n\dtw \n/\|\n\|^2}$

\begin{comment}
In the first integral, $\delta$  is the exterior co-derivative and  $\varphi*\omega_L$ is the convolution of  $\varphi$  and  $\omega_L$ $($taken with respect to the abelian group structure on the torus$)$.  In the second, $\nabla_{\!\y}$ indicates the gradient with respect to $\x$, the difference $\x-\x'$ is taken in the abelian group $\tor$,  and $d\x$ and $d\x'$ are volume elements.  
\end{comment}

\noindent
The notation is explained in the introduction.

In the subsections that follow, we obtain explicit expressions for the characteristic $2\textup{-form}$ $\omega_L$ and vector field $\vec_L$, and for the fundamental solution $\ph$ to the scalar Laplacian on $\tor$, and then establish formula (1) by making use of J.\,H.\,C.~Whitehead's integral formula for the Hopf invariant. Then, we show how to obtain formula (2) from formula (1). Finally, in 
\secref{B'} we review the calculus of differential forms and Fourier analysis on $T^3$, in order to make explicit the role and form of $\ph$ and obtain formula (3). 

\part{Explicit formula for the 2-form $\omega_L$ and the vector field $\vec_L$ on $T^3$}

Recall the formula for the characteristic map $g_L\colon T^3\to S^2$ of the link $L$ given in \propref{char},
$$
g_L(s,t,u) \ = \ \frac{F(x,y,z)}{\| F(x,y,z)\|},
$$
where $x=x(s)$, $y=y(t)$ and $z=z(u)$ parametrize the components of $L$,
and where $F$ was given in \secref{characteristic} as
$$
F(x,y,z) \ = \  (ix\dt y + iy\dt z + iz\dt x \ , \  jx\dt y + jy\dt z + jz\dt x \ , \  kx\dt y + ky\dt z + kz\dt x).
$$
Here we view $x$, $y$ and $z$ as quaternions for the multiplication and as vectors in $\reals^4$ when performing the dot product.   For simplicity, we write $g_L=F/\| F\|$, suppressing the appearance of the evaluation map $e_L(s,t,u)=(x(s),y(t),z(u))$.

Again let $\omega$ be the Euclidean area 2-form on the unit 2-sphere 
$S^2\subset\reals^3$, normalized so that the total area is 1 instead of $4\pi$.
If $\p$ is a point of $S^2$, and $\a$ and $\b$ are tangent vectors to $S^2$
at $\p$, then 
$$
\omega_{\p}(\a,\b) \ = \ \frac1{4\pi}(\a\times\b)\dtw\p.
$$
This 2-form $\omega$ on $S^2$ extends to a closed 2-form $\bah{\omega}$ in 
$\reals^3-\zero$ given by
$$
\bah{\omega}_{\p}(\a,\b) \ = \ {(\a\times\b)\dt\p\over 4\pi\|\p\|^3},
$$
which is the pullback of $\omega$ from $S^2$ to $\reals^3-\zero$ via the 
map $\p\mapsto \p/\|\p\|$.

Hence the pullback $g_L^\dast\omega$ of $\omega$ from $S^2$ to $T^3$ via
$g_L=F/\| F\|$ is the same as the pullback $F^\dast\bah{\omega}$
of $\bah{\omega}$ from $\reals^3-\zero$ to $T^3$ via $F$.

Write
$$g_L^\dast\omega \ = \ F^\dast\bah{\omega} 
\ = \ a(s,t,u)dt\wedge du+b(s,t,u)du\wedge ds+c(s,t,u)ds\wedge dt.$$
Then we have
$$
\aligned
a(s,t,u)& \ = \  F^\dast\bah\omega({\partial_t},{\partial_u}) \ = \  \bah\omega(F_\dast{\partial_t}, F_\dast{\partial_u}) \\
& \ = \ \bah{\omega}(F_t,F_u) \ = \  \frac{(F_t\times F_u)\dt F}{4\pi\| F\|^3},
\endaligned
$$
and likewise for $b(s,t,u)$ and $c(s,t,u)$, where the subscripts on $F$ denote
partial derivatives.

Therefore, the characteristic 2-form of the link $L$ is
$$\omega_L \ = \  g_L^\dast\omega \ = \ 
\frac{F_t\times F_u\dt F}{4\pi\| F\|^3}  \, dt\wedge du
+\frac{F_u\times F_s\dt F}{4\pi\| F\|^3} \, du\wedge ds
+\frac{F_s\times F_t\dt F}{4\pi\| F\|^3} \, ds\wedge dt,$$
and its corresponding characteristic vector field is 
$$\vec_L \ = \ 
\frac{F_t\times F_u\dt F}{4\pi\| F\|^3} \, \partial_s
+\frac{F_u\times F_s\dt F}{4\pi\| F\|^3} \, \partial_t
+\frac{F_s\times F_t\dt F}{4\pi\| F\|^3} \, \partial_u.$$

\part{Proof of \thmref{B}, formula (1)}

Let $L$ be a three-component link in the 3-sphere $S^3$ with pairwise
linking numbers $p$, $q$ and $r$ all zero.  By the first part of \thmref{A} these numbers are the degrees of the characteristic map $g_L\colon T^3\to S^2$ on the $2$-dimensional coordinate subtori. Since these degrees are all zero, $g_L$ is homotopic to a map $g\colon T^3\to S^2$ which collapses the 2-skeleton of $T^3$ to a point:
$$
g_L\ \simeq \ g\colon T^3 \stackrel{\sigma}{\longrightarrow}S^3 \stackrel{f}{\longrightarrow}S^2,
$$
where $\sigma$ is the collapsing map.  By the second part of \thmref{A}, Milnor's $\mu$-invariant of $L$ is equal to half of Pontryagin's $\nu$-invariant $\nu(g_L)$ (comparing $g_L$ to the constant map $f_{000}$ as explained in \secref{pontryagin}), which in turn is just the Hopf invariant of $f\colon S^3\to S^2$,
$$
\mu(L)\ = \ {1\over 2}\nu(g_L)\ = \ {1\over 2}\Hopf(f).
$$

We can thus use J.\,H.\,C.~Whitehead's integral formula for the Hopf invariant as the first ingredient in our formula for the $\mu$ invariant.  Starting from 
Hopf's definition of his invariant of a map $f\colon S^3\to S^2$ as the linking number between the inverse images of two regular values, Whitehead \cite{Whitehead} found an integral formula for $\Hopf(f)$ as follows.

Let $\omega$ be the area 2-form on $S^2$, normalized so that $\int_{S^2}
\omega=1$. Its pullback $f^\dast\omega$ is a closed 2-form on $S^3$,
which is exact because $H^2(S^3;\reals)=0$. Hence $f^\dast\omega=d\alpha$ for some 1-form $\alpha$ on $S^3$, and Whitehead showed that the Hopf invariant of $f$ is given by the formula
$$
\Hopf(f) \ = \ \int_{S^3}\alpha\wedge f^\dast\omega \ 
= \ \int_{S^3}\alpha\wedge d\alpha,
$$
the value of the integral being independent of the choice of $\alpha$. 

To make Whitehead's formula explicit requires a way to produce a 1-form
$\alpha$ whose differential $d\alpha$ is a given 2-form $f^\dast\omega$ on $S^3$ which is known to be exact. We give an explicit formula for this in the 
Appendix, since it is not needed in the body of the paper. Instead, we are
going to pull the whole situation back to $T^3$ and perform our calculations
there.

In particular, the formula
$$
\Hopf(f) \ = \  \int_{S^3}\alpha_0\wedge f^\dast\omega,
$$
where $\alpha_0$ is any 1-form on $S^3$ such that $d\alpha_0=f^\dast\omega$, pulls back to the formula
$$
\nu(g_L)\ = \ \int_{T^3}\bah{\alpha_0}\wedge g^\dast\omega,
$$
where $\bah{\alpha_0}$ is any 1-form on $T^3$ such that $d\bah{\alpha_0} = g^\dast\omega$.  One such choice of $\bah{\alpha_0}$ would be $\bah{\alpha_0} = \sigma^\dast\alpha_0$, with $\alpha_0$
obtained as in the Appendix.   However, since we don't have an explicit analytic formula
for $\sigma$, and since the formula for $\alpha_0$ is complicated, we will pursue 
a slightly different approach.

Taking advantage of the fact that $g_L$ is homotopic to $g$, we also have
$$
\nu(g_L) \ = \  \int_{T^3}\alpha\wedge g_L^\dast\omega \ = \ \int_{T^3}\alpha\wedge \omega_L
$$
for any $1$-form $\alpha$ on $T^3$ such that $d\alpha=\omega_L$.  Therefore,
what we really need is a canonical way to produce a $1$-form $\alpha$ on $T^3$ whose exterior derivative is the exact 2-form $\omega_L$.   We will prove in \propref{alp} that the $1$-form 
$$
\alpha_L \ = \ \delta(\ph * \omega_L)
$$
has the desired property that $d\alpha_L=\omega_L$.  Moreover, if $\widetilde\alpha$ is any other $1$-form such that $d\widetilde\alpha=\omega_L$, then $\|\alpha_L\| \le \|\widetilde\alpha\|$ in $L^2(\tor)$, with equality if and only if $\widetilde\alpha=\alpha_L$. 

We thus obtain the explicit integral formula for Milnor's $\mu$-invariant of the three component link $L$:
$$
\mu(L) \ = \ {1\over 2}\nu(g_L) \ = \ {1\over 2}\int_{T^3}\delta(\ph * \omega_L)\wedge\omega_L,
$$
which is formula (1) of \thmref{B}.

%\break

%%
\part{Proof of  \thmref{B}, formula (2)}

The 1-form $\delta(\ph * \omega_L)$ on $T^3$ which appears in formula (1) of \thmref{B}
converts to the vector field $\nabla\times(\ph * \vec_L)$, which can be regarded as the magnetic field 
on $T^3$ due to the current flow $\vec_L$. The customary minus sign here is now hidden in the 
definitions of the Laplacian and its inverse (Green's operator).

The integral formula for Milnor's $\mu$-invariant given in formula (1),
$$
\mu(L) \ = \ {1\over 2}\int_{T^3}\delta(\ph * \omega_L)\wedge\omega_L\,,
$$
then converts to the formula
$$
\mu(L) \ = \ {1\over 2}\int_{T^3}\left(\vphantom{x^2}\nabla\times(\ph * \vec_L)\right)\dt\vec_L\,d\vol
$$
in the language of vector fields. To obtain formula (2) of \thmref{B}, we first expand out the convolution integral
$$
(\ph * \vec_L)(\x) \ = \ \int_{T^3}\vec_L(\y)\,\ph(\x-\y)\,d\y.
$$
Then we compute its curl:
$$
\begin{aligned}
\left(\vphantom{x^2}\nabla_{\!\x}\times(\ph * \vec_L)\right)(\x)
\ &= \  \int_{T^3}\nabla_{\!\x}\times\left(\vphantom{x^2}\vec_L(\y)\,\ph(\x-\y)\right)\,d\y  \\
\ &= \  -\int_{T^3}\vec_L(\y)\times\nabla_{\!\x}\,\ph(\x-\y)\,d\y\,,
\end{aligned}
$$
using the product formula $\nabla\times(f\veca)=f(\nabla\times \veca)-\veca\times\nabla f$.
Note that $\nabla_{\!\x}\times\vec_L(\y)=\zero$ since, from the point of view of the variable
$\x$, the vector field $\vec_L(\y)$ is constant.

Inserting the expression for the curl of $\ph*\vec_L$ into the formula for $\mu(L)$, we get
$$
\aligned
\mu(L)
&\ = \ -\frac12\int_{T^3\times T^3}\left(\vphantom{x^2}\vec_L(\y)\times\nabla_{\!\x}\,\ph(\x-\y)\right)\dtw
\vec_L(\x)\,d\x\,d\y \\
&\ = \ -\frac12\int_{T^3\times T^3}\vec_L(\x)\times\vec_L(\y)\dtw\nabla_{\!\x}\,\ph(\x-\y)\,d\x\,d\y\\
&\ = \ \frac12\int_{T^3\times T^3}\vec_L(\x)\times\vec_L(\y)\dtw\nabla_{\!\y}\,\ph(\x-\y)\,d\x\,d\y\,,
\endaligned
$$
where at the last step we hid the minus sign by taking the gradient of $\ph$ with respect to
$\y$ instead of $\x$, and obtained
formula (2) of \thmref{B}.

%%%%%%%%%%%%%%%%%%%%%%%%%
%% SECTION 9: Proof of \thmref{B}, formula (3)
%%%%%%%%%%%%%%%%%%%%%%%%%

%%%%%%%%%%%%%%%%%%%%%%%%%%%
\section{Fourier series and the Proof of \thmref{B}, formula (3)}\label{sec:B'}
%%%%%%%%%%%%%%%%%%%%%%%%%%%

\vskip -.2in
\vskip -.2in
\part{Fourier series and the fundamental solution of the Laplacian}

In the proofs of formulas (1) and (2) of \thmref{B}, we needed to find a 1-form $\alpha$ on $T^3$ whose exterior derivative is the exact 2-form $\omega_L$ associated with the three-component link $L$ in $S^3$. We asserted that we can choose  $\alpha=\delta(\ph * \omega_L)$, where $\ph$ is the fundamental solution of the scalar Laplacian on $T^3$. Furthermore, we asserted that this choice of $\alpha$ is canonical in the sense that it has the smallest $L^2$ norm among all possible choices. We justify these assertions in this section and lay the groundwork for the proof of formula (3) of \thmref{B} by studying the calculus of differential forms on $T^3=(\br/2\pi\bz)^3  $ in terms of their Fourier series. 

We will prove the following two results:

%%%
\begin{comment}
\noindent\bf Theorem C\rm.\ \ \it
\rm (a) \it The fundamental solution 
of the scalar Laplacian on the $3$-torus \\
$T^3=S^1 \times S^1 \times S^1$ is given by the formula
$$
\varphi(x,y,z) \ = \  
{1\over 8\pi^3}\sum_{{ m,n,p = -\infty }\atop
{m^2+n^2+p^2\neq 0}}^\infty 
{e^{i(mx+ny+pz)}\over m^2+n^2+p^2}.
$$
The function $\varphi$ is $C^\infty$ at all points of $T^3$ except $(0,0,0)$, where it becomes infinite. 
\rm (b) \it If $\omega$ is any exact differential form on $T^3$ with $C^\infty$ coefficients, then 
$$
\alpha \ = \ \delta(\ph * \omega)
$$
satisfies $d\alpha=\omega$. Furthermore, if $d\widetilde\alpha=\omega$ as well, then $\Vert\alpha\Vert_{L^2}
\le\Vert\widetilde\alpha\Vert_{L^2}$, with equality if and only if $\widetilde\alpha=\alpha$.
\rm 
\end{comment}
%%%

\begin{proposition}\label{prop:fs}
\textbf{\boldmath  \ \ The fundamental solution of the scalar Laplacian on the\\
 $3$-torus $T^3 = (\br/2\pi\bz)^3$ is given by the formula
$$
\varphi(\x) \ = \
{1\over 8\pi^3}\sum_{\n\ne\zero}
e^{i\n\dte\x} / \|\n\|^2.
$$
The function $\varphi$ is $C^\infty$ at all points $
\x\in T^3$ except $\zero$, where it becomes infinite.} 
\end{proposition}

\bigskip

\begin{proposition}\label{prop:alp}
\textbf{\boldmath  If $\omega$ is any exact differential form on $T^3$ with $C^\infty$ coefficients, then 
$$
\alpha \ = \ \delta(\ph * \omega)
$$
is a $C^\infty$ differential form satisfying $d\alpha=\omega$.  Furthermore, if $d\widetilde\alpha=\omega$ as well, then $\Vert\alpha\Vert_{L^2}\le\Vert\widetilde\alpha\Vert_{L^2}$, with equality if and only if $\widetilde\alpha=\alpha$.} 
\end{proposition}

Before diving into calculations, we pause for some words of explanation: We write $\Omega^k(T^3)$ for the space of $C^\infty$ $k$-forms on $T^3$. With $d$ as the exterior differentiation operator taking $\Omega^k(T^3)$ to $\Omega^{k+1}(T^3)$, and $\delta$ the co-differentiation map adjoint to  $d$ in the $L^2$ sense, the Laplacian of a $k$-form $\alpha$ is 
$$
\Delta\alpha \ = \ (d\delta+\delta d)\alpha.
$$
This definition gives us the ``geometer's sign convention'' for the Laplacian on functions (0-forms):
$$
\Delta f \ = \  -\left({\partial^2 f\over\partial x_1^2}+{\partial^2 f\over\partial x_2^2} +
{\partial^2 f\over\partial x_3^2}\right).
$$

\part{Proof of \propref{fs}}

The \it fundamental solution \rm of the scalar Laplacian is a function $\varphi$, 
convolution with which ``inverts'' the Laplacian to the extent that this is possible. On $T^3$, only 
functions that integrate to zero are in the range of the Laplacian, and so ``the'' fundamental 
solution of $\Delta$ on $T^3$ is the function $\varphi$ which satisfies
$$
\int_{T^3}\varphi\,d\!\vol \ = \  0
\qquad\textup{and}\qquad 
\Delta(\varphi * f) \ = \ f
$$
for all $f\in C^\infty(T^3)$ such that $\displaystyle{\int_{T^3}f\,d\!\vol=0}$.

Even though we have expressed $\varphi$ in terms of complex exponentials, 
the value of $\varphi$ is real for real values of $\x$ because of the symmetry of the coefficients.

\figref{fundsol} shows the graph of the corresponding fundamental solution
$$\varphi(\x) \ = \
{1\over 4\pi^2}\sum_{\n\ne\zero} e^{i\n\dte \x} /  \|\n\|^2 
$$
of the scalar Laplacian on the 2-torus $S^1 \times S^1$, summed for $|\n| \leq 10$, and displayed over the range $[-3\pi,3\pi]\times[-3\pi,3\pi]$.

%%% FIGURE 31: Fund Sol %%%
\begin{figure}[h!]
\includegraphics[height=260pt]{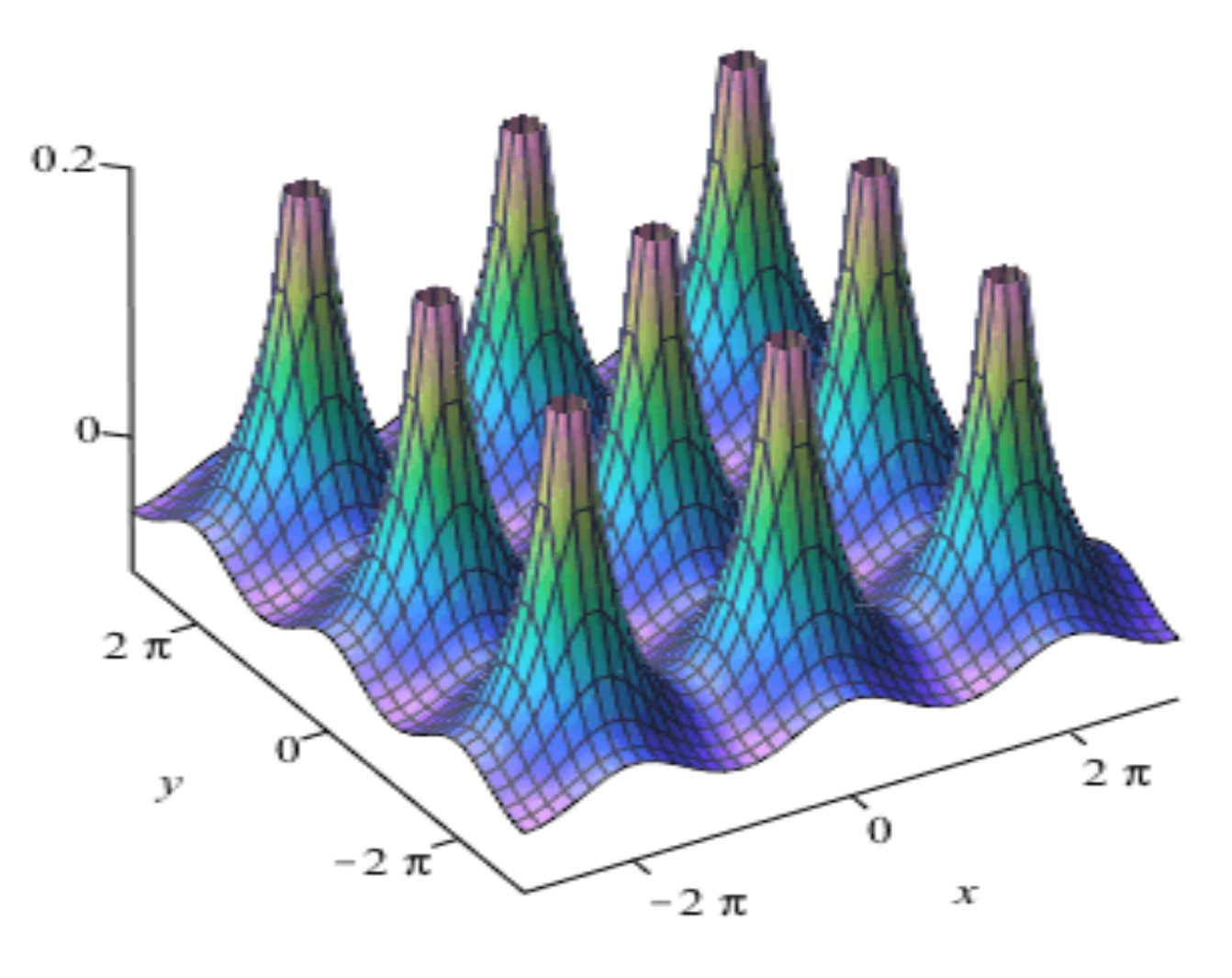}
\caption{Fundamental solution of the scalar Laplacian on $S^1 \times S^1$}
\label{fig:fundsol}
\end{figure}
%%%%%%%%%%%%%%%%%%  

With respect to the $L^2$ inner product
$$
\langle f,g\rangle \ = \  \int_{T^3}f(\x)\,\overline{g(\x)}\,d\x\,,
$$
the set of complex exponentials $\{e^{i\n\dte\x} \st \n\in\bz^3\}$ is an orthogonal set, and for all $\n$, 
we have $\langle e^{i\n\dte\x},e^{i\n\dte\x}\rangle = 8\pi^3$. 

Any function $f$ in $L^2(T^3)$ can be expanded into a Fourier series
$$
f\ \sim \ \sum_{\n\in\ints^3} c_{\n}e^{i\n\dte \x}$$
where 
$$
c_\n\ = \ {1\over 8\pi^3}\langle f,e^{i\n\dte\x}\rangle \ = \  {1\over 8\pi^3}\int_{T^3} f(\x)\,e^{-i\n\dte\x}\,d\x.
$$
Moreover, by Plancharel's theorem, the $L^2(T^3)$-norm of $f$ is equal to the $\ell^2$-norm of the 
sequence of coefficients:
$$
\Vert f\Vert_{L^2} \ = \  \int_{T^3}|f(x)|^2\,d\x \ = \  \sum_{\n\in\ints^3} 8\pi^3|c_\n|^2.
$$
This shows that the function $\varphi$ of \propref{fs} is in $L^2(T^3)$, since $\|c_\n\|^2$ is of order $1/\|\n\|^4$, while the number of lattice points at distance $\|\n\|$
from $\zero$ is of order $\|\n\|^2$. 

Because the range of the scalar Laplacian is the (closed) subspace of functions $f$ for which 
$$
\int_{T^3} f(\x)\,d\x\ = \ 0\,,
$$
we will denote this subspace of $L^2(T^3)$ by $L^2_0(T^3)$, and similarly for other function spaces, where a zero subscript will indicate that all functions in the space have average value zero.  In terms of Fourier coefficients, $f\in L^2_0(T^3)$ if and only if the Fourier coefficient $c_\zero=c_{(0,0,0)}=0$.  

For functions $f$ and $g$ in $C^\infty(T^3)$, their {\itb convolution}
$f*g$ is defined by
$$
\left(\vphantom{x^2}f*g\right)(\x)\ = \ \int_{T^3}f(\y)\,g(\x-\y)\,d\y\,,
$$
and also lies in $C^\infty(T^3)$.  One checks easily that $
e^{i\m\dte\x}*e^{i\n\dte\x} = 0$ if $\m\neq \n$, and 
$$
e^{i\n\dte\x}*e^{i\n\dte\x}\ = \ 8\pi^3 e^{i\n\dte\x}.
$$
Therefore, if $f$ and $g$ have Fourier series $\sum c_{\n}e^{i\n\dte\x}$ and
$\sum d_{\n}e^{i\n\dte\x}$ respectively, then the Fourier series of $f*g$ is
$$
8\pi^3\sum_{\n\in\ints^3}c_{\n}d_{\n}e^{i\n\dte\x}.
$$

From this
it follows that the operation of convolution satisfies $f*g=g*f$
and $(f*g)*h=f*(g*h)$ whenever all the convolutions are defined. Furthermore,
$$
{\partial\over \partial x_k}(f*g)\ = \ {\partial f\over \partial x_k}*g \ = \  f*{\partial g\over \partial x_k}.
$$

For $u\in C^\infty(T^3)$, with 
$$
u(\x)\ = \ \sum_{\n\in\ints^3} b_{\n}e^{i\n\dte\x}\,,
$$
we have
$$
\Delta u(\x)\ = \ \sum_{\n\neq\zero}|\n|^2b_{\n}e^{i\n\dte\x}.
$$
Because the $(0,0,0)$-coefficient of $\Delta u$ is zero, this shows why in order to solve $\Delta u=f$ we must restrict $f$ to be in $C^\infty_0(T^3)$.   And for $f\in C^\infty_0(T^3)$, with
$$
f(\x)\ = \ \sum_{\n\neq\zero}c_{\n}e^{i\n\dte\x},
$$
the unique solution $u\in C^\infty_0(T^3)$ of $\Delta u=f$ has $b_{\n}= c_\n/|\n|^2$,
and so, since 
$$
\ph(\x)\ = \ \frac{1}{8\pi^3}\sum_{\n\ne\zero}{e^{i\n\dte \x}}/{\|\n\|^2}\,,
$$
we conclude that $u=\varphi * f$.  This completes the proof of \propref{fs} except for the assertion about the smoothness of $\ph$ away from its singularity, which follows from well-known local regularity theorems for elliptic partial differential equations (see, for example, Folland~\cite{Folland},  Theorem 6.33).

\part{Fourier series and the calculus of differential forms on the 3-torus}

To prove \propref{alp}, we express the calculus of differential forms on $T^3$ in terms of Fourier series.  Except for the fundamental solution $\ph$ of the (scalar) Laplacian on $T^3$, we will assume all functions and forms are $C^\infty$.

We continue to express functions (0-forms) as Fourier series: for 
$\x=(s,t,u)$ and $\n=(n_1,n_2,n_3)$, we write
$$
f(\x) \ = \ \sum_{\n\in\ints^3}c_{\n}e^{i\n\dte\x}\ \in\ \Omega^0(T^3).
$$
Likewise, we employ the notation described in the introduction and express 1-forms using 
$\c_{\n}=(c_\n^{s},c_\n^{t},c_\n^{u})$
and $d\x=(ds,dt,du)$. We write
$$
\alpha(\x) \ = \ \sum_{\n\in\ints^3}\c_{\n}e^{i\n\dte\x}\dt d\x\ \in\ \Omega^1(T^3).
$$
With $\star d\x=(dt\wedge du,du\wedge ds,ds\wedge dt)$ as before, we can express a 2-form as
$$
\beta(\x) \ = \ \sum_{\n\in\ints^3}\c_{\n}e^{i\n\dte\x}\dt \star d\x\ \in\ \Omega^2(T^3).
$$
Finally, with $dV=ds\wedge dt\wedge du$, a 3-form can be written as
$$
\gamma(\x) \ = \ \sum_{n\in\ints^3}c_{\n}e^{i\n\dte\x}dV\ \in\ \Omega^3(T^3).
$$

It is straightforward to express both the exterior derivative $d$ and the
codifferential $\delta$ in terms of Fourier coefficients. With $f$, $\alpha$,
$\beta$ and $\gamma$ as above, we have
$$
\begin{array}{lll}
\displaystyle{df \ = \ \sum ic_{\n}\n e^{i\n\dte\x}\dt d\x} &
\mbox{\rm\ and\ } & \delta f \ = \ 0,\\[0.3cm]
\displaystyle{d\alpha \ = \ \sum i\n\times \c_\n e^{i\n\dte\x}\dt \star d\x} &
\mbox{\rm\ and\ } & \displaystyle{\delta\alpha \ = \ -\sum i\n\dt \c_\n e^{i\n\dte\x}},\\[0.3cm]
\displaystyle{d\beta \ = \ \sum i\n\dt\c_\n e^{i\n\dte\x}dV} &
\mbox{\rm\ and\ } & \displaystyle{\delta\beta \ = \ \sum i\n\times\c_\n e^{i\n\dte\x}
\dt d\x,}\\[0.3cm]
d\gamma \ = \ 0 & \mbox{\rm\ and\ } & 
\displaystyle{\delta\gamma \ = \ -\sum ic_\n \n e^{i\n\dte\x}\dt\star d\x}.
\end{array}
$$
Thus, exterior differentiation and co-differentiation are expressed in terms of vector algebraic operations on the Fourier coefficients.
From these expressions, we conclude the following about the kernel and image
of $d$ and $\delta$:

\begin{itemize}
\item For 0-forms, 
$$
\begin{aligned}
\ker d &\ = \ \{f \st c_\n=0\ \mbox{\rm for}\ \n\ne\zero,
\mbox{\rm\ while\ }
\ c_\zero
\ \mbox{\rm is arbitrary}\} \ \ \ \\[0.2cm]
\im d\, &\ = \  \{0\} \\[0.2cm]
\ker \delta &\ = \  \{\mbox{\rm all}\ f\} \\[0.2cm]
\im \delta\, &\ = \  \{f \st c_\zero=0\}
\end{aligned}
$$
\item For 1-forms,
$$
\begin{aligned}
\ker d &\ = \  \{\alpha \st \c_\n=\lambda_\n\n\ \mbox{\rm for}\ \n\ne\zero,
\mbox{\rm\ while\ }\c_\zero
\ \mbox{\rm is arbitrary}\} \\[0.2cm] 
\im d\, &\ = \  \{\alpha \st \c_\n=\lambda_\n\n\ \mbox{\rm for}\ \n\ne\zero,\ \c_\zero=\zero\} \\[0.2cm]
\ker \delta &\ = \  \{\alpha \st \c_\n\dt\n =0\ \mbox{\rm for}\ \n\ne\zero,
\mbox{\rm\ while\ }\c_\zero \ \mbox{\rm is arbitrary}\} \\[0cm]
%%%%%%% attempt to even out the spacing
\im \delta\, &\ = \  \{\alpha \st \c_\n\dt\n=0\ \mbox{\rm for}
\ \n\ne\zero,\ \c_\zero=\zero\}
\end{aligned}
$$
\item For 2-forms,
$$
\begin{aligned}
\ker d &\ = \  \{\beta \st \c_\n\dt\n=0\ \mbox{\rm for}\ \n\ne\zero,
\mbox{\rm\ while\ }\c_\zero
\ \mbox{\rm is arbitrary}\} \\[0.01cm]
\im d\, &\ = \  \{\beta \st \c_\n\dt\n=0\ \mbox{\rm for}\ \n\ne\zero,\ \c_\zero=\zero\} \\[0cm]
\ker \delta &\ = \  \{\beta \st \c_\n=\lambda_\n\n\ \mbox{\rm for}\ \n\ne\zero,
\mbox{\rm\ while\ }\c_\zero \ \mbox{\rm is arbitrary}\} \\[0cm]
\im \delta\, &\ = \  \{\beta \st \c_\n=\lambda_\n\n\ \mbox{\rm for}
\ \n\ne\zero,\ \c_\zero=\zero\}
\end{aligned}
$$
\item For 3-forms,
$$
\begin{aligned}
\ker d&\ = \ \{\mbox{\rm all}\ \gamma\}\\[0.15cm]
\im d\, &\ = \  \{\gamma \st c_\zero=0\} \\[0.2cm]
\ker \delta &\ = \  \{\gamma \st c_\n=0\ \mbox{\rm for}\ \n\ne\zero,
\mbox{\rm\ while\ }\ c_\zero
\ \mbox{\rm is arbitrary}\} \\[0.1cm]
\im \delta\, &\ = \ \{0\}
\end{aligned}
$$
\end{itemize}

Therefore $\im d\subset\ker d$ and $\im \delta\subset\ker\delta$, and
the following orthogonal (with respect to the $L^2$ inner product) decompositions hold:
$$
\Omega^k(T^3)\ = \ \ker\delta\oplus\im d \ = \ \im\delta\oplus\ker d
$$
for $k=0,\ldots,3$. The $k$-forms in $\ker d\cap\ker\delta$ are the forms
whose Fourier series contain only constant terms. These are called \it harmonic \rm
$k$-forms because they are in the kernel of
the Laplacian $\Delta=d\delta+\delta d$, so we write $\Har^k(T^3)=\ker d\cap\ker\delta$.
We then have the Hodge decomposition 
$$
\Omega^k(T^3)\ = \ \im\delta\oplus\Har^k(T^3)\oplus \im d.
$$
The Laplacian $\Delta\colon\Omega^k(T^3)\to\Omega^k(T^3)$ preserves this Hodge decomposition, 
taking $\im \delta$ bijectively to itself, killing $\Har^k(T^3)$, and taking $\im d$ bijectively to itself.

%%%%%
\part{Proof of \propref{alp}}
%%%%%

We must show that if $\omega$ is any $C^\infty$ exact differential $k$-form on $T^3$, then the $(k-1)$-form $\alpha=\delta(\ph*\omega)$ is $C^\infty$ and satisfies
$d\alpha=\omega$, and that if $d\widetilde\alpha=\omega$ as well, then $\Vert\alpha\Vert_{L^2}
\le\Vert\widetilde\alpha\Vert_{L^2}$, with equality if and only if $\widetilde\alpha=\alpha$.

Because it is the case we use in \thmref{B}, we will carry this out 
specifically for 2-forms. The proofs for 1-forms and 3-forms
are essentially the same. 

For a 2-form 
$$
\beta\ = \ \sum_{\n\in\ints^3}\c_\n e^{i\n\dte\x}\dt \star d\x\,,
$$ 
its Laplacian is given by
$$
\Delta\beta\ = \ \sum_{\n\ne\zero}|\n|^2\c_\n e^{i\n\dte\x}\dt \star d\x.
$$
We can use convolution with the fundamental solution $\ph$ of the scalar Laplacian to express the Green's operator,
$$
\Gr(\beta)\ = \ \ph * \beta \ = \  \sum_{\n\ne \zero}{\c_\n\over|\n|^2}e^{i\n\dte\x}\dt \star d\x.
$$
Clearly the Laplacian $\Delta$ and the Green's operator are inverses of one another when applied to 2-forms $\beta$ with $\c_\zero=\zero$, equivalently, to 2-forms $\beta$ orthogonal to $\Har^2(T^3)$.

To prove \propref{alp}, assume that $\beta$ is exact, that is, 
$\beta\in\im d$. Then $\Gr(\beta)$ is also exact, hence certainly closed, and therefore
$$
\beta\ = \ \Delta\Gr(\beta)\ = \ (d\delta+\delta d)\Gr(\beta)\ = \ d\delta\Gr(\beta)\ = \ d\delta(\ph*\beta).
$$
Thus $\alpha=\delta(\ph*\beta)$ satisfies $d\alpha=\beta$.

If $\widetilde\alpha$ is any other 2-form satisfying $d\widetilde\alpha=\beta$, then $\widetilde\alpha$ differs from $\alpha$ by some closed 2-form, which must be $L^2$-orthogonal to $\alpha$ since $\alpha\in\im\delta$. Thus, $\Vert\alpha\Vert_{L^2}\le\Vert\widetilde\alpha\Vert_{L^2}$
by the Pythagorean theorem.

Since the Fourier coefficients of a $C^\infty$ 2-form decrease
faster than any negative power of $|\n|$, we have that $\Gr(\beta)$ and 
$\delta\Gr(\beta)$ will be $C^\infty$ if $\beta$ is. This completes the proof
of \propref{alp}.

\part{Proof of \thmref{B}, formula (3)}

This formula expresses Milnor's invariant $\mu(L)$ in terms of the Fourier coefficients of $\omega_L$:
$$
\mu(L)\ = \ 8\pi^3\sum_{\n\ne\zero}\frac{\a_\n\times\b_\n\dtew\n}{\|\n\|^2}
$$
where 
$$
\omega_L\ = \ \sum_{\n\ne\zero}\c_\n e^{i\n\dte\x}\dtew \star d\x
$$
with $\c_\n=\a_\n+i\b_\n$ and with $\a_\n$ and $\b_\n$ real.

The formula for $\omega_L$ is summed over $\n\ne\zero$ because the hypothesis
of pairwise linking numbers zero is equivalent to the vanishing of the harmonic component of $\omega_L$, which in turn is equivalent to the vanishing of its Fourier coefficient $\c_\zero$.

We begin with formula (1) of \thmref{B} for $\mu(L)$:
$$
\mu(L)\ = \ {1\over 2}\int_{T^3}\delta(\ph*\omega_L)\wedge\omega_L.
$$
We have
$$
\ph*\omega_L\ = \ \sum_{\n\ne\zero}\frac{\c_\n e^{i\n\dte\x}}{\|\n\|^2}\dtew \star d\x\,,
$$
and hence, from our table of derivatives and co-derivatives in terms of Fourier
series,
$$
\delta(\ph*\omega_L)\ = \ \sum_{\n\ne\zero} \frac{(\n\times\c_\n)\,i\,e^{i\n\dte\x}}{\|\n\|^2}\dtew d\x.
$$
To compute $\delta(\ph*\omega_L)\wedge\omega_L$, we use the fact that if $\vv \dtew d\x$ is a 1-form and $\w\dtew \! \star \!d\x$ is a 2-form, then
$$
(\vv\dtew d\x)\wedge(\w\dtew \! \star\!d\x)\ = \ (\vv\dtew \w)\,dV.
$$
Thus
$$
\delta(\ph*\omega_L)\wedge\omega_L\ = \ \sum_{\m\ne\zero}\sum_{\n\ne\zero}
\frac{(\n\times\c_\n)\dtew\c_\m\,ie^{i(\m+\n)\dte\x}}{\|\n\|^2}\,dV.
$$
When we insert this double sum into the integral 
$\mu(L)={1\over 2}\int_{T^3}\delta(\ph*\omega_L)\wedge\omega_L,$
most of the terms in the summation will integrate to zero, leaving only the terms
where $\m+\n=\zero$. In those cases, $e^{i(\m+\n)\dte\x}=e^0=1$ integrates
to $8\pi^3$, the volume of $T^3$.

Thus
$$
\begin{aligned}
\mu(L)&\ = \ 4\pi^3\sum_{\n\ne\zero}\frac{i(\n\times\c_\n)\dtew\c_{-\n}}{\|\n\|^2} \\
&\ = \ 4\pi^3\sum_{\n\ne\zero}\frac{i(\c_\n\times\c_{-\n})\dtew\n}{\|\n\|^2}
\end{aligned}
$$
and it remains to simplify this last series.

Since $\c_\n=\a_\n+i\b_\n$ and since $\omega_L$ is real-valued, we have
$
\c_{-\n} = \bah{\c_\n} = \a_\n-i\b_\n.
$
Hence
$$
\c_\n\times\c_{-\n}\ = \ (\a_\n+i\b_\n)\times(\a_\n-i\b_\n)\ = \ 
-2i\,\a_\n\times\b_\n,$$
and therefore
$$
\begin{aligned}
\mu(L)&\ = \ 4\pi^3\sum_{\n\ne\zero}\frac{i(\c_\n\times\c_{-\n})\dtew\n}{\|\n\|^2}\\
&\ = \ 8\pi^3\sum_{\n\ne\zero}\frac{(\a_\n\times\b_\n)\dtew\n}{\|\n\|^2}\,,
\end{aligned}
$$
completing the proof of formula (3) and, with it, that of \thmref{B}.

\part{Numerical computation}

We used Maple and Matlab 
to calculate an approximation to Milnor's $\mu$-invariant for the three-component link $L$
in $S^3$ given by
$$
\aligned
x(s)&\ = \ [\textstyle{4\over 5}\sin s,\cos s,\textstyle{3\over 5}\sin s,0], \\
y(t)&\ = \ [\textstyle{4\over 5}\sin t,0,\cos t,\textstyle{3\over 5}\sin t], \\
z(u)&\ = \ [\textstyle{4\over 5}\sin u,\textstyle{3\over 5}\sin u,0,\cos u],
\endaligned
$$
for $s\in[0,2\pi]$,
$t\in[0,2\pi]$ and $u\in[0,2\pi]$,
which is a concrete realization of the Borromean rings with $\mu=-1$. 

%%%%%%%%%%%%%
%%% FIGURE 32: Borromean Rings %%%
\begin{figure}[h!]
\includegraphics[height=130pt]{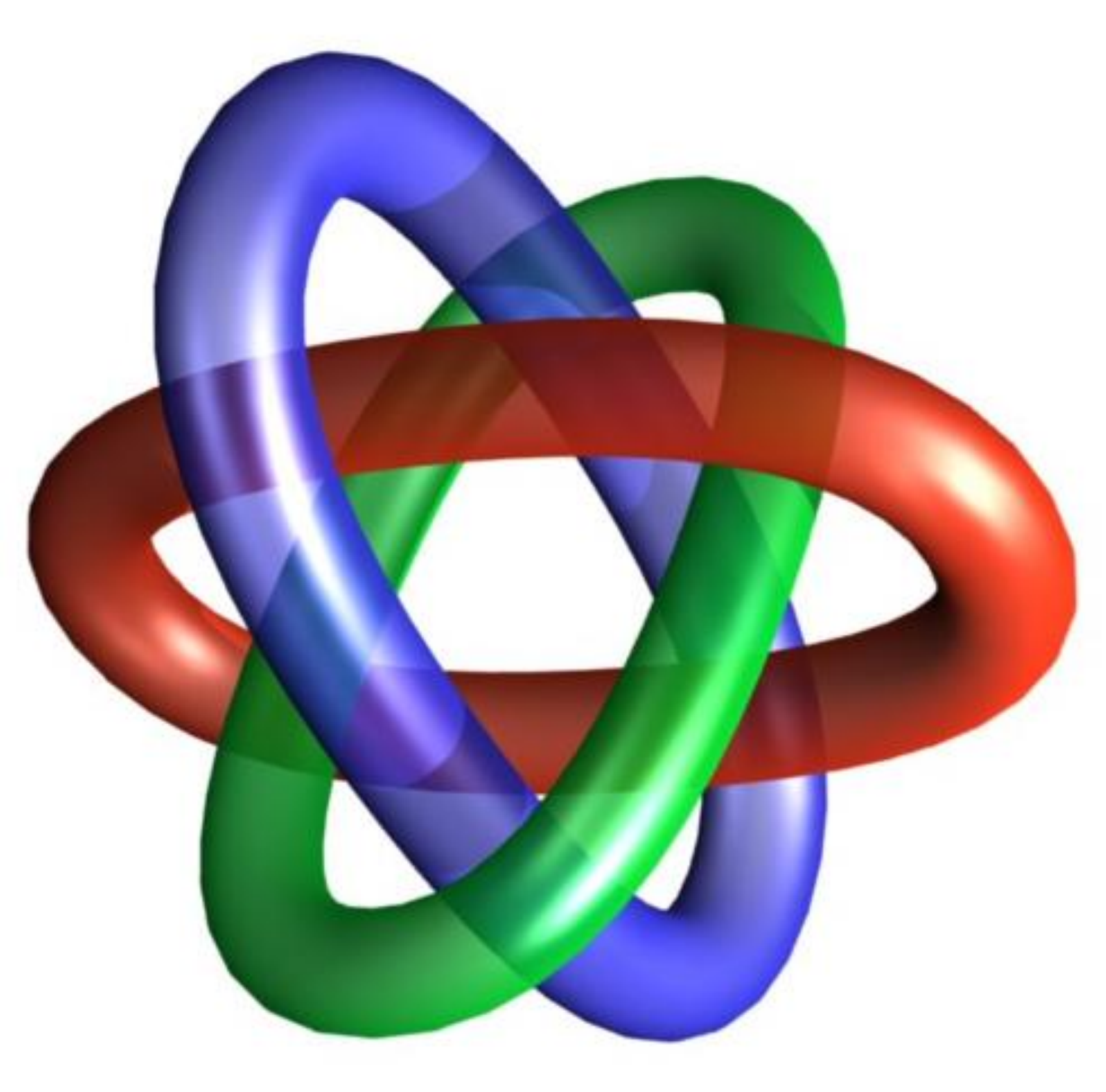}
\caption{Borromean rings}
\label{fig:bor}
\end{figure}
%%%%%%%%%%%%%%%%%%  
%%%%%%%%%%%%%

Maple gave the (long, ugly) formula for $\omega_L$ obtained from 
\propref{char}, and we used Matlab to calculate approximations to its Fourier coefficients $\c_\n$
for divisions of the $s$, $t$ and $u$ intervals into 60 subintervals and for $\n\in[-15,15]^3$. The approximation of $\mu$ we obtained in this way was $-0.97$. 

%%%%
\appendix
%%%%

\section*{Appendix.  Whitehead's integral formula for the Hopf invariant}

We take this opportunity to make J.\,H.\,C.~Whitehead's integral formula for the Hopf invariant of a map $f\colon S^3\to S^2$ explicit, and have relegated it to an appendix because we do not use the result in this paper.

Recall from \secref{B} Whitehead's formula, 
$$
\Hopf(f) \ = \ \int_{S^3}\alpha\wedge\omega_f,
$$
where $\omega_f=f^\dast\omega$ is the pullback to $S^3$ of the normalized area form $\omega$ on $S^2$, and $\alpha$ is any 1-form on $S^3$ such that $d\alpha=\omega_f$. To make this formula explicit requires a way to produce such an $\alpha$. 

To do this, we first write Whitehead's formula in the language of vector fields,
$$
\Hopf(f) \ = \ \int_{S^3}\veca\dtw\vec_f\ d\vol,
$$
where the 2-form $\omega_f$ on $S^3$ has been converted to the vector field $\vec_f$ in the usual way, and where $\veca$ is any vector field on $S^3$ such that $\nabla\times\veca=\vec_f$.

An explicit recipe for $\veca$ was given by DeTurck and Gluck~\cite{DeTurckGluck}:
$$
\veca(\y) \ = \ \BS(\vec_f)(\y) \ = \ \int_{S^3}P_{\y\x}\vec_f(\x)\times\nabla_{\!\y}\,\ph(\x,\y)\,d\x.
$$
Here $\BS$ is the Biot--Savart operator for vector fields on the $3$-sphere.  In the last integral, $P_{\y\x}$ indicates parallel transport in $S^3$ along the geodesic segment from $\x$ to $\y$, and the function $\ph$ is given by
$$
\ph(\alpha) \ = \ -{1\over 4\pi^2}(\pi-\alpha)\csc\alpha,
$$
and $\ph(\x,\y)$ is an abbreviation for $\ph(\alpha(\x,\y))$, with
$\alpha(\x,\y)$ the geodesic distance on $S^3$ between $\x$ and $\y$. 

The significance of the above function $\ph$ is that it is the fundamental solution of a shifted 
Laplacian on $S^3$: 
$$-\Delta\ph-\ph=\delta,$$
where $\delta$ is the Dirac delta function. 

The above formula for $\BS(\vec_f)$ is the analogue on $S^3$ of the classical Biot--Savart formula from 
electrodynamics in $\reals^3$, expressing the magnetic field $\BS(\vec_f)$ in terms of the current 
flow $\vec_f$. The equation $\nabla\times\BS(\vec_f)=\vec_f$ is just one of Maxwell's 
equations, transplanted to $S^3$.

Inserting this formula for $\BS(\vec)$ into the previous formula for the 
Hopf invariant and performing simple manipulations, we get
$$
\Hopf(f) \ = \ -\int_{S^3\times S^3} P_{\y\x}\vec_f(\x)\times\vec_f(\y)\dtw
\nabla_{\!\y}\,\ph(\x,\y)\,d\x\,d\y,
$$
the explicit version of Whitehead's integral formula on the 3-sphere $S^3$.

%% file: 9bibliography.tex
%%%%%%%%%%%%%%%%%
%% REFERENCES
%%%%%%%%%%%%%%%%%

\providecommand{\bysame}{\leavevmode\hbox to3em{\hrulefill}\thinspace}
\providecommand{\MR}{\relax\ifhmode\unskip\space\fi MR }
\providecommand{\MRhref}[2]{%
  \href{http://www.ams.org/mathscinet-getitem?mr=#1}{#2}
}
\providecommand{\href}[2]{#2}

\vskip 1in

%%% FIGURE: Borromean Rings %%%
  \begin{center}
  \includegraphics[height=125pt]{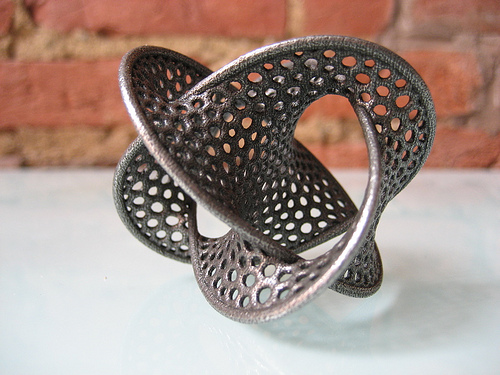}
  \put(-115,-24){\textsc{parting shot}}
  \put(-190,-40){The Borromean rings with their Seifert surface}
  \end{center}
%%%%%%%%%%%%%%%%%%%%